\documentclass[a4paper;11pt]{amsart}
\usepackage{mathrsfs}\usepackage{graphicx}
 \usepackage[arrow,matrix]{xy}
\usepackage{amsmath,amssymb,amscd,bbm,amsthm,mathrsfs}
\usepackage{color,xcolor}
\usepackage{graphicx}
\usepackage{manfnt}

 \usepackage{soul} 

\newtheorem{thm}{Theorem}[section]
\newtheorem{lem}{Lemma}[section]
\newtheorem{cor}{Corollary}[section]
\newtheorem{prop}{Proposition}[section]
\newtheorem{rem}{Remark}[section]

\theoremstyle{definition}

\begin{document}
\numberwithin{equation}{section}

 \title[ Characterizations of  Hardy spaces    on    tube domains   ]{  Characterizations of Hardy spaces on tube domains over polyhedral cones  }
 
\author {Zunwei Fu$^a$, Loukas Grafakos$^b$, Wei Wang$^c$ and Qingyan Wu$^a$}



\thanks{$a$ Department of Mathematics,
         Linyi University,
         Shandong  276005, China, Email: fuzunwei@lyu.edu.cn (Z. Fu), wuqingyan@lyu.edu.cn (Q. Wu);}
     \thanks{$ b$    Department of Mathematics, University of Missouri, Columbia, MO 65211, USA),  Email:  grafakosl@missouri.edu;}
         \thanks{$ c$ Department of Mathematics, Zhejiang University, Zhejiang 310058, China, Email: wwang@zju.edu.cn.}

\begin{abstract}  
This paper is devoted to the equivalence of various characterizations of  holomorphic $H^1$ Hardy spaces on tube domains over polyhedral cones. We establish a new iterated Poisson integral formula
 which reproduces holomorphic functions on such domains. However, this formula shows that holomorphic $H^1$ functions have boundary values in a new type of Hardy { space} of real variables on their 
 Shilov boundaries $\mathbb{R}^n$, which cannot be treated by   { 
 standard classical} 
 multi-parameter harmonic analysis. We overcome this difficulty by developing 
 {   techniques  suitably adapted in this setting}. 
Using the iterated Poisson integral as our approximation to the identity, and employing a lifting technique,
we introduce various notions of multi-parameter analysis adapted to tube domains,
 such as twisted rectangles, new non-tangential approach regions, 
non-tangential maximal functions and Littlewood-Paley type functions.
 All these notions exhibit new geometric features associated with polyhedral cones and involve hidden parameters, as in the flag setting.
We develop the necessary { multi-parameter  tools} to investigate these new Hardy spaces.
{ In particular, we apply these tools
to obtain equivalent characterizations of the
 holomorphic $H^1$ Hardy spaces on tube domains  in terms of non-tangential maximal,  Lusin-Littlewood-Paley area and Littlewood-Paley $g$-functions.} 
\end{abstract}

\thanks{The first  author  is supported
by the National Natural Science Foundation  (NNSF) (No.12431006). {The second  author  is supported by Simons Foundation Grant (No.00012681).}  The third author   is supported by
by  NNSF (No.12371082). The fourth  author   is supported by NNSF
(No.12171221) and Taishan Scholars Program for Young Experts of Shandong Province (tsqn202507265).
}

 \maketitle

  \tableofcontents

\section{Introduction }
It is well known that holomorphic functions in the Hardy space $H^p(\mathbb{R}_+^2)$ on the upper half-plane $\mathbb{R}_+^2$ can be characterized by their non-tangential maximal functions or 
Lusin-Littlewood-Paley area functions, and that the real parts of the boundary values of such functions constitute the Hardy space $H^p(\mathbb{R})$ of real-variable functions on $\mathbb{R}$.
 In order to define the Hardy space $H^p(\mathbb{R}^n)$ of real variables on $\mathbb{R}^n$, Stein and Weiss \cite{SW} introduced the notion of a generalized analytic function 
 $u=(u_0,\dots,u_n):\mathbb{R}_+^{n+1}\rightarrow \mathbb{R}^{n+1}$, which satisfies
\begin{equation*}
\sum_{j=0}^n \frac {\partial u_j}{\partial x_j}=0,\qquad
   \frac {\partial u_j}{\partial x_k}=\frac {\partial u_k}{\partial x_j}\quad \text{for } j\neq k,
\end{equation*}
and proved that the Hardy space $H^p(\mathbb{R}_+^{n+1})$ of generalized analytic functions for $p \ge \frac {n-1}{n}$ can also be characterized by non-tangential maximal functions or area functions. The 
components $u_0$ of the boundary values of such functions constitute the Hardy space $H^p(\mathbb{R}^n)$ of
real variables on $\mathbb{R}^n$.

There is a long history of extending the { classical} characterizations of analytic Hardy spaces to more general pseudoconvex domains in several complex variables and of establishing a correspondence 
between the boundary values of analytic Hardy functions and suitable Hardy spaces of real variables on their Shilov boundaries. For the Siegel upper half-space, Geller \cite{Gel} proved that these 
characterizations are equivalent and that the boundary values of holomorphic $H^p$ functions belong to the Hardy space $H^p$ on the Heisenberg group. Such characterizations can be extended to a large class of 
smooth pseudoconvex domains  with smooth boundary (cf. e.g. \cite{KL,Stout}).

However, for pseudoconvex domains with non-smooth boundary, new phenomena
 {appear.}
 Even for the bidisc, there are fundamental difficulties
 {caused by}
 the
 {higher-dimensional}
 geometric complexity.
 {In the 1970s, Malliavin and Malliavin \cite{MM} proved that}
 the maximal function characterization implies the characterization in terms of the Lusin-Littlewood-Paley area function, by introducing a technique
 {based on}
 an auxiliary function. This result was generalized to the $H^p$ theory for the bidisc by Gundy and Stein \cite{GuS}. Biharmonic or holomorphic $H^p$ functions on the bidisc have boundary values in the 
 bi-parameter Hardy space on its Shilov boundary,
 {namely}
 a torus. These results stimulated the development of multi-parameter harmonic analysis (cf. e.g. \cite{CF85,Feff-S,J,LPPW,P}),  {which remains an active field of research.}
 Later, Sato \cite{Sato} extended the theory to the equivalence of characterizations of Hardy
 {spaces} for generalized analytic functions on the product $\mathbb{R}_+^{n_1+1}\times \mathbb{R}_+^{n_2+1}$.

Since general Siegel domains or symmetric spaces have non-smooth boundaries, admissible convergence of Poisson integrals
 {becomes}
 very complicated (cf. e.g. \cite{Kr,Sj,St83}),
 {and the study of}
 holomorphic $H^p$ Hardy spaces is usually restricted to homogeneous Siegel domains and to $p>1$ (cf. e.g. \cite{BS,CP21} and
 {the} references therein). However, in \cite{WW23}, the second and third named authors proved that any holomorphic $H^1$ function on the product of two Siegel upper half-spaces has a boundary value in  {a} 
 bi-parameter Hardy space $H^1$ on its Shilov boundary, which is the product of two Heisenberg groups. By using bi-parameter harmonic analysis (cf. e.g. \cite{CDLWY}), the Cauchy-Szeg\H o projection is used to 
 decompose a holomorphic $H^1$ function into a sum of holomorphic atoms.
In \cite{WW24}, { the authors 
 study} a new class of flag-like singular integral kernels, { that includes the} Cauchy-Szeg\H o kernels on certain Siegel { domains. These kernels are neither of product type nor of flag 
 type on their Shilov boundaries, which have the structure of nilpotent Lie groups of step two, say $\mathscr N$. 
 In \cite{WW24} the   authors also introduce   new Hardy spaces defined by iterated heat kernel integrals on these  groups $\mathscr N$.

In general, holomorphic $H^p$ functions on different domains may have boundary values in different
 {types} of Hardy spaces of real variables on their Shilov boundaries, and one needs to use different
 {kinds} of multi-parameter harmonic analysis to deal with them. It was
 { attested} by Stein \cite{St98} that ``a suitable version of multi-parameter analysis will provide the missing theory of singular integrals needed for a variety of questions in several complex 
 variables. This is indeed an exciting prospect.'' Under this philosophy, Nagel, Ricci and Stein \cite{NRS} introduced and studied an important class of Siegel domains: tube domains over polyhedral cones.
A {\it polyhedral cone} $\Omega$ in $\mathbb{R}^n$ is the interior of the convex hull of a finite number of rays meeting  {at} the origin, among which there exist at least $n$ that are linearly independent. The 
{\it tube domain} over $\Omega$ is the domain in $\mathbb{C}^n$ defined by
\begin{equation*}
  T_\Omega:=\mathbb{R}^n+\mathbf{i}\Omega=\{x+\mathbf{i}y\in \mathbb{C}^n ; x\in \mathbb{R}^n,\  y\in \Omega \},
\end{equation*}
whose Shilov boundary is the Euclidean space $\mathbb{R}^{n}+\mathbf{i}0$.
The holomorphic Hardy space $H^p( T_\Omega )$
consists of all holomorphic functions $F$ on $ T_\Omega $ such that
\begin{equation}\label{eq:H2}
   \|F\|_{H^p( T_\Omega)}:=\left(\sup_{y \in \Omega}\int_{\mathbb{R}^{n}}\left|F(  x +\mathbf{i}y
    )\right|^p\,dx\right)^{1/p}
    <\infty.
\end{equation}
See \cite[Chapter 3]{SW71} for basic facts about $H^p$ Hardy spaces on tube domains.
The Cauchy-Szeg\H o projection $L^2(\mathbb{R}^{n})\rightarrow H^p( T_\Omega)$ has a kernel, called the {\it Cauchy-Szeg\H o kernel}, given by
\begin{equation*}
   C(x+ \mathbf{i }y)=\int_{\Omega^*}
e^{2\pi \mathbf{i}(x+ \mathbf{i } y)\cdot \xi}   \,d\xi,
\end{equation*}
where $\Omega^*:=\{\xi\in \mathbb{R}^n : y\cdot\xi\geq 0,\ y\in \Omega\}$ is the dual cone of $\Omega$.
Nagel, Ricci and Stein \cite{NRS} proved that
 {this kernel} is a sum of flag kernels when restricted to the Shilov boundary. However,  the study of the Hardy spaces associated to sums of flag kernels  is still in its early 
 stage (cf. e.g. \cite{HC} and the references therein).
This paper is devoted to
 {developing}
 the necessary tools of a ``suitable version of multi-parameter analysis'' to investigate new types of Hardy spaces of real variables on the Shilov boundary $\mathbb{R}^{n}$,   to applying them to the holomorphic 
 $H^1$ Hardy spaces on $T_\Omega$,   and  
 to proving the equivalence of various  of its characterizations.

It is well known \cite{FS} that Hardy spaces of real variables are defined in terms of approximations to the identity.
 Thus our first step is to find an integral representation formula for holomorphic functions on $T_\Omega$, which can also be viewed as an approximation to the identity on the Shilov boundary $\mathbb{R}^{n}$. 
  We use this representation  to define new Hardy spaces corresponding to holomorphic $H^1$ Hardy spaces on $T_\Omega$, and { study the latter.}
For simplicity, we restrict
 {ourselves}
 to $H^p( T_\Omega)$ with $p=1$ in this paper, and
 {assume}
 that the polyhedral cone is spanned by unit vectors $e_j$ in $\mathbb{R}^n$, $j=1,\dots,m$, i.e.
\begin{equation*}
    \Omega=\{\lambda_1e_1+\cdots+\lambda_me_m ; \lambda_1,\dots,\lambda_m> 0 \},
\end{equation*}
and
 {that}
 any $n$ of $e_1,\dots,e_m$ are linearly independent,
 {in order to} avoid degeneracy  (see  Figure 1). We must have $m\geq n$. 
  By definition, the polyhedral cone $\Omega$ has the form $\{\lambda y ; \lambda>0,\ y\in \mathcal{C}\}$, where $\mathcal{C}$ is the convex hull of the
 {vectors} $\{e_1,\dots,e_m\}$.
 \begin{figure}[h]
  \centering
\includegraphics[scale=0.5]{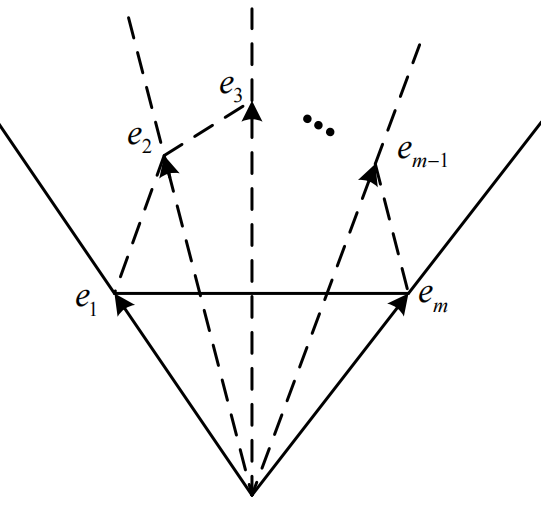}
  \caption{Polyhedral cone}
  \label{fig:fig1} 
\end{figure}

Let $P_t(s)={ \frac1\pi \frac{t}{t^2+s^2}}$ be the standard Poisson kernel on $\mathbb{R}$. Define
 {the} {\it Poisson integral along the line} $e_\mu$ by
\begin{equation}\label{eq:convolution}
   f*_\mu P_t(x):=\int_{\mathbb{R}} f(x-se_\mu)P_t(s)\,ds
\end{equation}
for $x\in \mathbb{R}^n$, and
 {the} {\it iterated Poisson integral} by
\begin{equation}\label{eq:iterated-Poisson-integral}
    { P_{\mathbf{t}} (f)(x)}:= f*_1 P_{t_1}\cdots *_m P_{t_m}(x),
\end{equation}
which is a function of
 {the} variables $(x,\mathbf{t})$ in the space $\mathbb{R}^n\times(\mathbb{R}_+)^m$.  
  We often   write   $P_{\mathbf{t}}(f)(x)$ as $ f*P_{\mathbf{t}} (x)$.     
As in the flag or flag-like cases (cf. e.g. \cite{MRS,WW24}), we consider the lifting space $ \mathbb{R}^m$ and the projection $\pi:\mathbb{R}^m\rightarrow \mathbb{R}^n$ given by
\begin{equation}\label{eq:lifting}
   (\lambda_1,\ldots,\lambda_m )\mapsto \lambda_1e_1+\cdots+\lambda_me_m .
\end{equation}

\begin{thm}  \label{prop:integral-representation}
Suppose that $T_\Omega$ is a tube domain over
 {a} polyhedral cone as above. For $F  \in H^p(T_\Omega)$ with $p>1$, there exists $F^b\in L^p(\mathbb{R}^n)$ such that
  \begin{equation}\label{eq:integral-representation}
 F ( x+\mathbf{i}\pi(\mathbf{t}))=
    F^b*  P_{\mathbf{t}}(x),
 \end{equation}
for
 {all} $( x,\mathbf{t})\in\mathbb{R}^n\times(\mathbb{R}_+)^m$. Moreover, if $F$ is continuous on $\overline{T_\Omega}$, then $
F^b=\lim_{\mathbf{t}\rightarrow 0} F ( x+\mathbf{i}\pi(\mathbf{t}))$.
\end{thm}

\begin{rem}
\textup{(1)}  This reproducing formula is unusual in the sense that we have to use the lifting cone $(\mathbb{R}_+)^m$, where a point $y$ in the cone $\Omega$ is represented by $\pi(\mathbf{t})$.
The iterated Poisson integral has {an $m$-parameter structure}, although $\mathbb{R}^n$ is only $n$-dimensional. Namely, there exist $m-n$ hidden parameters.

\textup{(2)}  This formula
 {identifies} the new type of Hardy spaces on $\mathbb{R}^n$ to which the boundary values of holomorphic Hardy functions belong, and
 {indicates}
 what kind of multi-parameter analysis is needed to
 {study}
 them. We will develop
 {the} necessary tools of the corresponding multi-parameter harmonic analysis on the Shilov boundary $\mathbb{R}^n$ to solve problems concerning holomorphic Hardy functions
 {on} the tube domain.
\end{rem}

The iterated Poisson integral is our approximation to the identity; 
 { it allows us to}
 introduce natural notions of multi-parameter analysis adapted to tube domains $T_\Omega$, such as twisted rectangles, new non-tangential regions, non-tangential maximal functions, Lusin-Littlewood-Paley area 
 functions and new types of Hardy spaces.
 {All of these differ} from the standard ones and have new geometric features associated with polyhedral cones.

For $\mathbf{t} :=\left(t_1 ,  \dots ,t_m \right)\in (\mathbb{R}_+)^m$,
the standard rectangle in $\mathbb{R} ^m$ is $\widetilde{R}(0,\mathbf{t})=(-t_1,t_1)\times\cdots\times (-t_m,t_m)$. Its image $\pi(\widetilde{R}(0,\mathbf{t}))$ under the projection $\pi$ in \eqref{eq:lifting} is
denoted by $R(0,\mathbf{t})$, and is called a {\it twisted rectangle} in
$\mathbb{R}^n$. For $x\in \mathbb{R}^n$, let $ R(x,\mathbf{t}):=x+ R(0,\mathbf{t})$. Then
  \begin{equation}\label{eq:twisted-rectangle}
   R(x,\mathbf{t}):=\{x+\lambda_1e_1+\cdots+\lambda_me_m ; |\lambda_1| < t_1, \dots, |\lambda_m| < t_m \}.
 \end{equation}
For $x\in \mathbb{R}^n$ and any $\beta >0$, we define the {\it non-tangential region in $\mathbb{R}^n \times(\mathbb{R}_+)^m$} by
  \begin{equation}\label{eq:nontangential-region}
     \Gamma_\beta( x):=\left\{(x' ,\mathbf{t} )\in \mathbb{R}^n \times(\mathbb{R}_+)^m ; x' \in R\left(x,\beta  \mathbf{t}  \right)
     \right\}.
  \end{equation}
 { We also denote $\Gamma_1( x)$ by $ \Gamma ( x)$ for short.}
 Define the {\it non-tangential maximal function}
on $\mathbb{R}^n$ by
 \begin{equation}\label{eq:nontangential}
  N^\beta( f )(x) := \sup_{(x',\mathbf{t})\in  \Gamma_\beta(x)}   \big| (f*P_{\mathbf{t} })(x' )\big|.
 \end{equation}
The Hardy space $  H^1_{ \max;\Omega} (\mathbb{R}^n)$
consists of
$f\in L^1( \mathbb{R}^n)$ such that $N^\beta( f )\in L^1(\mathbb{R}^n)$.
For two polyhedral cones ${\Omega}, \Omega'$ with
$\overline{\Omega}\subset\Omega'$, we have $H^1 (T_{\Omega'} )\subsetneqq H^1 (T_\Omega)$ and $ H^1_{ \max;\Omega'}(\mathbb{R}^n) \varsubsetneq  H^1_{ \max; \Omega}(\mathbb{R}^n)$
(cf. Proposition \ref{prop:inclusion}). {  If there is no ambiguity, 
we will omit the dependence of  $H^1_{ \max;\Omega}$ on $\Omega$.}

Define the {\it Lusin-Littlewood-Paley area function} of $f$ by
\begin{equation}\label{eq:Littlewood-Paley-area}
   S  (f)(x):=\left(\int_{ \Gamma (x)}\left| \nabla_1 \cdots \nabla_m( f*P_{\mathbf{t} })(x' )  \right|^2 \frac
   {\mathbf{t}\,d\mathbf{t}\,dx'}{|R(x ,\mathbf{t})|}\right)^{1/2},
\end{equation}
for all $x \in \mathbb{R}^n$, where $\nabla_j$ is the gradient in the plane spanned by $e_j$ and $t_j$, and
$
   \mathbf{t}\,d\mathbf{t}=t_1\cdots t_m \,dt_1\cdots dt_m$.
Define the {\it $S$-function Hardy space} { $H^1_{S}(\mathbb{R}^n)$ by setting} 
\[
    H^1_{S}(\mathbb{R}^n):=\{f\in L^1(\mathbb{R}^n) ;  S (f)\in L^1(\mathbb{R}^n)\}
\]
{ and equipped with the norm 
$ \|f\|_{H_{S }^1(\mathbb{R}^n)}:=\|S (f)\|_{L^1(\mathbb{R}^n)} .$ }
The {\it Littlewood-Paley $g$-function} of $f$ is defined by
\begin{equation*}
   g (f)(x):=\left(\int_{(\mathbb{R}_+)^m}\left| \nabla_1 \cdots \nabla_m( f*P_{\mathbf{t} })  \right|^2(x ) \,\mathbf{t}\,d\mathbf{t}
   \right)^{1/2},
\end{equation*}
for all $x \in \mathbb{R}^n$. We define the {\it $g$-function Hardy space}
\[
    H^1_{g }(\mathbb{R}^n):=\{f\in L^1(\mathbb{R}^n) ;  g (f)\in L^1(\mathbb{R}^n)\}
\]
{ and this space is naturally equipped with   the   norm}
  $\|f\|_{H_{g }^1(\mathbb{R}^n)}:=\|g (f)\|_{L^1(\mathbb{R}^n)}$.

Note that the  tube domain   $
T_\Omega=\mathbb{R}^n+\mathbf{i}\Omega $  is different from $\mathbb{R}^n \times(\mathbb{R}_+)^m$, but  $ \Omega=\pi((\mathbb{R}_+)^m )$. If we lift the function $F$ on    $
T_\Omega $  to a function
$
 u(x,\mathbf{t})= F(x  +\mathbf{i}\pi(\mathbf{t})   )
$
on
 $\mathbb{R}^n \times(\mathbb{R}_+)^m$, then by Theorem
 \ref{prop:integral-representation}, $u(x,\mathbf{t})=  F^b*  P_{\mathbf{t}}(x)$ if $  F  \in H^p(T_\Omega)$ with $p>1$. Motivated by  nontangential  maximal function,  area function and $g$-function for  iterated 
 Poisson  integrals above, we
  define the   {\it nontangential  maximal function} of a holomorphic function $F$
   on $T_\Omega$ by
 \begin{equation}\label{eq:nontangential}
  {\mathbb{ N}}^\beta( F )(x) := \sup\limits_{(x',\mathbf{t})\in  \Gamma_\beta(x)}   \left|F(x' +\mathbf{i}\pi(\mathbf{t})   ) \right|,
 \end{equation}for all $x \in \mathbb{R}^n$. { We also define 
the} {\it Lusin-Littlewood-Paley area function}
by
\begin{equation*}
\mathbb{   S  }(F)(x):=\left(\int_{ \Gamma (x)}
\left| \nabla_1 \cdots \nabla_m(F (x'+\mathbf{i}\pi(\mathbf{t})  )  \right|^2 \frac
   {\mathbf{t}d\mathbf{t}dx'}{|R(x ,\mathbf{t})|}\right)^{\frac 12},
\end{equation*}
{ and the}   {\it Littlewood-Paley $g$-function}   by 
    \begin{equation*}
    \mathbb{  {G} }(F) (x):=\left(\int_{(\mathbb{R}_+)^m  }
\left| \nabla_1 \cdots \nabla_m(F (x +\mathbf{i}\pi(\mathbf{t})  )  \right|^2
 { \mathbf{t}}  {d\mathbf{t}  }\right)^{\frac 12}.
\end{equation*}

{ Having established this notation, we are now able to state our main theorem.}

 \begin{thm}\label{thm:equivalence}  Suppose that $T_\Omega$ is a tube domain  over
 {a} polyhedral cone as above. Then, for a  holomorphic  function $F$  on
$ T_\Omega $, the following  properties are equivalent:
\\
\textup{(1)} $F\in H^1( T_\Omega)$;
\\
\textup{(2)} For any $\beta>0$,
  $\mathbb{ N}^\beta(F )\in L^1(\mathbb{R}^n)$;
  \\
\textup{(3)}    $\lim_{\substack {| y| \rightarrow +\infty\\
  y\in y_0+\Omega} } F(x+\mathbf{i} y )= 0 $ for any fixed $y_0\in  \Omega$,  and $\mathbb{ S}
(F)\in L^1(\mathbb{R}^n)$. \\
\textup{(4)}    $\lim_{\substack {| y| \rightarrow +\infty\\
  y\in y_0+\Omega} } F(x+\mathbf{i} y )= 0 $   for any fixed $y_0\in  \Omega$, and $\mathbb{ G}
(F)\in L^1(\mathbb{R}^n)$.
\\
Moreover,
 \begin{equation*}\|F\|_{H^1( T_\Omega)} \approx \|  \mathbb{ N}^\beta(F)\|_{L^1( \mathbb{R}^n)}\approx \| \mathbb{ S} (F)\|_{L^1( \mathbb{R}^n)}\approx
    \| \mathbb{ G} (F)\|_{L^1( \mathbb{R}^n)} .
 \end{equation*}
\end{thm}

\begin{rem}
A key step in our proof is to approximate holomorphic $H^1$ functions by holomorphic $H^p$ functions with $p>1$. For the latter class, we can apply Theorem \ref{prop:integral-representation} to establish a 
correspondence with real-variable $H^p$ functions defined via the iterated Poisson integral on the Shilov boundary. Then, by employing the tools of multi-parameter analysis we have { outlined} and
using the holomorphicity of $F$, we obtain a complete characterization of holomorphic $H^1$ Hardy spaces on
such tube domains. This approach also provides a partial characterization of a new type of Hardy space on the Shilov boundary $\mathbb{R}^n$.
\end{rem}

 For the equivalence of characterizations of Hardy spaces of real variables  in various settings, see, for example,  {\cite{CHW,CCLLO,CDLWY,DLWY,GLY,HLLW,HLS,YZ}}.
 
The proof of Theorem \ref{thm:equivalence}
 {proceeds}
 as follows. For holomorphic functions $F$, we can use the subharmonicity of $|F|^p$ for $p>0$ to prove $(1)\Rightarrow(2)$. The implication $(2)\Rightarrow(3)$ can be reduced to a Fefferman--Stein type 
 good-$\lambda$ inequality: there exist $ C > 0$ and  $\beta > 1$ such that for $f\in L^1( \mathbb{R}^n)$ and all $\lambda> 0$,
\begin{equation}\label{eq:Fefferman-Stein-inequality}
\begin{split}
\left|  \left \{x\in \mathbb{R}^n; { S  }(f)(x)> \lambda \right\}\right|
&\leq C\left| \left \{x\in  \mathbb{R}^n ; {N} ^\beta(f)(x)> \lambda\right\}\right|
+ \frac
C{\lambda^2}\int_{\{ x\in  \mathbb{R}^n ;   {N} ^\beta( f )(x)\leq \lambda\}}  {N} ^\beta( f )(x)^2 \,dx,
 \end{split}
\end{equation}
where $C$ is a constant
independent of $f$ and $\lambda$.
In the one-parameter case, this inequality was proved by Fefferman and Stein \cite{FeffS72, St70} via smooth surface approximation and an application of Green's formula in smooth domains. The main difficulty in 
the bi-parameter case is the lack of such an effective approximation, owing to the higher-dimensional geometric complexity.
To overcome this difficulty for biharmonic functions, Malliavin and Malliavin \cite{MM} introduced an auxiliary function. So far, maximal function
 {characterizations have} been established only in a few cases by proving analogues of \eqref{eq:Fefferman-Stein-inequality}: the $\mathbb{R}^n\times \mathbb{R}^n$ case  by Merryfield \cite{Merry}, the
Muckenhoupt--Stein Bessel operator setting by { Doung, Li, Wick, Yang \cite{DLWY},}
 the multi-parameter flag setting by  { Han, Lee, Li, Wick}  \cite{HLLW,HLLW25},
product stratified Lie groups by { Cowling, Fan, Li, Yan \cite{CLLP},} the flag setting
 {on} the Heisenberg group by  { Chen, Cowling, Lee, Li, Ottazzi \cite{CCLLO}} and the Shilov boundary of products of domains of finite type by Li \cite{L}, etc.
In our {setting},
 the difficulty lies
 {in the fact} that the number of parameters is arbitrary, unlike the aforementioned situations which involve only two. Thus we have to find general inductive formulae to overcome this difficulty.

In the product case $\mathbb{ R}_+^{n_1+1}\times \mathbb{ R}_+^{n_2+1}$, the implication $(3)\Rightarrow(1)$ was proved { by use of} the estimate $\|\mathfrak f \|_{1}   \lesssim \|
S (\mathfrak f)\|_{1}$ for Hilbert space-valued functions $\mathfrak f$
of one variable \cite{Sato}. Because our non-tangential region  $\Gamma_\beta( x)$ in \eqref{eq:nontangential-region}
 {cannot}
 be written as a product, we
 {cannot} { employ} this approach. However, since the domain $(\mathbb{R}_+)^m$ of integration for the $g$-function has a product structure, we add condition (4), which usually does not appear in 
 characterizations of (generalized) analytic Hardy spaces
(cf. e.g. \cite{Gel,Sato,SW}).
The implication $(3)\Rightarrow(4)$ follows from a pointwise control of the $\mathbb{ G}$-function by the $\mathbb{ S}$-function via the mean value formula.
The
 {real-variable} version of $(4)\Rightarrow(1)$ is
\begin{equation}\label{eq:3->1}
  \sup_{\mathbf{t}\in(\mathbb{R}_+)^m} \|f*P_{\mathbf{t}}\|_{L^1(\mathbb{R}^n)}\lesssim \|g(f)\|_{L^1(\mathbb{R}^n)}.
\end{equation}
We first establish a one-dimensional version of \eqref{eq:3->1} for Hilbert space-valued functions by applying the corresponding one-dimensional Plancherel-P\'olya type inequality. Then we use this estimate 
iteratively to prove
\eqref{eq:3->1}.

The paper is organized as follows. In Section 2, we discuss partial convolutions with the Poisson kernel along lines, which satisfy partial Laplace equations. In Section 3, we characterize the geometric shape of 
twisted rectangles and prove the equivalence of iterated and twisted maximal functions. We also
 {obtain} boundary growth estimates for iterated Poisson integrals and for holomorphic $H^p$-functions, respectively, and then show
 {that}
 the iterated Poisson integral reproduces holomorphic functions by using the maximum principle.
In Section 4, the
Fefferman-Stein type good-$\lambda$ inequality \eqref{eq:Fefferman-Stein-inequality} is reduced to the estimate of an integral over $\mathbb{R}^n
\times  (\mathbb{R}_+)^m$
  {by means of}
 an auxiliary function.
 {We use differential identities to estimate the terms appearing in this integral, the new terms arising after these estimates, and in general all terms appearing inductively.}
 In Section 5, we prove a one-dimensional version of \eqref{eq:3->1} for Hilbert space-valued functions, from which we deduce the
 {real-variable} version \eqref{eq:3->1} of $(4)\Rightarrow(1)$. In Section 6, we prove the equivalence of the various characterizations in the main Theorem \ref{thm:equivalence}.
In the appendix, we give the proof of the Plancherel-P\'olya type inequality for Hilbert space-valued functions for the convenience of the readers.

 \section{Preliminaries}
\subsection{Partial convolutions  along  lines and partial Laplace equations}

 Consider the {\it complex line} in $\mathbb{C}^n$  that contains    $e_\mu$:
 \begin{equation}\label{eq: L-mu}
    L_\mu:=\{(s +\mathbf{i}t )e_\mu ; s ,t \in \mathbb{R}\},
 \end{equation}and its upper half plane
 \begin{equation}\label{eq: L-mu+}
    L_\mu^+=\{(s +\mathbf{i}t )e_\mu; s  \in \mathbb{R},t >0\},
 \end{equation}with   boundary $\mathbb{R} e_\mu $ contained in the Shilov bounday $\mathbb{R}^n+\mathbf{i}0$ of $T_\Omega$ .
 Then $x+L_\mu^+$ may be contained in  the boundary of  the  tube domain  for any $x\in \mathbb{R}^n$, because $ e_\mu$ may be contained in $ \partial\Omega$. But we always have $x+iy+L_\mu^+\subset T_\Omega$ for 
 $y\in \Omega$.

  Denote by $  { {*}}_\mu$   the {\it partial convolution on $\mathbb{R}^n$ along the line  $ e_\mu $}, i.e. for
$f\in L^1(\mathbb{R}^n)$ and $ g\in L^1(\mathbb{R} )$,
 \begin{equation}\label{eq:convolution-subgroup} \begin{split}
f*  _\mu g (x   ):
=&\int_{  \mathbb{R} }  f\left(x-se_\mu \right)g ( s) d s
 ,
 \end{split}\end{equation}for almost all $x\in \mathbb{R}^n$. It is obvious that partial  convolutions   along two lines  $ e_\mu $ and $ e_\nu $ are commutative:
  \begin{equation}\label{eq:convolution-commutativity}
    f*  _\mu g  *_\nu h(x   )
=f*  _\nu h *_\mu g (x   ),
 \end{equation} for any  $ g,h\in L^1(\mathbb{R} )$.
 If
 write
 \begin{equation}\label{eq:e-mu}
     e_\mu= { (e_{\mu 1},\dots,e_{\mu n})} \in \mathbb{R}^n,
 \end{equation} 
  { $\mu=1,\dots,m,$}
 then the vector field
 \begin{equation}\label{eq:X_j}
   X_\mu:=\sum_{ k=1 }^n e_{\mu k}  \frac {\partial }{\partial x_k  },
 \end{equation}
  is the derivative along the direction $e_\mu$.  Define the {\it $\mu$-th Laplacian} and the gradient of a function $u$ along the upper half plane $L_\mu^+$ in \eqref{eq: L-mu+} as
 \begin{equation*}
    \triangle_\mu:=X_\mu^2+  \frac {\partial^2 }{\partial t_\mu^2},
\qquad
      \nabla_\mu u:=\left( X_\mu u ,\frac {\partial u }{\partial t_\mu } \right),
   \end{equation*}respectively,
 { $\mu=1,\dots,m$.}
 \begin{prop}\label{prop:triangle-j}  (1) If $F$ is holomorphic on  the  tube domain $  T_\Omega $, then { for fixed $x+\mathbf{i}y\in  T_\Omega $, the restriction of $F$ to} $x+\mathbf{i}y+L_\mu^+$ is  
 holomorphic,
    i.e. $F(x+\mathbf{i}y+ze_\mu)$ is holomorphic in $z\in \mathbb{C}$ for $\operatorname{Im} z>0$.
    \\
    (2) For  $f\in { \mathscr C^2(\mathbb{R}^n)}\cap L^1 (\mathbb{R}^n)$,     we have
    \begin{equation}\label{eq:holomorphic}
       \triangle_\mu [f*_\mu  P_{  t_\mu
     }(x)]=0,
   \qquad
 {\rm   and }
  \qquad
       \nabla_\mu \left[f*_\mu  P_{  t_\mu     }(x)\right]=f*_\mu \widetilde{\nabla}_\mu  P_{  t_\mu     }(x) ,
    \end{equation}
    where $ \widetilde{\nabla}_\mu $ is the gradient in the plane with variables $s_\mu,   t_\mu   $:
     $
       \widetilde{\nabla}_\mu v(  s_\mu,   t_\mu)  =\left(\frac {\partial v}{\partial s_\mu   }(  s_\mu,   t_\mu),\frac {\partial  v}{\partial {  t_\mu     }   }(  s_\mu,   t_\mu)
       \right).
  $
     \end{prop}
 \begin{proof} (1) If we write the coordinates $z_1,\dots,z_n$ of the tube domain as $z_j=x_j+\mathbf{i}y_j$  and $z=s+\mathbf{i}t$, then, { for fixed $\mu$ we have} 
  \begin{equation*}\begin{split}
    2\frac {\partial F}{\partial \overline{z}}&=\left(\frac {\partial   }{\partial s}+\mathbf{i}\frac {\partial  }{\partial
    t}\right)[F(x+\mathbf{i}y+(s+\mathbf{i}t)e_\mu)]\\&= { \sum_{k=1}^n}
    e_{\mu k}\frac {\partial F }{\partial
    x_k}(x+\mathbf{i}y+(s+\mathbf{i}t)e_\mu)+\mathbf{i}
    { \sum_{k=1}^n} e_{\mu k}\frac {\partial  F }{\partial y_k}(x+\mathbf{i}y+(s+\mathbf{i}t)e_\mu)\\
    &={ \sum_{k=1}^n}  e_{\mu k}\left(\frac {\partial  F }{\partial
    x_k}+\mathbf{i} \frac {\partial  F}{\partial y_k}\right) \left (x+\mathbf{i}y+(s+\mathbf{i}t)e_\mu\right)=0,
 \end{split}  \end{equation*}
 by using \eqref{eq:e-mu} and the Cauchy-Riemann equation satisfied by $F$.

 (2)   It is obvious that $f*_\mu  P_{  t_\mu     }(x)$   {lies in $ \mathscr C^2 (\mathbb{R}^n\times \mathbb{R}_+)$. } Since
  \begin{equation}\label{eq:directional}
    \frac {\partial}{\partial s} [ f(x-se_\mu)]=-X_\mu f(x-se_\mu),
 \end{equation}by definition \eqref{eq:X_j}, we get
 \begin{equation*}\begin{split}
   \triangle_\mu [f*_\mu  P_{  t_\mu     }(x)]=& \int_{\mathbb{R}} X_\mu^2f(x-se_\mu)P_{  t_\mu     }(s)ds+\int_{\mathbb{R}}  f(x-se_\mu) \frac {\partial^2 }{\partial t_\mu^2} P_{
   t_\mu     }(s)ds\\
   =& \int_{\mathbb{R}} \frac {\partial^2 }{\partial s^2} [ f(x-se_\mu)] P_{  t_\mu     }(s)ds+\int_{\mathbb{R}}  f(x-se_\mu) \frac {\partial^2 }{\partial t_\mu^2} P_{  t_\mu
   }(s)ds\\
   =&  \int_{\mathbb{R}} f(x-se_\mu) \left(\frac {\partial^2 }{\partial s^2}+\frac {\partial^2 }{\partial t_\mu^2} \right)P_{  t_\mu     }(s)ds=0,
 \end{split}  \end{equation*}by integration by parts,  and
 \begin{equation*}\begin{split}
   \nabla_\mu [f*_\mu  P_{  t_\mu     }(x)]=& \left(\int_{\mathbb{R}} X_\mu f(x-se_\mu)P_{  t_\mu     }(s)ds, \int_{\mathbb{R}}  f(x-se_\mu) \frac {\partial  }{\partial t_\mu } P_{
   t_\mu     }(s)ds\right)\\
   =& \left(\int_{\mathbb{R}} -\frac {\partial  }{\partial s } [ f(x-se_\mu)] P_{  t_\mu     }(s)ds,\int_{\mathbb{R}}  f(x-se_\mu) \frac {\partial  }{\partial t_\mu } P_{  t_\mu
   }(s)ds\right) \\
   =&  \int_{\mathbb{R}}  f(x-se_\mu)\widetilde{\nabla}_\mu P_{  t_\mu     }(s)ds.
 \end{split}  \end{equation*}
 The proposition is proved.\end{proof}

Denote $ \mathfrak m :=\{1 ,\dots,m\}$. For a subset $  \mathfrak j=\{ {j_1},\dots, j_a\}$ of $\mathfrak m $, { we introduce the notation} 
 \begin{equation}\label{eq:nabla-j}\mathbf{t}_{\mathfrak j}:=(t_{j_1},\dots, t_{j_a}),\qquad \qquad
   \nabla_{\mathfrak j }v:=\nabla_{j_1}\cdots \nabla_{j_a}v,\qquad \qquad  \mathbf{t}_{\mathfrak j} d\mathbf{t}_{\mathfrak j}:=t_{j_1} \cdots  t_{j_a}d t_{j_1} \cdots d t_{j_a},
\end{equation}
and for  $f\in   L^1 (\mathbb{R}^n)$ and $\phi\in L^1(\mathbb{R})$, { we set} 
 \begin{equation*}
 f *_{\mathfrak j } \phi_{\mathbf{t}_{\mathfrak j }}:=f *_{j_1} \phi_{t_{j_1}}\cdots  *_{j_a} \phi_{t_{j_a}}(x),
 \end{equation*} where $\phi_{  t}(s)= \frac 1t\phi (\frac st)$.
Let
 \begin{equation*}
   \mathfrak m_k=\{k ,\dots,m\}.
 \end{equation*}
 In { view of this notation}, we have $\mathbf{t}_{\mathfrak m}=\mathbf{t} $.
We will write $f * \phi_{\mathbf{t} }$ instead of $f *_{\mathfrak j } \phi_{\mathbf{t}_{\mathfrak j }}$ when $\mathfrak j =\mathfrak m$.

  For a subset ${\mathfrak j} $ of  $ \mathfrak m$ and $f \in L^p(\mathbb{R}^n )$ ($1\leq p<\infty$),  define the  {\it partial   Littlewood-Paley $g$-function} of $f$ by 
\begin{equation*}
   g_{{\mathfrak j}} (f)(x):=\left(\int_{(\mathbb{R}_+)^{|{\mathfrak j}|}}\left| \nabla_{\mathfrak j }(  f *_{\mathfrak j } P_{\mathbf{t}_{\mathfrak j }})(x )  \right|^2\mathbf{t} _{{\mathfrak
   j}}d\mathbf{t}_{{\mathfrak j}}
   \right)^{\frac 12}
    .
\end{equation*}

\begin{prop} \label{thm:g-function}    For    $f \in L
^p
(\mathbb{R}^n)$ with $1<p<\infty$, we have
$
      \|g_{{\mathfrak j}}  (f)\|_p\lesssim \|f\|_p
$.
\end{prop}
The proof is standard (cf. e.g. \cite{HLS,WW24}) by repeatedly using the $L^p$ boundedness of $g$-function for vector-valued functions.
See also the proof of Proposition \ref{prop:g-L1} for the converse part { when} $p=1$.
  We omit { the} details.

\subsection{A formula of integration by parts } It will be repeatedly used in the proof of     good-$\lambda$ inequality. For fixed $\mu$, we will use coordinates $x=x^\perp+ s e_\mu$, where $x^\perp$ is the part 
of $x$ perpendicular to $e_\mu$.
\begin{prop}\label{prop:Green}  Suppose that $U\in { \mathscr C^2}(\mathbb{R}^{n +1}_+)$ satisfies { the following three properties:} \\
\textup{(1)} For fixed $t_\mu>0$, we have $    \nabla^{a}  U (\cdot,t_\mu)  \in L^1(\mathbb{R}^n ) $, $a=0,1,2$, and   $ X_\mu  U(x^\perp+se_\mu,t_\mu) \rightarrow 0$    as $ s\rightarrow
\infty$ for almost all $x^\perp \in e_\mu^\perp$;
\\
\textup{(2)}  As $t_\mu\rightarrow  +\infty $,  we have $  U(\cdot,t_\mu),  t_\mu \partial_{t_\mu} U(\cdot,t_\mu) \rightarrow  0$  in $L^1(\mathbb{R}^n)$;
\\
\textup{(3)} As $t_\mu\rightarrow 0 $,  we have $  U(\cdot,t_\mu) \rightarrow U(\cdot,0)$ and $t_\mu\partial_{t_\mu}   U (\cdot,t_\mu) \rightarrow 0$ in $L^1(\mathbb{R}^n)$.
\\ Then
\begin{equation}\label{eq:Green}
  \int_0^\infty\int_{\mathbb{R}^n}  \triangle_\mu U  (x, {t}_\mu)   t_\mu dt_\mu  dx =  \int_{\mathbb{R}^n}    U   (x,0)\, dx.
\end{equation}
\end{prop}

\begin{proof} Write
\begin{equation*}\begin{split}
 \int_\varepsilon^T\int_{\mathbb{R}^n}   \triangle _\mu U  (x, {t}_\mu)  t_\mu dt_\mu  dx = &
 \int_\varepsilon^T\int_{\mathbb{R}^n}   X_{\mu  } ^2   U (x, {t}_\mu)  t_\mu dt_\mu dx  +  \int_\varepsilon^T\int_{\mathbb{R}^n}    \partial^2_{t_\mu}    U (x, {t}_\mu)
t_\mu dt_\mu dx
 =:{\rm I}_1+{\rm I}_2.
\end{split} \end{equation*}
Note that $x\rightarrow (x^\perp, s)$ preserves the volume form since we
assume $e_\mu$ to be a unit vector.
Hence,
\begin{equation*}\begin{split}
\int_{\mathbb{R}^n}   X_{\mu   }^2   U (x,t_\mu) dx&= \int_{\mathbb{R}^{n-1}}  \int_{\mathbb{R} }\frac {d^2}{ds^2}[ U(x^\perp+se_\mu,t_\mu)] \, d s\,  { dx^\perp} \\&= \int_{\mathbb{R}^{n-1}}  
\int_{\mathbb{R} }\left(\lim_{s\rightarrow+\infty} X_{\mu   }U(x^\perp+se_\mu,t_\mu)-\lim_{s\rightarrow-\infty} X_{\mu   }U(x^\perp+se_\mu,t_\mu)\right)  \, d s\,  { dx^\perp} 
=0
\end{split}\end{equation*} by using \eqref{eq:directional} and the assumption (1).
Thus,
$ {\rm I}_1 = 0  $.
For the term ${\rm I}_2$, by integration by parts, we get
\begin{equation*}\begin{split}
  {\rm I}_2&=-\int_{\mathbb{R}^n}   \int_\varepsilon^T \partial_{t_\mu}    U (x, {t}_\mu)
 dt_\mu dx + \int_{\mathbb{R}^n}     T  \partial_{t_\mu} U (x, T)dx
    -\int_{\mathbb{R}^n}  \varepsilon\partial_{t_\mu} U (x,\varepsilon) dx
\\& =  - \int_{\mathbb{R}^n}        U (x, T) dx
    + \int_{\mathbb{R}^n}  U (x,\varepsilon)  dx + \int_{\mathbb{R}^n}  T  \partial_{t_\mu}      U (x, T)dx
  - \int_{\mathbb{R}^n}\varepsilon \partial_{t_\mu}   U (x,\varepsilon) dx.
\end{split} \end{equation*}
Then   taking   limit $T\rightarrow+\infty$ and $\varepsilon\rightarrow0$, we 
{ deduce the required conclusions in view of}
 assumptions (2)  and  (3).
 \end{proof}

\section{New maximal functions associated to twisted  rectangles  and iterated Poisson  integral formula   }

\subsection{Twisted  rectangles,  iterated and twisted
  maximal functions
}

Write $\nabla_\mu=(X^{(1 )}_{\mu},X^{(2) }_{\mu})$, i.e. $X^{(1) }_{\mu}=X_{ \mu}$ and  $X^{(2) }_{\mu}=\partial_{t_ \mu}$. Then $\nabla_{\mathfrak j }v$ in \eqref{eq:nabla-j} is a
vector valued function  with  $\mathbb{R}^{2^a}$ entries, and
their inner product is defined as
  \begin{equation*}
   \nabla_{\mathfrak j }u\cdot\nabla_{\mathfrak j }v:=\sum_{\alpha_1,\dots,\alpha_a=1,2}  X_{j_1 }^{(\alpha_1)}{ \cdots  } X_{j_a }^{(\alpha_a)}u \cdot X_{j_1 }^{(\alpha_1)}{ \cdots  }  
   X_{j_a
   }^{(\alpha_a)}v.
\end{equation*}

 Define {\it maximal function     along the line  $ e_\mu $} as
 \begin{equation*}\begin{split}
   M_{ \mu}  (f)(x):=\sup_{ {r}_\mu\in \mathbb{R} _+} \frac 1{2r_\mu }\int_{
  |s|<r_\mu   }|f (x+s e_\mu)|ds,
 \end{split}\end{equation*}for $x\in \mathbb{R}^n$ and $f\in L^1_{loc}(\mathbb{R}^n) $.
It is convenient to { work with} the {\it iterated maximal function}
$
   M_{it}   := M_m\circ\dots \circ M_1.
$ Then for $f \in  L^1_{loc}(\mathbb{R}^n) $, there exists an absolute constant $C_0>0$ such that
\begin{equation}\label{eq:heat-max}
    |P_{\mathbf{t}}( f )(x') |\leq C_0M_{it}    (f)(x)
\end{equation}for any $(x',\mathbf{t} )\in \Gamma(x)$, by the property of  the Poisson kernel \cite{St70}.
Define the {\it twisted maximal function}   as
\begin{equation*}
   M_t   (f)(x):=\sup_{\mathbf{r}\in (\mathbb{R}_+)^m} \frac 1{|R (x,\mathbf{r})|}\int_{ R (x,\mathbf{r})}|f (x')|dx'.
\end{equation*}

Note that the   projection  of the ball in the lifting space is usually called tube in  the flag or
flag-like  cases (cf. e.g. \cite{CCLLO,WW24}). But in this paper, the   projection
$ R(x,\mathbf{r})$ in \eqref{eq:twisted-rectangle}
 is  called
twisted  rectangle for avoiding confusion with the tube domain.
Given  $\mathbf{r}=(r_1,\dots,r_m)  \in (\mathbb{R}_+)^m $, let $r_{l_1}, \dots,r_{l_n}$ be the largest $n$ numbers among $\{r_{ 1}, \dots,r_{ m}\}$, i.e. $r_j\geq r_\mu$ for any
$j\in \mathfrak l:=\{l_1 ,\dots, l_n\}$ and any $\mu\notin \mathfrak l$.
Set
\begin{equation}\label{eq:Rl}
   R_{\mathfrak l}(x,\mathbf{r}):=\{x+\lambda_{l_1}e_{l_1}+{ \cdots}+\lambda_{l_n} e_{l_n};|\lambda_{l_1}| < r_{l_1}, \dots, |\lambda_{l_n}| < r_{l_n} \},
 \end{equation}
which is a parallelohedron. Note that given  $\mathbf{r}$, the choice of $\mathfrak l$ may not be unique.

On the other hand, given any subset  $\mathfrak l:=\{ l_1 ,\dots, l _n\}$ of $\mathfrak m $,  $e_{l_1},\ldots, e_{l_n}$ are linearly independent by the assumption. So there exists numbers $A^{\mathfrak
l}_{\mu j}$ such that
 $
    e_\mu=\sum_{j=1}^n A^{\mathfrak l}_{\mu j} e_{ l_j},
$
 for $\mu\notin\mathfrak l$. { As} the number of   choices of such subsets $\mathfrak l$ is finite, we can define
  \begin{equation}
 \label{eq:A}
  \mathcal{ A}:= \sum_{\mathfrak l }\sum_{ \mu \notin \mathfrak l }\sum_{   j=1}^n\left|A^{\mathfrak l}_{\mu j}\right| <+\infty .
 \end{equation}

\begin{prop}  \label{prop:twist-R}   Given  $\mathbf{r} \in (\mathbb{R}_+)^m $, let $r_{l_1}, \dots,r_{l_n}$ be the largest $n$ numbers among $\{r_{ 1}, \dots,r_{ m}\}$. Then
\begin{equation} \label{eq:RRR}
   R_{\mathfrak l}(x,\mathbf{r})\subset   R (x,\mathbf{r})\subset  R_{\mathfrak l}(x,\widetilde{\gamma}_0 ^{-1}\mathbf{r}),
\end{equation} for any $x\in \mathbb{ R}^n $, where $\widetilde{\gamma}_0=(1+ \mathcal{ A})^{-1}<1$.
\end{prop}
 \begin{proof} The first inclusion is obvious by comparing definitions \eqref{eq:twisted-rectangle} and \eqref{eq:Rl}.
 For the second one, note that for any $x'\in  R (x,\mathbf{r})$, we can write
  \begin{equation*} \begin{split}
  x'=&x+\lambda_1e_1+{ \cdots}+\lambda_me_m=x+\sum_{j=1}^n\left(\lambda_{ l_j}+ \sum_{\mu\notin \mathfrak l}  A^{\mathfrak l}_{\mu j}\lambda_{\mu}\right)e_{ l_j} ,
  \end{split}  \end{equation*}
for some $|\lambda_1| <r_\mu\leq r_1, \dots, |\lambda_m| < r_m $. Then
    \begin{equation*}
     \left|\lambda_{ l_j}+ \sum_{\mu=n+1}^m A^{\mathfrak l}_{\mu j}\lambda_{\mu}\right|\leq (1+ \mathcal{ A})r_{ l_j}<\widetilde{\gamma}_0^{-1}r_{ l_j},
  \end{equation*}
    since $|\lambda_{\mu} |\leq  r_{ l_j}$ for $\mu\notin\mathfrak l$ by definition.
  \end{proof}
  As a corollary, we get
    \begin{equation}\label{eq:volum}
     |R (x,\mathbf{r})|  \approx  |R_{\mathfrak l}(x,\mathbf{r})|=2^nr_{l_1}  \cdots  r_{l_n}\det(e_{l_1},\dots, e_{l_n}).
  \end{equation}
Let $\chi   $ be the   characteristic function of the interval $[-1,1]$ and let 
{ $\chi_{t}(s):=\frac 1{t}\chi(\frac sr )$ for $t\in \mathbb R^+$ and $s\in \mathbb R$. Recall that  
$|f|* \chi_{\mathbf{t} } := |f|*_{j_1} \chi_{t_{j_1}}   \cdots *_{j_m }\chi_{t_{j_m}} $ 
for $\mathbf{t} =(t_{j_1}, \dots, t_{j_m})  \in (\mathbb{R}_+)^m $.} 
The iterated maximal function and twisted maximal
function are equivalent by the following proposition.
 \begin{prop} \label{prop:it-twist} Suppose that $ f \in L_{loc}^1( \mathbb{ R}^n )$. Then   for any 
 { $n<m$,} $x \in \mathbb{ R}^n$, and  $\mathbf{t}  \in (\mathbb{R}_+)^m $,
\begin{equation*}
     |f|*  \chi_{\gamma_0 \mathbf{t} }(x) \lesssim \frac 1{| R  (x,\mathbf{t})|}\int_{ R(x,\mathbf{t})}|f (x')|dx' \lesssim
     |f|*  \chi_{\gamma_0 ^{-1} \mathbf{t} }(x),
\end{equation*}where $\gamma_0=(1+ \mathcal{ A})^{-2}$.
Consequently, $M_t   (f) \approx M_{it}  (f)$.
    \end{prop}
 \begin{proof}  Without loss of generality, we { may assume that $t_{ 1}, \dots,t_{ n}$ are the largest} $n$ numbers among $\{t_{ 1}, \dots,t_{ m}\}$. The linear  independence of  $e_{1},\ldots, e_{n}$  
 implies that there exists numbers $A_{\mu j}$ such that
  \begin{equation}\label{eq:e-mu-ej}
    e_\mu=\sum_{j=1}^n A_{\mu j} e_{ j},
 \end{equation} for any $\mu=n+1,\dots ,
 m$.
Then
\begin{equation} \label{eq:f-chi}\begin{split}
     |f|*  \chi_{ \mathbf{t} }(x)&=\int_{\mathbb{  R}^m} \left |f\left(x-\sum_{j=1}^ms_j e_j\right)\right|\chi_{t_1}(s_1)\cdots \chi_{t_m}(s_m)ds_1 \cdots ds_m\\
    &=\frac 1{t_1 \cdots  t_m}\int_{[-  t_1, t_1]\times\cdots\times [-  t_m, t_m]} \left |f\left(x-\sum_{j=1}^n\left(s_{ j} +\sum_{\mu=n+1}^m  A_{\mu j}s_{\mu}\right)e_{ j}\right)
    \right  |ds_1 \cdots ds_m.
 \end{split} \end{equation}
 Take the coordinate  transformation $\Theta:\mathbb{R}^m\rightarrow \mathbb{R}^m$, $\mathbf{s}\mapsto \mathbf{s}'$ given by
  \begin{equation}\label{eq:transformation-Theta}
    s_j'=\left\{\begin{array} {ll}  s_j+\sum_{\mu=n+1}^m  A_{\mu j}s_{\mu},\qquad & j=1,\ldots,n,\\ s_{j},\qquad & j=n+1,\ldots,m.
 \end{array}   \right.
 \end{equation}
  It is obvious that $\Theta$ is one to one and its   Jacobian    identically equals to $1$. We claim that
  \begin{equation}\label{eq:inclusion}
    \prod_{a=1}^m [-\widetilde{\gamma}_0 t_a, \widetilde{\gamma }_0 t_a]\subset \Theta\left(  \prod_{a=1}^m [- t_a,   t_a] \right)\subset  \prod_{a=1}^m\left [-\widetilde{\gamma}_0^{-1} t_a,\widetilde{ 
    \gamma}_0^{-1} t_a\right],
 \end{equation} with $\widetilde{\gamma}_0
 =(1+ \mathcal{ A})^{-1}$.
  For the first inclusion, note that for each $  \mathbf{s }'\in  \prod_{a=1}^m [-\widetilde{\gamma}_0 t_a, \widetilde{\gamma}_0 t_a]$, there exists a unique preimage   $\mathbf{s }$ in $\prod_{a=1}^m [- t_a,
  t_a] $, i.e.  $\Theta(\mathbf{s })=\mathbf{s }'$. This is because $ | s_j|\leq t_j$ obviously for $j=n+1,\ldots,m$ by definition \eqref{eq:transformation-Theta}, while for $j=
  1,\ldots,n$, we have
  \begin{equation*}
   | s_j|=\left|s_j'-\sum_{\mu=n+1}^m  A_{\mu j}s_{\mu}'\right|\leq \widetilde{\gamma}_0 t_j+ \mathcal{ A} \widetilde{\gamma}_0 t_j= t_j,
 \end{equation*} by $|s_{\mu}'|\leq \widetilde{\gamma}_0 t_{\mu}\leq \widetilde{\gamma}_0 t_j$ for any $\mu>n\geq j$ by the assumption. It is similar to show the second inclusion.

Now it follows from \eqref{eq:f-chi} and \eqref{eq:inclusion} that
 \begin{equation*} \begin{split}
     |f|*  \chi_{ \mathbf{t} }(x)&\geq \frac 1{t_1 \cdots  t_m}\int_{[-\widetilde{\gamma}_0 t_1,\widetilde{\gamma}_0 t_1 ]\times\cdots\times [-\widetilde{\gamma}_0 t_m,\widetilde{\gamma}_0 t_m ] }
     \left |f\left(x-\sum_{j=1}^n s_{ j} ' e_{ j}\right)\right|ds_1'\cdots  ds_m' \\
     &= \frac {(2\widetilde{\gamma}_0)^{m-n}}{t_1 \cdots  t_n}\int_{[-\widetilde{\gamma}_0 t_1,\widetilde{\gamma}_0 t_1 ]\times\cdots\times [-\widetilde{\gamma}_0 t_n,\widetilde{\gamma}_0 t_n ] }
     \left |f\left(x-\sum_{j=1}^n s_{ j} ' e_{ j}\right)\right|ds_1'\cdots  ds_n' \\&  =\frac { \widetilde{\gamma}_0 ^{m-n}2^m}{ |   R_{\mathfrak l} (0, \mathbf{t})|} \int_{ R_{\mathfrak l}(0,\widetilde{\gamma}_0 
     \mathbf{t}) }| f\left(x-x'\right) |  dx'  \\
     &{ \geq  \frac { \widetilde{\gamma}_0 ^{m-n}2^m}{ |   R_{\mathfrak l} (0, \mathbf{t})|} \int_{ R (0,\widetilde{\gamma}_0^2 \mathbf{t}) }| f\left(x-x'\right) |  dx',}
 \end{split} \end{equation*}  
 by  taking coordinates transformation \eqref{eq:transformation-Theta} and 
 { using}  the  transformation $\mathbb{R}^n\rightarrow  \mathbb{R}^n$ given by
  \begin{equation}\label{eq:coordinates-s-x'}
{    (s_{ 1} ' , \dots , s_{ n} ') } \mapsto x'=\sum_{j=1}^n s_{ j} ' e_{ j} 
  \end{equation} 
  and Proposition \ref{prop:twist-R}.
  The   { Jacobian of this transformation} is $\det(e_{ 1},\dots, e_{ n})\neq0$
since $e_{1},\ldots, e_{n}$ are linearly independent. Similarly,
 \begin{equation*} \begin{split}
   | f*  \chi_{ \mathbf{t} }(x)|&\leq  \frac 1{t_1 \cdots  t_m}\int_{[-\widetilde{\gamma}_0^{-1} t_1,\widetilde{\gamma}_0^{-1} t_1 ]\times\cdots\times [-\widetilde{\gamma}_0^{-1} t_m,\widetilde{\gamma}_0^{-1} t_m 
   ] }
     \left |f\left(x-\sum_{j=1}^n s_{ j} ' e_{ j}\right)\right|ds_1'\cdots  ds_m'
    \\&
   \approx\frac 1{|   R_{\mathfrak l} (0,\mathbf{t})|} \int_{ R_{\mathfrak l}(0,\widetilde{\gamma}_0^{-1}\mathbf{t}) } f\left(x-x'\right)   dx'.
 \end{split} \end{equation*}
 The { required conclusion} follows by applying Proposition \ref{prop:twist-R} .
\end{proof}

\begin{rem} Such equivalence of
    maximal functions is also proved in the flag   case   \cite{CCLLO}  and the flag-like  case   \cite[Proposition 3.2]{WW24}. But the proof given here is   simpler than that in \cite{WW24}.
\end{rem}

\subsection{Boundary growth estimates for   iterated Poisson  integrals and holomorphic $H^p$-functions
}\begin{prop}\label{prop:decay-t}   For $f\in  L^1(\mathbb{R}^n ) $, let $u(x,\mathbf{t}) :=f* P_{ {\mathbf{t}}}$. Then,  for $ b_1,\cdots,  b_m \in \mathbb{N}_0$,
 \begin{equation}\label{eq:L1-f-Pt}\begin{split}
t_1^{ b_1}\cdots t_m^{ b_m}  \left \|\nabla_1^{b_1}\cdots \nabla_m^{b_m}u(\cdot ,\mathbf{t} ) \right\|_ {L^1(\mathbb{R}^n)}& \lesssim \| f  \|_{L^1(\mathbb{R}^n)},
 \\ t_1^{ b_1}\cdots t_m^{ b_m}
\left |  \nabla_1^{b_1}\cdots \nabla_m^{b_m}  u(x,\mathbf{t}) \right|&\lesssim \frac {\| f  \|_{L^1(\mathbb{R}^n)} }{|R(0,\mathbf{ t})|},
\end{split} \end{equation}for $(x,\boldsymbol {t})\in  \mathbb{ R}^n \times(\mathbb{R}_+)^m$.
  \end{prop}
\begin{proof} In the sequel we use notation
 $
     \boldsymbol {\varepsilon }:=(\varepsilon,\ldots,\varepsilon)\in (\mathbb{R}_+)^m
$ for $\varepsilon>0$. If  { we} let $f_{
   \varepsilon  }:=f* P_{    \boldsymbol    \varepsilon  }$, then we can write  $u(x,\mathbf{t}) :=f_{
   \varepsilon  }* P_{  \mathbf{t}- \boldsymbol
   \varepsilon  }$ for small $\varepsilon>0$ and $ f_{
   \varepsilon  }$ is smooth. So we { may assume that  $f$ is} smooth.
For fixed $\mu$, we use coordinates $x=x^\perp+ s e_\mu$ as above. Then,
\begin{equation} \label{eq:f*P-L1} \begin{split}
 \left \| f *_\mu P_{t_\mu}\right\|_{L^1(\mathbb{R}^n)}&=\int_{\mathbb{R}^{n -1}}  \int_{\mathbb{R}  }\left|\int_{\mathbb{R}  }f(x^\perp+(s-s')e_\mu)P_{t_\mu}(s') ds'\right| { ds\, dx^\perp } \\
 &=  \int_{\mathbb{R}^{n -1}} \|f(x^\perp+\cdot e_\mu)*P_{t_\mu}\|_{L^1(\mathbb{R})} dx^\perp\\&
\leq  \int_{\mathbb{R}^{n -1}} \|f(x^\perp+\cdot e_\mu)\|_{L^1(\mathbb{R})} \|P_{t_\mu}\|_{L^1(\mathbb{R})} dx^\perp\\
 & =\int_{\mathbb{R}^{n -1}} \|f(x^\perp+\cdot e_\mu)\| _{L^1(\mathbb{R} )}
 dx^\perp=  \| f  \|_{L^1(\mathbb{R}^n)}.
\end{split} \end{equation}
Applying this equality repeatedly, we see that $\|u(\cdot ,\mathbf{t })   \|_{L^1(\mathbb{R}^n)} \leq \| f  \|_{L^1(\mathbb{R}^n)}$. By replacing $ P_{t_\mu}$ by
$t_\mu^{b_\mu}\widetilde{\nabla}_\mu^{b_\mu}P_{t_\mu}$, the above
argument gives us the first estimate in \eqref{eq:L1-f-Pt}, since we also have $\|t_\mu^{b_\mu}\widetilde{\nabla}_\mu^{b_\mu}P_{t_\mu}\|_{L^1(\mathbb{R})}\lesssim 1$.

By Proposition \ref{prop:triangle-j} (2), $u(x + s_\mu' e_\mu ,\mathbf{t}' ) $ is harmonic on $\{ (s_\mu',t_\mu'); t_\mu'>0\}$, which contains the disc $|s_\mu'+\mathbf{i}(t_\mu-t_\mu')|<
t_\mu$. So we can apply the mean value  formula   to { obtain}
  \begin{equation*}
  u(x,\mathbf{t})=\frac 1{ \pi  t_\mu ^2}\int_{|s_\mu'+\mathbf{i}t_\mu'|<  t_\mu} u\left(x+ s_\mu' e_\mu, t_1,\dots, t_\mu+t_\mu',\dots\right) ds_\mu' d t_\mu' .
 \end{equation*}
 Repeating this procedure, we { deduce}
  \begin{equation} \label{eq:mean-value} \begin{split}
     |u(x ,\mathbf{t}  )|&
     \leq \frac 1{ \pi^m t_1^2\cdots t^2_m}\int_{|s_1'+\mathbf{i}t_1'|< t_1} \cdots\int_{|s_m'+\mathbf{i}t_m'|<  t_m}  \left|u \left (x+\sum_{\mu=1}^m s_\mu' e_\mu
     ,\mathbf{t} +\mathbf{t} '  \right)\right| { ds_m'd t_m'\cdots ds_1' d t_1' }
     \\& \leq \frac 1{ \pi^m t_1^2\cdots t^2_m}\int_{-   t_1}^{   t_1 } \cdots 
     \int_{-   t_m}^{   t_m } \int_{-   t_1}^{   t_1 } \cdots \int_{-  t_m}^{   t_m } 
   \left|u \left (x+\sum_{\mu=1}^m  s_\mu' e_\mu ,\mathbf{t} +\mathbf{t} '  \right)\right|
 {  d t_m'\cdots dt_1' d s_m' \cdots ds_1'.}      
     \end{split}  \end{equation}
     Note that
          \begin{equation*}\begin{split}
      \frac 1{  t_1 \cdots t _m} \int_{-   t_1}^{   t_1 } \cdots \int_{-  t_m}^{   t_m } \left|u \left (x+\sum_{\mu=1}^m s_\mu' e_\mu ,\mathbf{t} +\mathbf{t} '
      \right)\right|
  {   d s_m' \cdots ds_1'}       
      &=\left|u \left ( \cdot ,\mathbf{t} +\mathbf{t} '  \right)\right| *  \chi_{  \mathbf{t} }(x) \\
        & \lesssim \frac 1{| R  (x,\gamma_0^{-1} \mathbf{t})|}\int_{ R(x,\gamma_0^{-1} \mathbf{t})}\left|u \left ( x' ,\mathbf{t} +\mathbf{t} '  \right)\right|dx' \\
        & \lesssim \frac 1{| R  (x,\gamma_0^{-1} \mathbf{t})|}\int_{ \mathbb{R}^n }\left|u \left ( x' ,\mathbf{t} +\mathbf{t} '  \right)\right|dx'\\
        & \lesssim \frac 1{| R  (x,\mathbf{t})|}\| f  \|_{L^1(\mathbb{R}^n)},
   \end{split}   \end{equation*} by applying Proposition \ref{prop:it-twist} and  the first estimate in \eqref{eq:L1-f-Pt}. Substitute this estimate into \eqref {eq:mean-value} to obtain the second inequality in  
   \eqref{eq:L1-f-Pt}   for $b_1=\dots=b_m=0$.

   For the general case, note that each component of $\nabla_1^{b_1}\cdots \nabla_m^{b_m}u (x + s_\mu' e_\mu ,\mathbf{t}' ) $ is still harmonic on $\{ (s_\mu',t_\mu'); t_\mu'>0\}$.
   The result follows from the same argument by using  the first estimate in \eqref{eq:L1-f-Pt}.
\end{proof}

 \begin{cor}\label{cor:conditions} Fix $\varepsilon>0$  and let  $f_{\varepsilon} :=f* P_{\boldsymbol
   \varepsilon  }$ for $f\in  L^1(\mathbb{R}^n ) $. Then,
\\
\textup{(a)} $U(x,t_\mu)=(f_{\varepsilon}*_\mu P_{t_\mu}(x))^2$ satisfies conditions in Proposition \ref{prop:Green}.
 \\
 \textup{(b)} Let  ${\mathfrak j } $ be a subset of ${\mathfrak m } $ and $\mu\notin {\mathfrak j }$. Then for fixed $\mathbf{t}_{\mathfrak j}$,
  $U(x,t_\mu)=f_{\varepsilon}*_{\mathfrak j} \widetilde{\nabla}_{\mathfrak j} P_{\mathbf{t}_{\mathfrak j}}*_\mu \widetilde{\nabla}_\mu P_{t_\mu}(x) $ also satisfies conditions in Proposition \ref{prop:Green} 
  except for $  U(\cdot,t_\mu) \rightarrow U(\cdot,0)$ replaced by $  U(\cdot,t_\mu) \rightarrow 0$.
\end{cor}
\begin{proof} (a) For condition (1), write $U(x,t_\mu)=u^2 (x,t_\mu)$ with $u(x,t_\mu):=  f_{\varepsilon}*_\mu P_{t_\mu}(x)   $. It is obviously smooth on $\mathbb{R}^{n +1}_+$ since $f_{\varepsilon}$ is. Because
$ u (\cdot,t_\mu) ,\nabla_\mu u(\cdot,t_\mu)
\in  L^1  \cap  L^\infty(\mathbb{R}^n ) $ with norms only depending on $\varepsilon$ and $\|f\|_1$ by Proposition \ref{prop:decay-t}, we see that $     U (\cdot,t_\mu)  ,  \nabla_\mu   U (\cdot,t_\mu) =2 u 
(\cdot,t_\mu) \nabla_\mu u(\cdot,t_\mu) \in  L^1  \cap  L^\infty(\mathbb{R}^n ) $. It is similar to see $    \nabla_\mu^2   U (\cdot,t_\mu)  \in
 L^1  \cap  L^\infty(\mathbb{R}^n ) $. It is standard to see that for $x=  x^\perp + s  e_\mu$,
\begin{equation*}
      X_\mu  U(x,t_\mu)= 2 u(x,t_\mu) \cdot \int_{\mathbb{R}  }  f _{\varepsilon}(x^\perp +(s -s')e_\mu)    \partial_{s'} P_{t_\mu}(s') ds' \rightarrow 0,
   \end{equation*} for almost all $x^\perp\in  e_\mu^\perp$ as $ s\rightarrow \infty$, since  $f _{\varepsilon}(x^\perp +\cdot e_\mu)\in L^1(\mathbb{R}) $ for almost all $x^\perp\in  e_\mu^\perp$.

For condition (2), apply    estimates \eqref{eq:L1-f-Pt}  to $u(x,t_\mu) =  f* P_{\boldsymbol
   \varepsilon  }*_\mu P_{t_\mu}(x)  $ to get
\begin{equation*}\begin{split}
   \|t_\mu  \partial_{t_\mu}   U (\cdot,t_\mu)  \|_{L^1(\mathbb{R}^n)}&   \lesssim
     \|   u (\cdot,t_\mu)  \| _{L^ \infty(\mathbb{R}^n)} \|t_\mu \partial_{t_\mu}  u(\cdot,t_\mu)   \| _{L^1(\mathbb{R}^n)}
\lesssim \frac {\| f  \|^2_{L^1(\mathbb{R}^n)}}
 {\varepsilon^{n-1}(\varepsilon+t_\mu)} \rightarrow 0
   \end{split}   \end{equation*}as $t_\mu\rightarrow+\infty$.
  Similarly,  $\|    U (\cdot,t_\mu)  \|_{L^1(\mathbb{R}^n)}\rightarrow0$ as $t_\mu\rightarrow+\infty$.

For condition
(3), note that $u(x,0)= f _{\varepsilon}$. Write   $\triangle(x):=\|f _{\varepsilon}( \cdot  -x) - f _{\varepsilon}(\cdot )\|_{L^1(\mathbb{R}^n)}$. It is well known that
$\triangle(x)\rightarrow 0$ as $|x|\rightarrow0$. Hence,
\begin{equation*}\begin{split}
  \| u(\cdot,t_\mu)-u(\cdot,0)\|_{L^1(\mathbb{R}^n)}&=\int_{\mathbb{R}^{n  }} \left|\int_{\mathbb{R}  }\left[f _{\varepsilon}(x^\perp+(s-s')e_\mu)-f
  _{\varepsilon}(x^\perp+ s e_\mu)\right]P_{t_\mu}(s') ds'\right| { dx^\perp  ds} \\
  &\leq \int_{\mathbb{R}  }\triangle(s'e_\mu)P_{t_\mu}(s') ds'=\int_{\mathbb{R}  }\triangle(t_\mu s'e_\mu)P_{1}(s') ds'\rightarrow 0,
   \end{split}   \end{equation*}
   as $t_\mu\rightarrow 0$, by rescaling and using  Lebegues' dominated convergence theorem (cf. \cite[P. 63]{St70}). We see that
   $\|t_\mu\partial_{t_\mu}   u (\cdot,t_\mu)\|_{L^1(\mathbb{R}^n)} \rightarrow 0$ as $t_\mu\rightarrow 0$ follows by the same argument since
      \begin{equation*}
      t_\mu\partial_{t_\mu}   u (x,t_\mu)= \int_{\mathbb{R}  } \left[ f _{\varepsilon}(x -s' e_\mu)- f _{\varepsilon}(x  ) \right ]  t_\mu\partial_{t_\mu}  P_{t_\mu}(s') ds'
   \end{equation*}by $\int_{\mathbb{R}  }    t_\mu\partial_{t_\mu}  P_{t_\mu}(s') ds'=0$. The result for $U(x,t_\mu)=u^2 (x,t_\mu)$ follows from that $  |u(x,t_\mu)| $ has an upper
   bound only depending on $\varepsilon$.

(b) { We apply the proof in part (a) with}  
 $P_{t_\mu}$ replaced by $\widetilde{\nabla}_\mu P_{t_\mu}$ and $ f_{\varepsilon}$ replaced by $ f_{\varepsilon}*_{\mathfrak j} \widetilde{\nabla}_{\mathfrak j} P_{\mathbf{t}_{\mathfrak j}} (x) $, which is in $
   L^1  \cap  L^\infty (\mathbb{R}^n) $ with a norm  only depending on $\varepsilon$.
\end{proof}

 \begin{prop} \label{prop:upperbound-hardy}  Let $ F \in {H^p( T_\Omega)}$ with $p\geq1$. Then for  $(x,\boldsymbol {r})\in  \mathbb{ R}^n \times(\mathbb{R}_+)^m$,   we have
  \begin{equation}\label{eq:upperbound-hardy}
 | F ( x+\mathbf{i}\pi(\mathbf{r}))|\lesssim \frac {  \|F\|_{H^p( T_\Omega)}  } {|   R  (0,\mathbf{r})|^{ \frac 1p }} .
 \end{equation}
    \end{prop}
 \begin{proof} Given $y\in \Omega$, write $y=\pi(\mathbf{r})=\sum_{\mu=1}^m r_\mu e_\mu\in \Omega$ for some $\mathbf{r}\in (\mathbb{R}_+)^m $. Note that the disc $\{ x+\mathbf{i}y +(s_\mu +\mathbf{i}t_\mu 
 )e_\mu;|s_\mu +\mathbf{i}t_\mu |<  r_\mu \}$ is
 contained in the tube domain $
  T_\Omega$. This is because $r_\mu+t_\mu > r_\mu- r_\mu=0$. So we can apply the mean value formula for this disc in the complex line $\mathbb{C}e_\mu$ to get
  \begin{equation*}
   F ( x+\mathbf{i}y) =\frac 1{ \pi  \gamma_0^2 r_\mu ^2}\int_{|s_\mu  +\mathbf{i}t_\mu |< \gamma_0 r_\mu} F\left ( x+\mathbf{i}y+(s_\mu +\mathbf{i}t_\mu )e_\mu\right) ds_\mu  d t_\mu  .
 \end{equation*}
 Repeating this procedure, we get
  \begin{equation*} \begin{split}
     |F ( x+\mathbf{i}y)|&
     \leq  \int_{|s_1+\mathbf{i}t_1|<\gamma_0  r_1}  \cdots\int_{|s_m+\mathbf{i}t_m|< \gamma_0  r_m} 
     \frac {\left|F\left (
     x+\mathbf{i}y+\sum_{\mu=1}^m(s_\mu+\mathbf{i}t_\mu)e_\mu\right)\right|}{\pi^m \gamma_0^{2m} r_1^2\cdots r^2_m} \, { ds_md t_m\cdots ds_1 d t_1}
     \\& \leq \int_{- \gamma_0   r_1}^{ \gamma_0   r_1 } \cdots \int_{-  \gamma_0  r_m}^{ \gamma_0   r_m } \int_{- \gamma_0   r_1}^{\gamma_0    r_1 } \cdots \int_{- \gamma_0  r_m}^{ \gamma_0   r_m
     } \frac{\left|F\left ( x+\mathbf{i}y+\sum_{\mu=1}^m(s_\mu+\mathbf{i}t_\mu)e_\mu\right)\right|}{\pi^m\gamma_0^{2m}r_1^2\cdots r^2_m}
     {\, dt_m\cdots dt_1ds_m\cdots ds_1}.
\end{split}  \end{equation*}
     Note that
          \begin{equation*}\begin{split}
        \int_{-  \gamma_0  r_1}^{  \gamma_0  r_1 } \cdots \int_{-  \gamma_0 r_m}^{  \gamma_0  r_m } &\frac{\left|F\left ( x+\mathbf{i}y+\sum_{\mu=1}^m s_\mu e_\mu +\mathbf{i}\sum_{\mu=1}^m  t_\mu 
        e_\mu\right)\right|}{\gamma_0^{ m}r_1 \cdots r _m} {\, dt_m\cdots dt_1 }=\left|F\left
        ( x+ \sum_{\mu=1}^m s_\mu e_\mu +\mathbf{i}\cdot\right)\right| *  \chi_{ \gamma_0  \mathbf{r} }(y) \\
        & \lesssim \frac 1{\left| R \left (y,  \mathbf{r}\right)\right|}\int_{ R\left(y,  \mathbf{r}\right)}\left|F \left( x+ \sum_{\mu=1}^m s_\mu e_\mu +\mathbf{i}y'\right)\right|dy',
   \end{split}   \end{equation*} by applying Proposition \ref{prop:it-twist}, where $ R\left(y,  \mathbf{r}\right)\subset\Omega$. Apply the same procedure for the integral over $s_1,\dots , s_m$ to get
   \begin{equation} \label{eq:estimate-cauchy}\begin{split} |F ( x+\mathbf{i}y)|&\lesssim\frac 1{  |R\left(0,    \mathbf{r}\right)|^2}
     \int_{R\left(x,  \mathbf{r}\right)}  \int_{R\left(y,  \mathbf{r}\right)}  \left|F  (  x'+\mathbf{i} y'   )\right|
     { dy' \, dx'}\\
     &     \leq \frac 1{  \left|R\left(0,    \mathbf{r}\right)\right|^{1+\frac 1p }}
\int_{R\left(y,   \mathbf{r}\right)} \left(\int_{\mathbb{R}^n}\left|F  ( x'+\mathbf{i} y'   \right|^p dx'\right)^{\frac 1p}   { dy'} \\
&{ \lesssim \frac 1{  |R(0, \mathbf{r})|^{ \frac 1p }} \|F\|_{H^p( T_\Omega)} .}
      \end{split}  \end{equation}
{ This concludes the proof of \eqref{eq:upperbound-hardy}.}
\end{proof}

 \begin{prop} \label{prop:inclusion}   For  two  polyhedral cones    $\Omega$ and $\Omega'$ in $ \mathbb{R} ^n$ such that
   $\overline{\Omega}\subset\Omega'$, then $  H^1 (T_{\Omega'} )\subsetneqq H^1 (T_\Omega)$ and $ H^1_{ max;\Omega'}(\mathbb{R}^n) \varsubsetneq  H^1_{ max; \Omega}(\mathbb{R}^n) $.
    \end{prop}
 \begin{proof}
 Since $T_\Omega\subset  T_{\Omega'} $,
$  H^1 (T_{\Omega'} )\subset H^1 (T_\Omega)$ { in view of} definition \eqref{eq:H2}. For 
{ an}  $  H^1(T_{\Omega'} ) $-function
$F$, we have $F( \cdot +  \mathbf{i}y )\in   H^2 (T_{\Omega'}) $ for fixed $y  \in  \Omega$,  by the boundary growth estimate \eqref{eq:upperbound-hardy}. So
the  Fourier transformation $\widehat{F( \cdot +  \mathbf{i}y )}$ with respect to variable $x$ must be supported in $(\Omega')^*$ by Theorem 3.1 in \cite{SW71}.

Note that $\overline{(\Omega')^*}\varsubsetneq\Omega^*$ by definition.
Let $\psi$ be a smooth function supported in a bounded subset of $  \Omega ^*$,  but is not contained in $ (\Omega')^*$. Consider a holomorphic function on $ T_{\Omega } $
\begin{equation*}
   F_0(x+ \mathbf{i }y)=\int_{ \Omega ^*}
e^{2\pi \mathbf{i}(x+ \mathbf{i } y) t} \psi(t)  dt.
\end{equation*}
 We see that $  F_0\in H^1 (T_{\Omega } )$ by the estimates: for $|x|<1$, $| F_0(x+ \mathbf{i }y)| \lesssim  1 $,  since $y \cdot t\geq0$ for $y\in \Omega, t\in \Omega ^* $; while for $|x|>1$,
\begin{equation*}\begin{split}
  | F_0(x+ \mathbf{i }y)|&=\left|\int_{\mathbb{R}^n }\frac {\triangle^n e^{2\pi \mathbf{i}(x+ \mathbf{i } y) t} }{(2\pi \mathbf{i}(x+ \mathbf{i } y))^{2n}}
\psi(t)  dt\right|\\
&=\left|\int_{\mathbb{R}^n }\frac {1}{(2\pi \mathbf{i}(x+ \mathbf{i } y))^{2n}}
e^{2\pi \mathbf{i}(x+ \mathbf{i } y) t} \triangle^n\psi(t)  dt\right|
\lesssim \frac {1}{ |x|^{2n}}
,
  \end{split}
\end{equation*} where $\triangle$ is the Laplacian on $\mathbb{R}^n_t$, and implicit constants are independent of $y$. Since the support of $\widehat{F_0( \cdot +  \mathbf{i}y )}(t)=e^{-2\pi  y  t} \psi(t) $ is
not contained in $ (\Omega')^*$, we   have $ F_0\notin H^1 (T_{\Omega '} )$. Hence,
  $  H^1 (T_{\Omega'} )\neq H^1 (T_\Omega)$.

   Since  $F_0( \cdot +  \mathbf{i}y_0 )\in   H^2 (T_\Omega) $ for fixed $y_0\in \Omega $,  by the
   integral representation in Theorem \ref{prop:integral-representation}, we have  $F_0( x +  \mathbf{i}\pi(\mathbf{t})+  \mathbf{i}y_0)= F_0(  \cdot +  \mathbf{i}y_0)*P_{\mathbf{t}}(x)$. So by  Theorem 
   \ref{thm:equivalence}, $F_0( \cdot +  \mathbf{i}y_0)|_{\mathbb{R}^n}$ belongs to $  H^1_{ max;\Omega}(\mathbb{R}^n) $, but  not  to  $H^1_{ max;\Omega ' }(\mathbb{R}^n) $. Otherwise, $F_0( \cdot +  
   \mathbf{i}y_0 )\in    H^1 (T_{\Omega'} ) $ by  Theorem \ref{thm:equivalence}, which contradicts 
   { the fact that} $ F_0\notin H^1 (T_{\Omega '} )$.
\end{proof}

\subsection{Proof of   iterated Poisson  integral formula   and harmonic control
} In the   integral representation  formula \eqref{eq:integral-representation},
for  $\mathbf{t}\neq\mathbf{t}'$ with  $\pi(\mathbf{t})= \pi(\mathbf{t}')$, we must have $ F^b*  P_{\mathbf{t}}= F^b*  P_{\mathbf{t}'}$.
  This  formula follows from the uniqueness of functions  satisfying partial Laplace equations  $
\triangle_\mu u=0$
with the same boundary condition,
by using the  maximum principle for harmonic functions.
 \begin{proof}[Proof of Theorem  \ref{prop:integral-representation}] For
 $  F  \in H^p(T_\Omega)$, we have $f_\varepsilon:=F (  \cdot +\mathbf{i}\pi(\boldsymbol{\varepsilon})) \in L^p(\mathbb{R}^n)$ for $\varepsilon>0$ by definition,  and  it is obviously smooth on $\mathbb{R}^n$. 
 Then
    \begin{equation}\label{eq: f-varepsilon-x'}
  f_{\varepsilon,x^\perp}(\cdot) :=F \left(  x^\perp+ { (\cdot)}  e_\mu +\mathbf{i}\pi(\boldsymbol{\varepsilon})\right) \in L^p(\mathbb{R} )
  \end{equation}
  for almost all $x^\perp\in e_\mu^\perp\subset \mathbb{R}^n$. Now fix such a point $x^\perp\in e_\mu^\perp $ and let
\begin{equation}\label{eq: v-varepsilon-x'}
   v_{\varepsilon,x^\perp}(s_\mu ,t_\mu ):=(f_{\varepsilon,x^\perp}*P_{t_\mu})(  s_\mu )=(f_\varepsilon*_\mu P_{t_\mu})(  x^\perp+s_\mu e_\mu ),
\end{equation}which is smooth on $ \overline{\mathbb{R}_+^2}$. { Note that}   $v_{\varepsilon,x^\perp}(\cdot ,t_\mu )\rightarrow  f_{\varepsilon,x^\perp}$  as $t_\mu \rightarrow 0$, by
the standard
property of Poisson integrals \cite[P. 62]{St70}. Namely,
\begin{equation}\label{eq:t-mu=0}
   v_{\varepsilon,x^\perp}(s_\mu ,0)= f_{\varepsilon,x^\perp}(  s_\mu )= F \left(  x^\perp+s_\mu e_\mu+ \mathbf{i}\pi(\boldsymbol{\varepsilon} )\right).
\end{equation}

  We claim that for any $\delta>0$, there exists $R,b>0$ sufficiently large such that for  $(s_\mu ,  t_\mu)$  outside the rectangle  $\mathcal{R}_{2R,b}:=\{  (s_\mu ,t_\mu )\in\mathbb{R}_+^2; |s_\mu|\leq 2R,  0< 
  t_\mu<b\}$, we have
\begin{equation}\label{eq:claim}
   |v_{\varepsilon,x^\perp}(s_\mu ,t_\mu )|<\delta,\qquad \left | F\left ( x^\perp+(s_\mu+\mathbf{i}t_\mu )e_\mu+ \mathbf{i}\pi(\boldsymbol{\varepsilon} )\right)\right|<\delta.
\end{equation}
Then by the maximal principle \cite{St70} for harmonic functions on the rectangle  $\mathcal{R}_{2R,b}$, we { obtain}
\begin{equation*}| v_{\varepsilon,x^\perp}(s_\mu ,t_\mu )-
  F\left ( x^\perp+(s_\mu+\mathbf{i}t_\mu )e_\mu+ \mathbf{i}\pi(\boldsymbol{\varepsilon} )\right)|\leq2\delta.
\end{equation*}
Since $\delta>0$ { was} arbitrarily chosen, we get
\begin{equation*}
   (f_\varepsilon*_\mu P_{t_\mu})(  x^\perp+s_\mu e_\mu )= F \left( x^\perp+(s_\mu+\mathbf{i}t_\mu )e_\mu+
   \mathbf{i}\pi(\boldsymbol{\varepsilon} )\right),
\end{equation*} for almost all $x^\perp\in e_\mu^\perp$, by \eqref{eq: v-varepsilon-x'}. Consequently,  
we { deduce} 
$
   (f_\varepsilon*_\mu P_{t_\mu} )(  x  )=F (  x +\mathbf{i}t_\mu  e_\mu +\mathbf{i}\pi(\boldsymbol{\varepsilon}))
$ for any $x\in \mathbb{R}^n$, $t_\mu \in \mathbb{R}$, { by continuity.}
   Repeating this procedure, we get
  \begin{equation}\label{eq:reproduce-integral-representation}
   (f_\varepsilon*_1P_{t_1}\cdots *_m P_{t_m})(  x  )={ F (  x+\mathbf{i}\pi(\mathbf{t})+\mathbf{i}\pi(\boldsymbol{\varepsilon}))}
 \end{equation}
 for any $x\in \mathbb{R}^n$, $\mathbf{t}  \in (\mathbb{R}_+)^m$.

 To prove the claim \eqref{eq:claim} of decay, note that by definition, we have 
{ $$\int_{  \mathbb{R}^{n}}\left|F(  x +\mathbf{i}y
    )\right|^pdx
    \leq  \|F\|_{H^p( T_\Omega)}^p
    $$
    } for any $y\in \Omega$, and so for fixed $b>0$,
{$$ 
     \int_{  \mathbb{R}^{n}  }     \int_{ \{y\in\Omega ;|    y|<2b\} }\left|F(  x +\mathbf{i}y
    )\right|^pdx dy<\infty.
  $$}
     Hence, for any $\delta'>0$,  there exists $R>0$ such that
 { $$
       \int_{ |  x|>R  } \int_{\{y\in\Omega ;|    y|<2b\} }\left|F(  x +\mathbf{i}y
    )\right|^pdx dy<\delta'
$$}
    if $R$ is sufficiently large.

On the other hand,   $R(x, \boldsymbol{\varepsilon}) +\mathbf{i}  R(y +\pi( \boldsymbol{\varepsilon}),  \boldsymbol{\varepsilon})\subset T_\Omega $ for any $y\in  \Omega $,
since
 { $$
   R(y +\pi( \boldsymbol{\varepsilon}),  \boldsymbol{\varepsilon})= \Big\{y+\sum_{j=1}^m (\varepsilon+\lambda_j) e_j;|\lambda_j|<\varepsilon,  j=1,\dots, m \Big\}\subset  \Omega.
$$}
So
by using the estimate similar to \eqref{eq:estimate-cauchy} in  Proposition \ref{prop:upperbound-hardy}, for $|  x|>2R, |y|<b$ { we obtain} 
\begin{equation}\label{eq:F-claim}\begin{split}
\left | F\left(x+ \mathbf{i}y +\mathbf{i}\pi (\boldsymbol{\varepsilon})\right)\right|&\lesssim \frac 1{  |R(0,  \boldsymbol{\varepsilon})|^2}
     \int_{R(x,  \boldsymbol{\varepsilon})} \int_{R(y,  \boldsymbol{\varepsilon})} \left|F  (    x'+\mathbf{i}\pi( \boldsymbol{\varepsilon})+\mathbf{i}  y'   )\right| \, d y' { \, d x'} 
     \\&\lesssim \frac 1{  |R(0,  \boldsymbol{\varepsilon})|^{2-\frac 2q}}\left(
     \int_{R(x,  \boldsymbol{\varepsilon})}  \int_{R(y,  \boldsymbol{\varepsilon})} \left|F  (   x'+\mathbf{i}\pi( \boldsymbol{\varepsilon})+\mathbf{i}\hat y   )\right|^p \, d  y '  { \, d x'}  
     \right)^\frac 1p\\
     &\lesssim \frac 1{  |R(0,  \boldsymbol{\varepsilon})|^{2-\frac 2q}}\left(
     \int_{|\hat x|>R }  \int_{\{\hat y\in\Omega ;|     y'|<2b\} }\left|F  (  x'+\mathbf{i}  y'   )\right|^p d  y'\right)^\frac 1p { \, d x'}  <\delta.
\end{split}\end{equation} The claim \eqref{eq:claim} for  $F (x+ \mathbf{i}y +\mathbf{i}\pi (\boldsymbol{\varepsilon}))$ is proved, while
the claim \eqref{eq:claim} for  $v_{\varepsilon,x^\perp}$ holds by the decay of  the Poisson integral of an  $L^p$ function for $p\geq 1$.

   For $   t_\mu>b$ with $b$ large, the claim \eqref{eq:claim} directly follows from the estimate in Proposition \ref{prop:decay-t} and \ref{prop:upperbound-hardy}.

  Since   $ L^p(\mathbb{R}^n)$  for $p> 1$ is reflexive and  $\|f_{ {\varepsilon } } \|_{L^p(\mathbb{R}^n)}\leq \| F \|_{ H^p(T_\Omega)}$, there exists a
subsequence    $  \varepsilon  _k \rightarrow0$   such that $\{f_{ {\varepsilon }_k} \} $ is weakly convergent to some $F^b \in L^p(\mathbb{R}^n)$
by { the} Banach-Alaoglu theorem. For fixed $\mathbf{t}$, we can assume $t_{ 1}, \dots,t_{ n}$ are the largest $n$ numbers among $\{t_{ 1}, \dots,t_{ m}\}$ and  \eqref{eq:e-mu-ej} holds. Then, 
{ one has} 
\begin{equation} \label{eq:F-pi-P}\begin{split}
f_{ {\varepsilon }_k}*P_{\mathbf t}(x)&=\int_{\mathbb{  R}^m} f_{ {\varepsilon }_k}\left(x-\sum_{j=1}^ms_j e_j\right)P_{t_1}(s_1)
\cdots  P_{t_m}(s_m)\, d\mathbf{s}\\&=\int_{\mathbb{  R}^m}  f_{ {\varepsilon }_k}\left(x-\sum_{j=1}^n\left(s_{ j} +\sum_{\mu=n+1}^m  A_{\mu j}s_{\mu}\right)e_{ j}\right)  P_{t_1}(s_1)
\cdots  P_{t_m}(s_m)\, d\mathbf{s}\\
    &=\int_{\mathbb{  R}^m}  f_{ {\varepsilon }_k}\left(x-\sum_{j=1}^n s_{ j}' e_{ j}\right) \prod_{j=1}^n P_{t_j}\left(s_{ j}' -\sum_{\mu=n+1}^m  A_{\mu j}s_{\mu}'\right) \cdot
    \prod_{j=n+1}^m P_{t_{j}}(s_{j}')  \,d\mathbf{s}' \\
    &=\frac 1{\det(e_{ 1},\dots, e_{ n})}\int_{\mathbb{  R}^n} f_{ {\varepsilon }_k}\left(x-x'\right)\pi(P_{t_1} \cdots  P_{t_m})(x')\, dx',
 \end{split} \end{equation} 
{ using the}  coordinates transformations \eqref{eq:transformation-Theta},
where
\begin{equation*}
   \pi(P_{t_1} \cdots  P_{t_m})(x'):=\int_{\mathbb{  R}^{m-n}} \prod_{j=1}^n P_{t_j}\left(s_{ j}' -\sum_{\mu=n+1}^m  A_{\mu j}s_{\mu}'\right) \cdot   \prod_{j=n+1}^m P_{t_{j}}(s_{j}')
\,   ds'_{n+1} \cdots ds'_m
\end{equation*}
for $x'=\sum_{j=1}^n s_{ j}' e_{ j}$. { This function lies} 
 in $L^q(\mathbb{R}^n)$ for any $q>1$, since
\begin{equation*} \begin{split}
 \left\|\pi(P_{t_1} \cdots )\right\|_{L^q}&\approx\left(\int_{\mathbb{  R}^{n}} \left|\int_{\mathbb{  R}^{m-n}} \prod_{j=1}^n P_{t_j}\left(s_{ j}'-\sum_{\mu=n+1}^m  A_{\mu
 j}s_{\mu}'\right)    \prod_{j=n+1}^m P_{t_{j}}(s_{j}')  ds'_{n+1}\cdots ds'_m\right|^qds'_{1} \cdots ds'_n\right)^\frac 1q\\
 &\leq  \int_{\mathbb{  R}^{m-n}}\left(\int_{\mathbb{  R}^{n}} \left|\prod_{j=1}^n P_{t_j}\left(s_{ j}'-\sum_{\mu=n+1}^m  A_{\mu j}s_{\mu}'\right)     \prod_{j=n+1}^m
 P_{t_{j}}(s_{j}')\right|^qds'_{1} \cdots ds'_n\right)^\frac 1q ds'_{n+1} \cdots ds'_m\\
 &= (t_1\cdots t_n)^{\frac 1q-1}\|P_1\|_{L^q(\mathbb{R})}^n { <\infty, }
 \end{split} \end{equation*}
{ in view of the}  coordinates transformation
\eqref{eq:coordinates-s-x'}
and   Minkowski's inequality.

Now by { the weak}  convergence of $\{f_{ {\varepsilon }_k} \} $  to  $F^b \in L^p(\mathbb{R}^n)$, we get
\begin{equation*}
   f_{ {\varepsilon }_k}*\pi(P_{t_1} \cdots  P_{t_m})(x)\rightarrow  F^b*\pi(P_{t_1} \cdots  P_{t_m})(x),
\end{equation*}for any $x \in\mathbb{R}^n $.
Then,  $f_{ {\varepsilon }_k}*P_{\mathbf t}(x)\rightarrow F^b*P_{\mathbf t}(x)$ as $k\rightarrow+\infty $ by \eqref{eq:F-pi-P}. On the other hand,
\begin{equation*}
   f_{ {\varepsilon }_k}*P_{\mathbf t}(x)= F (  x+\mathbf{i}\pi(\mathbf{t})+\mathbf{i}\pi( \boldsymbol{\varepsilon}_k))\rightarrow F (  x+\mathbf{i}\pi(\mathbf{t}) )
\end{equation*}
by \eqref{eq:reproduce-integral-representation}, since $F$ is smooth on $ T_\Omega $. The  integral representation formula \eqref{eq:integral-representation}   follows.
     \end{proof}

\begin{prop}  \label{prop:subharmonic-ineq}  Suppose $F\in H^1( T_\Omega )$.  Then for $0 < q<+ \infty  $, we have
  \begin{equation}\label{eq:subharmonic-ineq}
   | F ( x+\mathbf{i}\pi(\mathbf{t}+\boldsymbol
  { {\varepsilon }}))|^q \leq  |    f_\varepsilon
    |^q *  P_{\mathbf{t}  } (x),
 \end{equation} for { any} 
 $(x,\boldsymbol {t})\in  \mathbb{ R}^n \times(\mathbb{R}_+)^m$, where
  $
   f_\varepsilon ( x  ) :=  F ( x + \mathbf{i}\pi( \boldsymbol
  { {\varepsilon }})).
 $
\end{prop}
 \begin{proof} It is well known that if $h$ is holomorphic on a domain $D$ in $\mathbb{C}$, then $|h|^q$ for $q>0$ is subharmonic on $D$. We see that $|F(x^\perp  +(s_\mu +\mathbf{i} t_\mu)e_\mu+\mathbf{i}\pi 
 (\boldsymbol{\varepsilon}))|^q$ is subharmonic on $\mathbb{R}_+^2$, since $ F(x^\perp  +z e_\mu+\mathbf{i}\pi (\boldsymbol{\varepsilon}))$ is holomorphic in $z\in \mathbb{C}$ for $\operatorname{Im} z>0$ by 
 Proposition \ref{prop:triangle-j}.
{ As in} the proof of Theorem  \ref{prop:integral-representation},  let
 \begin{equation}\label{eq:v-|f|^q}
   v_{x^\perp,\varepsilon}(s_\mu ,t_\mu ):=(|f_\varepsilon|^q*_\mu P_{t_\mu})(  x^\perp+s_\mu e_\mu+ \mathbf{i}\pi(\mathbf{\varepsilon} )),
\end{equation}
 on the upper half plane $\mathbb{R}_+^2$, which   is harmonic by Proposition \ref{prop:triangle-j}. So
  \begin{equation*}
    | F(x^\perp  +(s_\mu +\mathbf{i} t_\mu)e_\mu+\mathbf{i}\pi (\boldsymbol{\varepsilon}))|^q-v_{x^\perp,\varepsilon}(s_\mu ,t_\mu )
 \end{equation*}
   is also subharmonic on $\mathbb{R}_+^2$. Moreover, it vanishes  at $t_\mu=0$, i.e.
  \begin{equation*}v_{x^\perp,\varepsilon}(s_\mu ,0 )=|    f_\varepsilon(x^\perp  + s_\mu  e_\mu )
    |^q =
    | F(x^\perp  + s_\mu  e_\mu+\mathbf{i}\pi (\boldsymbol{\varepsilon}))|^q,
 \end{equation*}by definition \eqref{eq:v-|f|^q}.
 The claim \eqref{eq:claim} holds similarly by \eqref{eq:F-claim} for $p=1$ and $|f_\varepsilon|^q\in  L^{\frac 1q}(\mathbb{R}^n)$ with $\frac 1q>1$. Then, we obtain \begin{equation*}
    | F(x^\perp  +(s_\mu +\mathbf{i} t_\mu)e_\mu+\mathbf{i}\pi (\boldsymbol{\varepsilon}))|^q-v_{x^\perp,\varepsilon}(s_\mu ,t_\mu )\leq0
 \end{equation*} by applying the maximal principle for
 subharmonic functions. The result follows by repeating the procedure as in { the} proof of Theorem  \ref{prop:integral-representation}. { However,} we do not need the last step of taking 
 limit ${\varepsilon }_k\rightarrow0$ here.
 \end{proof}

  \section{The     maximal function
characterization}

 \subsection{Reduction of the Fefferman-Stein type good-$\lambda$ inequality to an integral estimate}

  \begin{thm}\label{thm:Fefferman-Stein-inequality}
 For   $f\in L^1( \mathbb{R}^n)$ and all $\lambda> 0$, if we choose $\beta$   sufficient large,  then the Fefferman-Stein type good-$\lambda$ inequality \eqref{eq:Fefferman-Stein-inequality}  holds.
 \end{thm}

 The following is a direct corollary of this theorem.
 \begin{cor} \label{cor:Fefferman-Stein-inequality} If we choose $\beta$   sufficient large, then  $\|   S (f)
\|_{L^1({  \mathbb{ R}^n})}   \lesssim
     \| N ^\beta( f ) \|_{L^1({  \mathbb{ R}^n})}$.
 \end{cor}

  To prove  the Fefferman-Stein type  good-$\lambda$ inequality
 in Theorem
\ref{thm:Fefferman-Stein-inequality},
  let
  \begin{equation*}
    E_\beta(\lambda):=\left \{x \in \mathbb{ R}^n;   {N}^\beta( f )(x)\leq \lambda\right\},
  \end{equation*}
for $f \in L^1(  \mathbb{ R}^n)$ with $  N ^\beta( f ) \in L^1(  \mathbb{ R}^n)$ and $\lambda\geq0$,  and let
 \begin{equation*}
    A_\beta(\lambda):=\left \{x \in \mathbb{ R}^n;  M_{it}  (  \chi_{ E_\beta(\lambda)^c} )(x)\leq \frac 1{10 C_0}\right\},
 \end{equation*} where $C_0$ is given by \eqref{eq:heat-max}. By definition, $
 E_\beta(\lambda)^c\subset A_\beta(\lambda)^c$, and so
 $
    A_\beta(\lambda) \subset
 E_\beta(\lambda) ,
$ up to a set of measure zero.
   Moreover,
 \begin{equation*}
    \left|A_\beta(\lambda)^c\right|\lesssim \| M_{it}  (  \chi_{ E_\beta(\lambda)^c} )\|_{{ L^2}}^2\leq C\left| E_\beta(\lambda)^c\right|,
 \end{equation*}
 by the $L^2$-boundedness of
 iterated  maximal function, where $C$ is independent of  $\beta$. Therefore, we have
 \begin{equation}\label{eq:S-estimate} \begin{split}
    |\{x \in \mathbb{ R}^n;  S (f)>\lambda \}|& \leq |\{x\in  A_\beta(\lambda)^c;  S (f)(x)>\lambda \}|+ |\{x\in
    A_\beta(\lambda) ;  S (f)(x)>\lambda \}|\\
    &\leq C\left| E_\beta(\lambda)^c\right|+\frac 1{\lambda^2}\int_{ A_\beta(\lambda)}S (f)^2  (x) dx.
 \end{split} \end{equation}
  Consider the domains (see Figure 2)
 \begin{equation}\label{eq:widetilde-W}
    W_\beta:=\bigcup_{x\in A_\beta(\lambda) } \Gamma (x),\qquad {\rm
  and } \qquad
   \widetilde{ W}_\beta:=\bigcup_{x\in E_\beta(\lambda) } \Gamma_\beta (x).
 \end{equation}
 \begin{figure}[h]
  \centering
\includegraphics[scale=0.52]{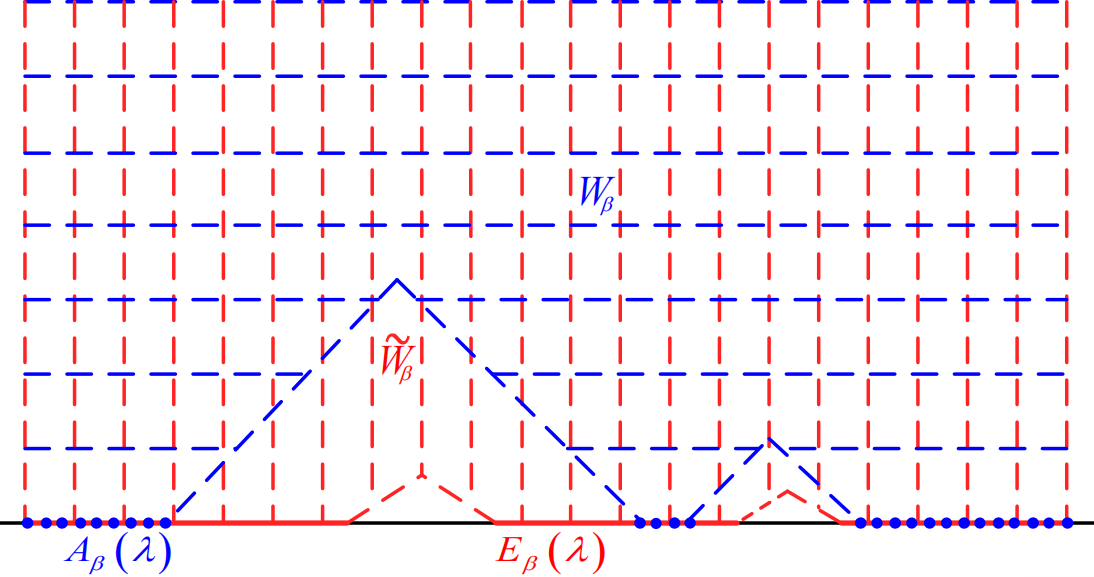}
  \caption{ \color{blue} Domains $ W_\beta$ and  $\widetilde{ W}_\beta$}
  \label{fig:fig2}
\end{figure}

 \begin{prop} \label{prop:W}  \textup{(1)} If $(x,\mathbf{t})\in { W}_\beta$, then
  $
      P_{\mathbf{t}}(\chi_{E_\beta(\lambda) })(x )> \frac 9{10}.
 $
  \\
 \textup{(2)} If we choose $\beta$   sufficient large, then there is a constant $C_1 \in (0,
9/10 )$ such that for any $(x,\mathbf{t})\in\widetilde{ W}_\beta^c=(  \mathbb{ R}^n \times  (\mathbb{R}_+)^m)\setminus \widetilde{ W}_\beta$, we have
    $
       P_{\mathbf{t}}(\chi_{E_\beta(\lambda) })(x )\leq C_1.
 $
 \end{prop}
 \begin{proof} (1) For $(x,\mathbf{t})\in { W}_\beta$, there exists $x'\in  A_\beta(\lambda) $ such that $(x,\mathbf{t})\in\Gamma(x')$. Thus
  \begin{equation*}
     P_{\mathbf{t}}(\chi_{E_\beta(\lambda)^c })(x )\leq C_0M_{it}(\chi_{E_\beta(\lambda)^c })(x' ) \leq C_0\frac 1{10 C_0}=\frac 1{10  },
 \end{equation*}by \eqref{eq:heat-max},
 and so $ P_{\mathbf{t}}(\chi_{E_\beta(\lambda) })(x )= P_{\mathbf{t}}(1-\chi_{E_\beta(\lambda)^c })(x )=1- P_{\mathbf{t}}( \chi_{E_\beta(\lambda)^c })(x )> \frac 9{10}.$

(2)
 Let $(x,\mathbf{t})\in\widetilde{ W}_\beta^c $. Given $\mathbf{s}\in(\mathbb{R}_+)^m$, if   $x'=x- \pi(\mathbf{s})\in E_\beta(\lambda)$,
     we must have $ x\notin
R \left(x',\beta  \mathbf{t}  \right)$ by definition, otherwise, we have $(x,\mathbf{t})\in\widetilde{ W}_\beta  $. So   $  \pi(\mathbf{s})\notin
R \left(0,\beta  \mathbf{t}  \right)$, i.e. $   \mathbf{s} \notin
 \widetilde{R}\left(0,\beta  \mathbf{t}  \right)$.   Consequently,  $|s_\mu| \geq \beta  t_\mu $ holds at least for one
 $\mu$. Thus,
 \begin{equation*}
 \begin{split}
 0\leq  P_{\mathbf{t}}(\chi_{E_\beta(\lambda) })(x )&=\int\chi_{E_\beta(\lambda) }\left(x- \sum_{j=1}^m s_je_j \right )P_{t_1} (s_1 ) \cdots  P_{t_m} (s_m) \, d\mathbf{ s}\\
&\leq \sum_{\mu=1}^m   \int_{|s_\mu| \geq \beta  t_\mu}
 P_{t_\mu} (s_\mu) \, d s_\mu
\prod_{\nu\neq \mu}\int_{ \mathbb{R}}
 P_{t_\nu} (s_\nu)  \,d s_\nu \\
&\leq \sum_{\mu=1}^m   \int_{|s_\mu| \geq \beta  t_\mu}
 P_{t_\mu} (s_\mu) \, d s_\mu
 \leq \frac C \beta\rightarrow 0,
     \end{split} \end{equation*}
     as $\beta\rightarrow+\infty$.
 \end{proof}

Note that the    Littlewood-Paley area function \eqref{eq:Littlewood-Paley-area} of $f$ can be written as
\begin{equation}\label{eq:Littlewood-Paley-area'}
   S (f)(x):=\left(\int_{ \Gamma (x)}\left| f  {{*}}_{1 }\widetilde{\nabla}_1 P_{t_1}  {{*}}_{ 2}\cdots {*  }_m  \widetilde{\nabla}_m P_{t_m} (x')\right|^2\frac
   {\mathbf{t}d\mathbf{t}dx'}{| R(0
,\mathbf{t})|}\right)^{\frac 12},
\end{equation}
for all $x \in \mathbb{R}^n $, by Proposition \ref{prop:triangle-j} (2).
    For $\mathbf{t} \in (\mathbb{R}_+)^m  $ and $x \in \mathbb{ R}^n $, define a function on $\mathbb{ R}^n\times  (\mathbb{R}_+)^m$:
  \begin{equation}\label{eq:u-Hf}
   U_1  (x, \mathbf{t} ):=f  {{*}}_{1 }   P_{t_1}  *_2\widetilde{\nabla}_2 P_{t_2}  *\cdots * _m \widetilde{\nabla}_m P_{t_m} (x ) .
 \end{equation}
Let $\phi$   be a non-negative ${ \mathscr C^\infty(\mathbb{R})}$
functions   such that $\phi(r) = 1$ if $r \geq  \frac 9{10  }$, $ \phi(r) = 0$ if $r \leq C_1$ 
(cf. \cite{MM,Sato}).
 But by Proposition \ref{prop:W} (1),  for $(x,\mathbf{t})\in { W}_\beta$, we have
 $
      P_{\mathbf{t}}(\chi_{E_\beta(\lambda) })(x )> \frac 9{10}$. Therefore, $\phi  (P_{\mathbf{t }} (\chi_{E_\beta(\lambda) })(x ) )=1$ and so
\begin{equation} \label{eq:S-estimate-I0} \begin{split}
    \int_{ A_\beta(\lambda)}S (f)^2  (x) \,dx&= 
    \int_{ A_\beta(\lambda)}\int_{ \Gamma (x)}\left| f  {{*}}_{1 }\widetilde{\nabla}_1
   P_{t_1}  *\cdots * _m \widetilde{\nabla}_m P_{t_m}\right|^2(x')\frac {\mathbf{t}d\mathbf{t}dx'}{|R(0 ,\mathbf{t})|} \,{ dx}\\
&\lesssim\int_{  {W}_\beta} \left| f  {{*}}_{1 }\widetilde{\nabla}_1
   P_{t_1}  *\cdots * _m \widetilde{\nabla}_m P_{t_m}\right|^2(x' ) \,\,\mathbf{t}d\mathbf{t}dx'  \\
   &{ = \int_{ { W}_\beta} \left|  \nabla_1  U _1 ( x,\mathbf{t })\right|^2 \mathbf{t}d\mathbf{t}dx } \\
&\leq \int_{ \mathbb{ R}^n\times ( \mathbb{R}_+ )^{m }  } \left|  \nabla_1   U_1  ( x,
\mathbf{t} )\right|^2\phi^2 (P_{\mathbf{t }} (\chi_{E_\beta(\lambda) } )(x )  \, \mathbf{t} d\mathbf{t}  dx
 .
 \end{split} \end{equation}

   \subsection{Estimates  of   integrals  by     differential identities }
To prove { the} Fefferman-Stein type inequality { of}  Theorem \ref{thm:Fefferman-Stein-inequality}, we need to estimate  { the} integral 
in the R. H. S. of \eqref{eq:S-estimate-I0}  in terms of
integrals only involving $ U _1  $,  instead of gradients of $ U _1 $. This can be done by using  differential identities to integrate by parts. We will at first estimate the integral in   \eqref{eq:S-estimate-I0}
with $\chi_{E_\beta(\lambda) } $ and $f$ replaced by their smoothing $
     \chi_\varepsilon:=P_{\boldsymbol {\varepsilon }} (\chi_{E_\beta(\lambda) } )$ and $ f_\varepsilon:=P_{\boldsymbol {\varepsilon }}(f)$,
  respectively, and then taking { the} limit $\varepsilon\rightarrow0$.

 \begin{lem}\label{lem:estimate0} Let $\varepsilon>0$. Suppose that $f_{\varepsilon} :=f* P_{\boldsymbol \varepsilon}$ for some  real valued function  $f \in  L^1(\mathbb{R}^n ) $, $\chi_{\varepsilon}
 :=\chi * P_{\boldsymbol \varepsilon}$ with $ \|\chi\|_{ L^\infty}\leq 1$ and  $1-\chi\in L^1(\mathbb{R}^n ) $, and
 \begin{equation}\label{eq:u-v}
     u:=f_{\varepsilon} *_\mu P_{t_\mu} , \qquad v:=\chi_{\varepsilon} *_\mu P_{t_\mu}  .
 \end{equation}  Then,
 we have
    \begin{equation}\label{eq:estimate0}
     \int_{\mathbb{R}^n \times  \mathbb{R}_+   } \left|  \nabla_\mu  u  \right|^2\phi^2 (v){t}_\mu d {t}_\mu dx\lesssim \int_{\mathbb{R}^n   }     u^2(x,0)\phi ^2 (v(x,0)) dx+
     \int_{\mathbb{R}^n\times  \mathbb{R}_+   }u^2   \Phi_1^2 (v)
     \left| \nabla_\mu   v \right|^2 {t}_\mu d {t}_\mu dx   ,
    \end{equation}where $\Phi_1$ is a ${ \mathscr C^{2m+2}} [0,1]$ function satisfying $\Phi_1(a)=0$ for $a\in [0,C_1]$ and for $a\in [0, 1]$,
\begin{equation}\label{eq:Phi}
   \Phi_1 (a)\geq  |\phi(a)|+|\phi'(a)|+|\phi''(a)|.
\end{equation}
 \end{lem}
 \begin{proof} By direct differentiation, we see that
   \begin{equation} \label{eq:direct-differentiation}\begin{split}  (X_{\mu  }^{(a)} )^2( u^2\phi ^2 (v))=&2u  (X_{\mu  }^{(a)} )^2  u \cdot\phi ^2 (v)+ 2 |X_{\mu  }^{(a)}u |^2     \phi
   ^2 (v)+ 8 u \phi (v)\phi' (v)
   X_{\mu }^{(a)}u\cdot  X_{\mu  }^{(a)}  v
   \\&+ 2 u^2\phi (v)\phi' (v)   (X_{\mu  }^{(a)} )^2v+ 2 u^2(\phi'(v)^2+ \phi (v)\phi''(v))  | X_{\mu  }^{(a)}v |^2,
 \end{split} \end{equation}for $a=1,2$.
Since $u$ and $v$ satisfy partial Laplacian equation
  \begin{equation}\label{eq:partial-heat}
     \triangle_\mu u=0,\qquad \triangle_\mu v=0,
  \end{equation}
  on $ \mathbb{R}^n \times  \mathbb{R}_+ $,  by Proposition \ref{prop:triangle-j} (2),
the  summation of \eqref{eq:direct-differentiation} over $a=1,2$ gives us
  \begin{equation} \label{eq:S-estimate-F}\begin{split}
     \left|  \nabla_\mu  u  \right|^2\phi^2 (v)= & \frac 12 \triangle_\mu\left(   u^2\phi ^2 (v)\right)
    - 4 u \phi (v)  \phi '(v)  \nabla_\mu u \cdot  \nabla_\mu v
    -    u^2  ( \phi '(v) ^2
    + \phi (v)  \phi'' (v) ) |  \nabla_\mu v|^2\\
= :& T_1 +T_2 +T_3  .
 \end{split} \end{equation}
  But
 \begin{equation*} \begin{split}|T_2 |&\leq \frac 1 {10}\left|  \nabla_\mu  u  \right|^2\phi^2 (v)+ 20u^2  \phi '(v)^2  | \nabla_\mu v |^2 :=T_{2 1} +T_{2 2} .
 \end{split} \end{equation*}The integral $\int_{\mathbb{R}^n \times  \mathbb{R}_+  } T_{2 1}(x, {t}_\mu) 
 {t}_\mu d {t}_\mu dx$
can be absorbed by the integral of the L.H.S in \eqref{eq:S-estimate-F},  while
\begin{equation*}
T_{2 2}  + | T_3 |  \leq 40u^2   \Phi_1 ^2 (v)\left| \nabla_\mu v\right|^2   .
\end{equation*}
Then
\begin{equation*}\begin{split}
       \int_{\mathbb{R}^n \times  \mathbb{R}_+   } \left|  \nabla_\mu  u  \right|^2\phi^2 (v){t}_\mu d {t}_\mu dx&\leq \frac {5}9\int_{\mathbb{R}^n \times  \mathbb{R}_+
       }\triangle_\mu\left(   u^2\phi ^2 (v)\right){t}_\mu d {t}_\mu dx+ \frac {400}9\int_{\mathbb{R}^n \times  \mathbb{R}_+  } u^2   \Phi_1 ^2 (v)\left| \nabla_\mu v\right|^2
  {t}_\mu d {t}_\mu dx.
 \end{split} \end{equation*}
Since  $U=u^2\phi ^2 (v)$ satisfies  assumptions in Proposition \ref{prop:Green} by     Corollary
 \ref{cor:conditions},  we can apply Proposition \ref{prop:Green} to $U= u^2\phi ^2 (v)$  to get
 \begin{equation*} \begin{split} \int_{\mathbb{R}^n \times  \mathbb{R}_+  } \triangle_\mu\left(   u^2\phi ^2 (v)\right){t}_\mu d {t}_\mu dx=& \int_{\mathbb{R}^n    }
 u^2(x,0)\phi ^2 (v(x,0)) dx.
 \end{split} \end{equation*}
The estimate follows.
\end{proof}

If we  apply Lemma \ref{lem:estimate0}
      to    $u=  U_1$ in \eqref{eq:u-Hf} with $f$ replaced by $f_{\varepsilon} $, there appears a term like
      \begin{equation}\label{eq:U2-term}
          \int_{\mathbb{R}^n \times  \mathbb{R}_+   }  \left|  \nabla_2 U_2 \right|^2\Phi_1^2   (v) \left| \nabla_1 v\right|^2 {  {t}_\mu d {t}_\mu dx }.
      \end{equation}Here  for   $k=1,\dots,m$,  we define   functions
\begin{equation}\label{eq:u-Hf-2}
  U_{ k }      (x, \mathbf{t} ):=f_\varepsilon  {{*}}_{1 }   P_{t_1} *\cdots*_k P_{t_k} *_{k+1} \widetilde{\nabla}_{k+1} P_{t_{k+1}} *\cdots * _m \widetilde{\nabla}_m P_{t_m} (x ),
\end{equation}on $\mathbb{ R}^n\times  (\mathbb{R}_+)^m$, which satisfy
\begin{equation}\label{eq:nabla-U-k}
   \nabla_{k }   U_{k }     (x, \mathbf{t} )=U_{k-1 } (x, \mathbf{t} ),
\end{equation}by Proposition \ref{prop:triangle-j} (3). So we need the following
      estimate further to estimate
      the   term \eqref{eq:U2-term}.

 \begin{lem}\label{lem:estimate1} Assume as in Lemma \ref{lem:estimate0}. Let
 \begin{equation*}
     u(x,t_\mu, t_\nu) =f_{\varepsilon}*_\mu P_{t_\mu}*_\nu P_{t_\nu} , \qquad v(x,t_\mu, t_\nu)= \chi_{\varepsilon} *_\mu P_{t_\mu}*_\nu P_{t_\nu}  ,
 \end{equation*}for $\mu\neq\nu$. Then for fixed $t_\mu$,
 we have,
    \begin{equation}\label{eq:estimate-1}\begin{split}
   \left.  \int_{\mathbb{R}^n \times  \mathbb{R}_+  }  \left|  \nabla_\nu {u}  \right|^2 \Phi_1^2  (v)\left| \nabla_\mu v\right|^2\right |_{(x,t_\mu,{t}_\nu)}
   {t}_\nu d {t}_\nu dx\lesssim &
   \left.\int_{\mathbb{R}^n  }   {u}  ^2 \Phi_1^2
     (v ) \left| \nabla_\mu
     v \right|^2\right |_{(x,t_\mu,0)} dx \\& +\left.\int_{\mathbb{R}^n \times  \mathbb{R}_+   }    {u}  ^2\Phi_2 ^2 (v)  \Sigma_{\mu\nu}(v) \right |_{(x,t_\mu,{t}_\nu)}   {t}_\nu d {t}_\nu dx,
  \end{split}  \end{equation}
where
{ $$
    \Sigma_{\mu\nu}(v):=\left| \nabla_\mu   v \right|^2\left| \nabla_ \nu  v \right|^2 + \left| \nabla_\mu \nabla_\nu v \right|^2,
$$}
 and  $\Phi_2$ is a ${ \mathscr C^{2m}}[0,1]$ function satisfying  $\Phi_2(a)=0$ for $a\in [0,C_1]$, and for $a\in [0, 1]$,
{ $$
   \Phi_2 (a)\geq |\Phi_1  (a)| + |\Phi_1 '(a)|+|\Phi_1 ''(a)| .
$$}
 \end{lem}
 \begin{proof}  By direct differentiation similar to \eqref{eq:S-estimate-F},  we have the following identity:
   \begin{equation} \label{eq:differential-identity2}\begin{split}
     \left|  \nabla_\nu {u} \right|^2 \Phi_1^2  (v)\left| \nabla_\mu  v\right|^2= &\frac 12 \triangle_\nu\left(    {u}  ^2 \Phi_1^2  (v) \left| \nabla_\mu  v\right|^2\right)\\&
    - 4 {u}\Phi_1  (v)\Phi_1'  (v) \nabla_\nu  u\cdot   \nabla_\nu   v  \left| \nabla_\mu  v\right|^2- 4 {u}\Phi^2_1  (v) \nabla_\nu  u\cdot  (\nabla_\mu v  \cdot  \nabla_\nu
    \nabla_\mu
    v )\\& -    {u}^2(\Phi_1'  (v)^2 +\Phi_1  (v)\Phi_1''  (v) ) \left| \nabla_\nu   v\right|^2\left|  \nabla_\mu  v\right|^2\\&- 2 {u}^2\Phi_1   (v)\Phi_1'  (v) \nabla_\nu v\cdot
    (\nabla_\mu v  \cdot  \nabla_\nu \nabla_\mu   v )
    -    {u}^2\Phi_1^2  (v)\left| \nabla_\nu\nabla_\mu  v\right|^2
     \\
= &:\tilde{T}_1 +\tilde{T}_2 +\tilde{T}_3 +\tilde{T}_4 +\tilde{T}_5+\tilde{T}_6,
 \end{split} \end{equation}where
$
   \nabla_ \nu  u \cdot  (\nabla_\mu v \cdot  \nabla_\nu \nabla_\mu   v ):=\sum_{a,b=1,2}   X_{\nu }^{(a)} u \cdot
   X_{\mu }^{(b)} v  \cdot X_{\nu  }^{(a)} X_{\mu  }^{(b)} v .
$
 Since   $U=   {u}  ^2 \Phi_1^2  (v) \left| \nabla_\mu  v\right|^2$  satisfies assumptions in Proposition \ref{prop:Green} by the following Lemma
 \ref{prop:conditions-general}, we can
 apply Proposition \ref{prop:Green} to get
\begin{equation*}
   \int_{\mathbb{R}^n \times  \mathbb{R}_+ }\triangle_\nu\left(    {u}  ^2 \Phi_1^2  (v) \left| \nabla_\mu  v\right|^2\right) {t}_\nu d {t}_\nu dx =  \left.\int_{\mathbb{R}^n   }
   {u} ^2 \Phi_1^2  (v) \left| \nabla_\mu  v \right|^2\right|_{(x,t_\mu,0) } dx
   .
\end{equation*}
Apply  Cauchy-Schwarz inequality to
get\begin{equation*}
  \left |  \int_{\mathbb{R}^n \times  \mathbb{R}_+ }\tilde{T}_2    {t}_\nu d {t}_\nu dx\right| \leq\frac 1{10} \int_{\mathbb{R}^n \times  \mathbb{R}_+  }   \left|  \nabla_\nu {u} \right|^2\Phi
  ^2_1 (v)   \left| \nabla_\mu  v \right|^2 {t}_\nu d {t}_\nu dx+C\int_{\mathbb{R}^n \times  \mathbb{R}_+  }    {u}  ^2\Phi ^2_2 (v)  \left|  \nabla_\nu v \right|^2 \left| \nabla_\mu  v
  \right|^2  {t}_\nu d {t}_\nu dx .
\end{equation*}
by   definition   of
  $ \Phi_2$, and similarly
\begin{equation*}
 \left |  \int_{\mathbb{R}^n \times  \mathbb{R}_+}\tilde{T}_3{t}_\nu d {t}_\nu dx\right| \leq  \frac 1{10}\int_{\mathbb{R}^n \times  \mathbb{R}_+ }  \left|  \nabla_\nu {u} \right|^2
 \Phi_1^2  (v)
 \left| \nabla_\mu  v \right|^2
{t}_\nu d {t}_\nu dx + C\int_{\mathbb{R}^n \times  \mathbb{R}_+  }  {u}^2\Phi_1^2  (v) | \nabla_\nu \nabla_\mu   v |^2 {t}_\nu d {t}_\nu dx.
\end{equation*}The first terms in the R. H. S. above are absorbed by the integral of the L.H.S. of \eqref{eq:differential-identity2}.
The   integrals of $\tilde{T}_4,\tilde{T}_5,\tilde{T}_6$  are directly controlled by the second  terms in the R. H. S.   of \eqref{eq:estimate-1}. The Lemma is proved.
 \end{proof}
 \subsection{Estimates in  the general case}
If we  apply Lemma \ref{lem:estimate0}
      to   the term  \eqref{eq:U2-term}, there will  appear  a term like
      \begin{equation}\label{eq:U3-term}
        \int_{\mathbb{R}^n \times  \mathbb{R}_+   }   \left|  \nabla_3 U_3 \right|^2\Phi_2^2 (v)  \Sigma_{12}(v) \, t_{3 }d t_{3 } dx.
      \end{equation} Thus we need the following general estimates inductively.
 For a subset ${\mathfrak j} $ of $\mathfrak  m $, denote
      \begin{equation}\label{eq:Sigma-j}
    \Sigma_{\mathfrak j} (v):=\sum_{{\mathfrak j}_1,  \cdots, {\mathfrak j}_l }\left|\nabla_{{\mathfrak j}_1}v\right|^2  \cdots\left| \nabla_{{\mathfrak j}_l} v \right|^2  ,
    \end{equation}
where the summation is taken over all partitions ${\mathfrak j}_1, \cdots,{\mathfrak j}_l$   of ${\mathfrak j}$, i.e.  ${\mathfrak j}$ is the disjoint union of ${\mathfrak j}_1,
\cdots,{\mathfrak j}_l$.

 \begin{lem}\label{lem:estimate2} Assume as in Lemma \ref{lem:estimate0}.
   For a subset ${\mathfrak j} $ of $\{1,2,\dots,k-1\}$ with $a=|\mathfrak j|$,
 let
 \begin{equation}\label{eq:u-tj-tk}\begin{split}
     u(x, \mathbf{t}_{\mathfrak j},t_k) &=f_{\varepsilon}*_{{\mathfrak j} } P_{\mathbf{t}_{{\mathfrak j} }} *_{k } P_{t_{k }}(x), \qquad
      v(x, \mathbf{t}_{\mathfrak j},t_k) =\chi_{\varepsilon}*_{{\mathfrak j} } P_{\mathbf{t}_{{\mathfrak j} }}*_{k } P_{t_{k }}(x) .
    \end{split} \end{equation}   Then for fixed $\mathbf{t}_{\mathfrak j}\in (\mathbb{R}_+)^a$,
 we have
    \begin{equation}\label{eq:estimate2'}\begin{split}
   \left.  \int_{\mathbb{R}^n \times  \mathbb{R}_+  }  |\nabla_k  {u}|   ^2   \Phi^2 _a(v)  \Sigma_{\mathfrak j} (v)\right|_{(x, \mathbf{t}_{\mathfrak j},t_{k })}t_{k }d t_{k }
   dx\lesssim &
   \left.  \int_{\mathbb{R}^n   }{u}   ^2\Phi^2_a(v ) \Sigma_{\mathfrak j} (v)\right|_{(x, \mathbf{t}_{\mathfrak j},0)}dx\\& + \left. \int_{\mathbb{R}^n \times  \mathbb{R}_+  }
   {u}^2 \Phi_{a+1}(v)^2  \Sigma_{{\mathfrak j} \cup\{k\}} (v)\right|_{(x, \mathbf{t}_{\mathfrak j},t_{k })}t_{k }d t_{k } dx,
    \end{split} \end{equation}
    where  $\Phi_{a+1}$ is a ${ \mathscr C^{2m+2-2a}} [0,1]$ function satisfying
   $
       \Phi_{a+1}|_{[0,C_1]}=0$ and  
       $
          \Phi_{a+1} \geq |\Phi_a|+|\Phi_a'   |+|\Phi_a''  |
$
    on $[0,1]$.
        \end{lem}
 \begin{proof}   Note that $ \frac 12  \triangle_k ( u^2) = |\nabla_k  {u}|   ^2 $ and
 \begin{equation*}\begin{split}
   \frac 12   \triangle_k (     \Phi ^2 _a(v))&= (\Phi _a(v)\Phi _a'' (v)+\Phi _a'(v)^2) \left| \nabla_k v \right|^2 ,\\
     \frac 12 \triangle_k (  \left| \nabla_{{\mathfrak j}_\alpha} v \right|^2 )&= 2\nabla_{{\mathfrak j}_\alpha} v  \cdot\triangle_k \nabla_{{\mathfrak j}_\alpha} v + \left|
     \nabla_k   \nabla_{{\mathfrak j}_\alpha} v  \right|^2
       =   \left| \nabla_k   \nabla_{{\mathfrak j}_\alpha} v   \right|^2,
 \end{split}\end{equation*}
 by  $\triangle_k u=0=\triangle_k v$  and $\triangle_k \nabla_{{\mathfrak j}_\alpha} v= \nabla_{{\mathfrak j}_\alpha}\triangle_k v=0$ for any $ {{\mathfrak j}_\alpha}$ and $k$.
 It follows from direct differentiation that \begin{equation} \label{eq:differential-identity3}\begin{split}
   |\nabla_k  {u}|   ^2   \Phi^2_a(v) \prod_{\beta } \left|\nabla_{{\mathfrak j}_\beta}v\right|^2   = &\frac 12\triangle_k\left(    {u}  ^2\Phi ^2_a(v) \prod_{\beta }
   \left|\nabla_{{\mathfrak j}_\beta}v\right|^2
   \right)  + \breve{T} _2 +\breve{ T}_3
= :  \breve{ T}_1 + \breve{T} _2 +\breve{ T}_3  ,
 \end{split} \end{equation}where
 \begin{equation} \label{eq:differential-identity3''}\begin{split}
   \breve{ T}_2:=& -4u\Phi_a(v) \Phi_a'(v) \nabla_k  {u}\cdot \nabla_k v\prod_{\beta } \left|\nabla_{{\mathfrak j}_\beta}v\right|^2 \\&-4u\Phi_a ^2 (v)\sum_\alpha\nabla_k
   {u}\cdot
 \left (  \nabla_k    \nabla_{{\mathfrak j}_\alpha} v\cdot  \nabla_{{\mathfrak j}_\alpha}v\right)\prod_{\beta\neq\alpha} \left|\nabla_{{\mathfrak j}_\beta}v\right|^2
= :  \breve{ T}_{21}+  \breve{ T}_{22},
 \end{split} \end{equation}are terms   involving first-order  derivatives of $u$,
 and
 \begin{equation} \label{eq:differential-identity3'}\begin{split}
   \breve{ T}_3:=&-   {u}^2
   (\Phi_a(v)\Phi_a''
   (v)+\Phi_a'(v)^2)|\nabla_k  v|^2  \prod_{\beta } \left|\nabla_{{\mathfrak j}_\beta}v\right|^2    \\& - 4  {u}^2
   \Phi_a(v)\Phi_a'(v) \sum_\alpha\nabla_k     v\cdot ( \nabla_k   \nabla_{{\mathfrak j}_\alpha} v  \cdot    \nabla_{{\mathfrak j}_\alpha} v ) \prod_{\beta\neq\alpha}
   \left|\nabla_{{\mathfrak j}_\beta}v\right|^2   \\&
   - {u}  ^2\Phi_a^2 (v) \sum_\alpha\left| \nabla_k   \nabla_{{\mathfrak j}_\alpha} v    \right|^2\prod_{\beta\neq\alpha}
   \left|\nabla_{{\mathfrak j}_\beta}v\right|^2  \\&
   - 4{u}  ^2\Phi_a^2  (v)\sum_{ \alpha \neq\gamma}  (\nabla_k   \nabla_{{\mathfrak j}_\alpha} v \cdot  \nabla_{{\mathfrak j}_\alpha} v)\cdot (\nabla_k   \nabla_{{\mathfrak j}_\gamma} v
   \cdot    \nabla_{{\mathfrak j}_\gamma} v)
   \prod_{\beta\neq\alpha,\gamma} \left|\nabla_{{\mathfrak j}_\beta}v\right|^2 \\
= &:   \breve{ T}_{31}+  \breve{ T}_{32}+ \breve{ T}_{33}+ \breve{ T}_{34} ,
 \end{split} \end{equation}
 are terms only involving derivatives of $v$.

 Since   $U=  {u}  ^2\Phi ^2_a(v) \prod_{\beta } \left|\nabla_{{\mathfrak j}_\beta}v\right|^2  $  satisfies assumptions in Proposition \ref{prop:Green} by the { ensuing} Lemma
 \ref{prop:conditions-general}, we can apply Proposition \ref{prop:Green} to get
 the integral of $ \breve{ T}_1 $ is controlled by the first term in  R.H.S. of \eqref{eq:estimate2'}. Similarly, by applying Cauchy-Schwarz inequality, we have
  \begin{equation}\label{eq:differential-identity6}\begin{split}
 \int_{\mathbb{R}^n \times  \mathbb{R}_+   }|\breve{ T}_{21}(x, \mathbf{t}_{\mathfrak j},t_k)|t_{k }\, d t_{k } dx &\leq \frac 1{10} \int_{\mathbb{R}^n \times  \mathbb{R}_+ }|\nabla_k
 {u}|   ^2
 \Phi ^2_a(v)\prod_{\beta } \left|\nabla_{{\mathfrak j}_\beta}v\right|^2t_{k }d t_{k } dx\\& \qquad +C
  \int_{\mathbb{R}^n \times  \mathbb{R}_+  }  {u}  ^2\Phi^2_{a+1}(v) |\nabla_k v|^2 \prod_{\beta } \left|\nabla_{{\mathfrak j}_\beta}v\right|^2t_{k }d t_{k } dx
,\\
 \int_{\mathbb{R}^n \times  \mathbb{R}_+  } |\breve{ T}_{22}(x, \mathbf{t}_{\mathfrak j},t_k)|t_{k }d t_{k } dx&\leq\frac 1{10} \int_{\mathbb{R}^n \times  \mathbb{R}_+ }|\nabla_k
 {u}|   ^2
 \Phi_a ^2 (v)\prod_{\beta } \left|\nabla_{{\mathfrak j}_\beta}v\right|^2t_{k }\, d t_{k } dx\\&\qquad +C
  \int_{\mathbb{R}^n \times  \mathbb{R}_+}   {u}  ^2\Phi_a^2  (v)\sum_\alpha\left| \nabla_k   \nabla_{{\mathfrak j}_\alpha} v   \right|^2\prod_{\beta\neq\alpha}
  \left|\nabla_{{\mathfrak j}_\beta}v\right|^2t_{k }d t_{k } dx .
 \end{split}\end{equation}
The first terms in  the  R.H.S. of \eqref{eq:differential-identity6} can be
absorbed by the integral of the  L.H.S. of  \eqref{eq:differential-identity3}.  The second terms in  the  R.H.S. of \eqref{eq:differential-identity6} can be controlled by the second terms in the   R.H.S. of
\eqref{eq:estimate2'}, by definition of  $\Sigma_{{\mathfrak j} \cup\{k\}} $.

The integrals of $\breve{ T}_{3j}$'s   are controlled directly by the second term in  R.H.S. of \eqref{eq:estimate2'} by
applying Cauchy-Schwarz inequality. The estimate \eqref{eq:estimate2'} follows.
\end{proof}

 \begin{lem}\label{prop:conditions-general}   Let $\varepsilon>0$ and let $u,v $ be given  as in Lemma \ref{lem:estimate2}. Then
  $U=  {u}  ^2\Phi ^2_a(v) \prod_{\beta } \left|\nabla_{{\mathfrak j}_\beta}v\right|^2  $  satisfies conditions in Proposition \ref{prop:Green}.
\end{lem}
\begin{proof} Although there  may be $\chi_{\varepsilon} :=\chi * P_{\boldsymbol \varepsilon}  \notin  L^1(\mathbb{R}^n)$, we must have  $\nabla_l\chi_{\varepsilon} =(1-  \chi ) *
\widetilde{\nabla}_l P_{\boldsymbol \varepsilon}\in  L^1(\mathbb{R}^n) $. But $   v \in     L^\infty (\mathbb{R}^n) $. Thus by the boundary growth estimate in Proposition \ref{prop:upperbound-hardy}, we see that 
for any nonempty subset $\mathfrak {l } $ of $\mathfrak {m } $,  $\left|\nabla_{\mathfrak {l } }  v\right|\in  L^1  \cap  L^\infty (\mathbb{R}^n) $ with norms only depending on $\varepsilon$.
  Note that
    \begin{equation}\label{eq:conditions-Proposition}
    {u}  ^2  ,  \left|\nabla_{{\mathfrak j}_\beta}v\right|^2\hskip 2mm  {\rm satisfy \hskip 2mm  conditions\hskip 2mm  in\hskip 2mm  Proposition} \hskip 2mm \ref{prop:Green},
  \end{equation}except for $  \left|\nabla_{{\mathfrak j}_\beta}v\right|^2(\cdot,t_k)\rightarrow 0$ as $t_k\rightarrow 0 $,
 by Corollary \ref{cor:conditions} (1)-(2) for $\mu=k$, and
\begin{equation*} \label{eq:nabla-U} \begin{split}
   \nabla_k U=&\Big(\nabla_k {u}^2 \Phi ^2_a(v)   + 2{u}  ^2\Phi  _a(v)\Phi '_a(v) \nabla_k v \Big)\prod_{\beta }
   \left|\nabla_{{\mathfrak j}_\beta}v\right|^2
   +\sum_\alpha {u}  ^2\Phi ^2_a(v)\nabla_k(\left|\nabla_{{\mathfrak j}_\alpha}v\right|^2)\prod_{\beta\neq\alpha }\left|\nabla_{{\mathfrak j}_\beta}v\right|^2.
 \end{split}\end{equation*}
 Then $\nabla_k U(\cdot,t_k) \in L^1(\mathbb{R}^n)$ because $\Phi ,\Phi '$ are bounded, and all factors in the R.H.S. above is both in $L^1\cap L^\infty(\mathbb{R}^n)$ by \eqref{eq:conditions-Proposition}.   
 Similarly,  $\nabla^a U(\cdot,t_k) \in L^1(\mathbb{R}^n)$ for $a=0,2$.

{ Note that} 
 $ \lim_{t_k\rightarrow+\infty} \|t_k^a\partial_{t_k}^{a} U(\cdot,t_k)\|_{L^1(\mathbb{R}^n)}= 0$ for $a=0,1$, because $t_\mu^a\partial_{t_k}^{a}({u}  ^2)$,
   $t_k^a\partial_{t_k}^{a} (\left|\nabla_{{\mathfrak j}_\alpha}v\right|^2)\rightarrow 0 $ in $ L^1(\mathbb{R}^n)$  by \eqref{eq:conditions-Proposition}, and other
 factors are all bounded by a constant only depending on $\varepsilon$.

 By   \eqref{eq:conditions-Proposition} again, as $t_k\rightarrow 0 $, ${u}  ^2(\cdot,t_k) \rightarrow  {u}  ^2(\cdot,0)  $,  $\left|\nabla_{{\mathfrak j}_\beta}v\right|^2 (\cdot,t_k) \rightarrow
 \left|\nabla_{{\mathfrak j}_\beta}v\right|^2   (\cdot,0)  $ in $L^1(\mathbb{R}^n)$,  and $\Phi ^2_a(v)  (\cdot,t_k)$ $\rightarrow  \Phi ^2_a(v)  (\cdot,0)$ a.e. Consequently,  $
 U(\cdot,t_k) \rightarrow U(\cdot,0)$. Similarly,  $t_k\partial_{t_k}   U (\cdot,t_k) \rightarrow 0$ in $L^1(\mathbb{R}^n)$ as $t_k\rightarrow 0 $.
\end{proof}
  \subsection{Proof of the Fefferman-Stein type good-$\lambda$ inequality  }

   \begin{prop}\label{prop:estimate-Fefferman-Stein0} Let $\varepsilon>0$. Suppose that $f_{\varepsilon} :=f* P_{\boldsymbol \varepsilon}$ for some   $f \in  L^1(\mathbb{R}^n ) $,
   $\chi_{\varepsilon} :=\chi_E* P_{\boldsymbol \varepsilon}$ with $E^c$ having bounded measure, and
 $   v(x,\mathbf{t} )=
    \chi_\varepsilon {{*}}  P_{\mathbf{t }}.
    $ Then, we have the estimate
    \begin{equation} \label{eq:S-estimate-m} \begin{split}
    I&=\int_{  \mathbb{ R}^n\times (\mathbb{R}_+ )^{m }  } \left|   \nabla_\mathfrak {m}( f_\varepsilon*P_{\mathbf{t} })  \right|^2\phi^2 (v  )  \mathbf{t} d\mathbf{t}  dx   \lesssim   \left.   \sum_{ \mathfrak j 
    }  \int_{\mathbb{ R}^n\times
(\mathbb{R}_+)^{ |\mathfrak j |}} P_{\mathbf{t}_{\mathfrak j}  }(f_\varepsilon)  ^2   \Phi ^2_{| \mathfrak j |}(v )
    \Sigma_{{\mathfrak j}  } (v)\right|_{(x,{\mathbf{  0 }}_{{\mathfrak j} ^c} ,{\mathbf{  { t}}}_{\mathfrak j} )} {\mathbf{  { t}}}_{\mathfrak j}   d{\mathbf{ { t}}}_{\mathfrak j}
    dx,
\end{split} \end{equation}
where the summation is taken over subsets ${\mathfrak j}  $   of $\mathfrak {m}$, $\Phi_0=\phi$,
    $\Sigma_{\mathfrak j}  (v)$ is given by \eqref {eq:Sigma-j},  and $
 {\mathfrak j}  ^c:=\mathfrak {m}\setminus {\mathfrak j}  $.
  \end{prop}
  \begin{proof}
     Applying estimate \eqref{eq:estimate0} in Lemma \ref{lem:estimate0}
      to   real and imaginary parts of $    U_1  ( x,t_1, \mathbf{t }_{{\mathfrak m}_2})  $ given by \eqref{eq:u-Hf-2}-\eqref{eq:nabla-U-k} for fixed $\mathbf{t }_{{\mathfrak m}_2}$,
      we get the integral $I$ in \eqref{eq:S-estimate-m}  can be estimated as
      \begin{equation} \label{eq:S-estimate-I} \begin{split}
    I=&\int_{   (\mathbb{R}_+ )^{m-1}  } 
    \int_{ \mathbb{ R}^n\times  \mathbb{R}_+   } \left|  \nabla_1   U_1   ( x,t_1,
    \mathbf{t }_{{\mathfrak m}_2})\right|^2\phi^2 (v )t_1d {t_1} dx \, 
    { \mathbf{t}_{{\mathfrak m}_2}d\mathbf{t}_{{\mathfrak m}_2}}\\
\lesssim &\left.
 \int_{   (\mathbb{R}_+ ) ^{m-1}  } 
 \left(  \int_{\mathbb{R}^n\times  \mathbb{R}_+   }|U_1| ^2   \Phi_1^2  (v) \left| \nabla_1
 v\right|^2\right|_{( x, \mathbf{t} )}   {t}_1  d {t}_1 dx   +\left.\int_{\mathbb{R}^n   }    | U_1|^2 \phi
 ^2 (v )\right|_{( x,0,\mathbf{t}_{{\mathfrak m}_2})} dx\right)\, 
  { \mathbf{t}_{{\mathfrak m}_2}d\mathbf{t}_{{\mathfrak m}_2}} \\
=&\left.\int_{\mathbb{ R}^n\times  (\mathbb{R}_+)^m } \left|  \  U_{ 1 }  \right|^2 \Phi_1^2  (v) \left| \nabla_1  v \right|^2\right|_{( x, \mathbf{t} )}
\mathbf{t}d\mathbf{t}dx+\left.\int_{\mathbb{ R}^n\times (\mathbb{R}_+ )^{m-1}  } \left|    U_{ 1 }\right|^2\phi ^2(v)\right|_{( x,0,\mathbf{t}_{{\mathfrak m}_2})}\, 
\mathbf{t}_{{\mathfrak m}_2}d\mathbf{t}_{{\mathfrak m}_2}  dx.
 \end{split} \end{equation}

We claim that
  \begin{equation} \label{eq:S-estimate-k} \begin{split}
    I\lesssim  &  \sum_{a=0}^{k-1}\sum_{|{\mathfrak j} |={ a }} S_{\mathfrak j} ^{(k)}, \end{split} \end{equation}
    where the summation is taken over all subsets ${\mathfrak j}  $   of $\{1,2,\dots,k-1\}$, and
\begin{equation} \label{eq:S-inductive}
 S_{\mathfrak j} ^{(k)}:=  \left.\int_{\mathbb{ R}^n\times  \mathbb{R}_+^{m-k+1+|\mathfrak j|} } \left|      U_{k-1 } \right|^2 \Phi ^2_{|{\mathfrak j} |}(v ) \Sigma_{{\mathfrak j}  }
 (v)\right|_{(x,\mathbf{0}_{{\mathfrak j} _k^c}, {\mathbf{  {t}}}_{{\mathfrak j} \cup {\mathfrak m}_k})} {\mathbf{  t}}_{{\mathfrak j} \cup {\mathfrak m}_k} d{\mathbf{
 t}}_{{\mathfrak j} \cup {\mathfrak m}_k} dx,
\end{equation}
  $k=2,\dots, m+1$. Here $\Sigma_{{\mathfrak j}  }
 (v)=1$ if ${\mathfrak j}=\emptyset$,  $U_{ k } $ is given by \eqref{eq:u-Hf-2},
$    U_{k -1} = \nabla_k   U_{k } $, ${\mathfrak m}_k =\{k,\dots,m\}$, and
  \begin{equation*}
     {\mathfrak j} _k^c:=\{1,2,\dots,k-1\}\setminus {\mathfrak j}   .
  \end{equation*}

 Let us prove the claim \eqref{eq:S-estimate-k} inductively. \eqref{eq:S-estimate-I} implies that the claim
\eqref{eq:S-estimate-k} already holds for $k=2$.
Assume that \eqref{eq:S-estimate-k} holds for positive integer $k\leq m+1$.
Let
\begin{equation*}\begin{split}
   \hat   u(x, \mathbf{t}_{{\mathfrak j}  },t_k) &=  \hat f_\varepsilon*_{{\mathfrak j} } P_{\mathbf{t}_{\mathfrak j} } *_{k } P_{t_{k }}(x),  \qquad
   \hat    v(x, \mathbf{t}_{{\mathfrak j}  },t_k) =  \hat \chi_\varepsilon *_{{\mathfrak j} } P_{\mathbf{t}_{\mathfrak j} }*_{k } P_{t_{k }}(x) ,
    \end{split} \end{equation*}for fixed  $\mathbf{t}_{{\mathfrak m}_{k+1}} $,
where
 \begin{equation*}\begin{split}
      \hat f (x )= f    *_{k+1} \widetilde{\nabla}_{k+1} P_{t_{k+1}} *\cdots * _m \widetilde{\nabla}_m P_{t_m}(x ), \qquad
       \hat \chi(x ) &=\chi   *_{k+1}  P_{t_{k+1}} *\cdots * _m   P_{t_m} (x ).
    \end{split} \end{equation*}Note that   $\hat f\in   L^1(\mathbb{R}^n )$ by applying \eqref{eq:f*P-L1} repeatedly, and similarly,  $ \|\hat\chi\|_{ L^\infty}\leq 1$ and  $1-\hat\chi\in L^1(\mathbb{R}^n ) $.  
    Moreover their $L^1$-norms are independent of $\mathbf{t}$. Then, by definition,
        \begin{equation*}
        U_{k }|_{(x,\mathbf{0}_{{\mathfrak j} _k^c}, {\mathbf{  {t}}}_{{\mathfrak j} \cup {\mathfrak m}_k})}=\hat   u(x, \mathbf{t}_{{\mathfrak j}  },t_k), \qquad v
        |_{(x,\mathbf{0}_{{\mathfrak j} _k^c}, {\mathbf{  {t}}}_{{\mathfrak j} \cup {\mathfrak m}_k})}=\hat   v(x, \mathbf{t}_{{\mathfrak j}  },t_k)
    \end{equation*}for fixed  $\mathbf{t}_{{\mathfrak m}_{k+1}} $.  Applying estimate \eqref{eq:estimate2'} in Lemma \ref{lem:estimate2} to  $u= \hat   u$ and $ v=\hat   v$ for fixed $ {\mathbf{  t}}_{{\mathfrak 
    j}  } $  and ${\mathbf{  t}}_{  {\mathfrak m}_{k+1}} $
 to get
        \begin{equation}\label{eq:S-estimate-I1} \begin{split}
       S_{\mathfrak j} ^{(k)}= &  \left. \int_{   (\mathbb{R}_+)^{m-k+|\mathfrak j| } } 
    \int_{\mathbb{ R}^n\times  \mathbb{R}_+  } \left|  \nabla_k   \hat u   \right|^2 \Phi ^2_{|\mathfrak j|}( \hat v ) \Sigma_{{\mathfrak j}   } (\hat v )  \right|_{(x,
    \mathbf{t}_{{\mathfrak j}  },t_k) } t _k dt_k  dx \, 
        {   {\mathbf{  t}}_{{\mathfrak j} \cup {\mathfrak m}_{k+1}}  d{\mathbf{  t}}_{{\mathfrak j} \cup
       {\mathfrak m}_{k+1}}  } \\
 \lesssim&
\int_{   (\mathbb{R}_+)^{m-k+|\mathfrak j| } } 
\Bigg(\int_{  \mathbb{ R}^n} \left|  \hat   u  \right|^2 \Phi ^2_{|\mathfrak j|}(\hat v ) \Sigma_{{\mathfrak j} } ( \hat v) |_{ (x,
 \mathbf{t}_{{\mathfrak j}
 },0)}
 dx 
 \\ &\qquad  \qquad \qquad\qquad +\left. 
 \int_{\mathbb{
R}^n\times
\mathbb{R}_+  } \left|  \hat   u   \right|^2 \Phi ^2_{|\mathfrak j|+1}( \hat v ) \Sigma_{{\mathfrak j}  \cup\{ k\} } ( \hat v )\right|_{(x, \mathbf{t}_{{\mathfrak j} },t_k) }  t _k
dt_k
dx { \Bigg)  }  {   {\mathbf{  t}}_{{\mathfrak j} \cup {\mathfrak m}_{k+1}}  d{\mathbf{  t}}_{{\mathfrak j} \cup {\mathfrak m}_{k+1}}  }
\\
=&\left.\int_{ (  \mathbb{R}_+)^{m-k+|\mathfrak j| } } \left|    U_{k  }    \right|^2 \Phi ^2_{|\mathfrak j|}(v ) \Sigma_{{\mathfrak j} } (v)\right|_{(x,
  \mathbf{0}_{{\mathfrak j} ^c_{k+1}},
  {\mathbf{  t}}_{{\mathfrak j} \cup {\mathfrak m}_{k+1}})}  {\mathbf{  t}}_{{\mathfrak j} \cup {\mathfrak m}_{k+1}} d{\mathbf{  t}}_{{\mathfrak j} \cup {\mathfrak m}_{k+1}}   dx
\\ & \qquad+\left.\int_{\mathbb{ R}^n\times  (\mathbb{R}_+)^{m-k+|\mathfrak j|+1} } \left|     U_{k  }    \right|^2
\Phi ^2_{|\mathfrak j|+1}(v )\Sigma_{{\mathfrak j} \cup\{ k\} } (v)\right|_{(x,\mathbf{0}_{{\mathfrak j} _k^c},  \mathbf{  t}_{{\mathfrak j} \cup {\mathfrak m}_{k }})}
{\mathbf{  t}}_{{\mathfrak j} \cup {\mathfrak m}_{k }} d{\mathbf{  t}}_{{\mathfrak j} \cup {\mathfrak m}_{k }}  dx  
\\
=&S_{\mathfrak j} ^{(k+1)} + S_{{\mathfrak j} \cup\{ k\}}^{(k+1)}.
 \end{split} \end{equation}It follows that
  \begin{equation*}
    \sum_{a=0}^{k-1}\sum_{|{\mathfrak j} |={ a }} S_{\mathfrak j} ^{(k)}\lesssim \sum_{a=0}^{k }\sum_{|  {\mathfrak l} |={ a }} S_{  {\mathfrak l} }^{(k+1)}
 \end{equation*}
 by definition, where the summation in the right hand side is taken over subsets $ {\mathfrak l}  $   of $\{1,2,\dots,k \}$. The R.H.S. of the estimate \eqref{eq:S-estimate-m}
 is exactly \eqref{eq:S-estimate-k} for $k=m+1$, since $U_m|_{(x,\mathbf{0}_{{\mathfrak j} ^c},  \mathbf{  t}_ {\mathfrak j}  )} = P_{\mathbf{t}_{\mathfrak j}  }(f)$ by definition \eqref{eq:u-Hf-2}.
      \end{proof}

 Integrals in Proposition \ref{prop:estimate-Fefferman-Stein0} can be further estimated as   follows.
  \begin{lem}\label{prop:estimate-u2-G-ki} Assume as in Proposition \ref{prop:estimate-Fefferman-Stein0}. Then for a nonempty  subset  ${\mathfrak j}  $   of $\mathfrak {m}$,
   $E= E_\beta(\lambda) $, and $\lambda>0$, we have
     \begin{equation}\label{eq:estimate-u2-G-ki}
  \left.  \int_{\mathbb{ R}^n\times  (\mathbb{R}_+ )^{|{\mathfrak j} | } }  P_{\mathbf{t}_{\mathfrak j}  }(f_\varepsilon)  ^2   \Phi_{|{\mathfrak j} | }(v)  ^2 \Sigma_{\mathfrak j}
  (v)\right|_{(x,\mathbf{0}_{{\mathfrak j} ^c}, \mathbf{t}_{{\mathfrak j}  })}\mathbf{t}_{{\mathfrak j}  }d\mathbf{t}_{{\mathfrak j}  }dx
\lesssim\lambda^2 | E_\beta(\lambda)^c|.
     \end{equation}
     \end{lem}
  \begin{proof}
   For $ \Phi_a(P_{{\mathbf{t}_{\mathfrak j}  }} (\chi_\varepsilon ))(x)\neq0$, we have
  \begin{equation*}
   P_{\mathbf{t}_{\mathfrak j}  +  \boldsymbol {\varepsilon }} (\chi  )(x)=  P_{\mathbf{t}_{\mathfrak j}  } (\chi_\varepsilon )(x)>C_1
   \end{equation*}by definition of $\Phi_a$. Here by abusing of notations, we denote $( \mathbf{0}_{{\mathfrak j} ^c}, \mathbf{t}_{{\mathfrak j}  })$ by $\mathbf{t}_{{\mathfrak j}  }$
   briefly.
Consequently,  $(x,\mathbf{t}_{\mathfrak j}  +  \boldsymbol {\varepsilon }) \in\widetilde{ W}_\beta$ by Proposition
 \ref{prop:W}. Hence, there exists $x'\in E_\beta(\lambda)$ such that $(x,\mathbf{t}_{\mathfrak j}  +  \boldsymbol {\varepsilon })\in \Gamma_\beta(x')$, and so
 \begin{equation*}
 P_{\mathbf{t}_{\mathfrak j}   } (f_\varepsilon)(x)=   P_{\mathbf{t}_{\mathfrak j}  +  \boldsymbol {\varepsilon }} (f)(x)\leq  {N}^\beta( f ) (x')\leq\lambda.
 \end{equation*}Therefore,
 \begin{equation*}\begin{split}
  \left. \int_{\mathbb{ R}^n\times  (\mathbb{R}_+ )^{|\mathfrak j|} }  P_{\mathbf{t}_{\mathfrak j}  }(f_\varepsilon)  ^2   \Phi  ^2   _{|\mathfrak j|}(v ) \Sigma_{\mathfrak j}
   (v )\right|_{(x,\mathbf{t}_{{\mathfrak j}  })}\mathbf{t}_{\mathfrak j} d \mathbf{t}_{\mathfrak j}  dx
\,\leq&\, C  \left.\sum_{{\mathfrak j} _1, \cdots,{\mathfrak j} _l}\lambda^2\int_{\mathbb{ R}^n\times  (\mathbb{R}_+)^{|\mathfrak j|} }      \left|\nabla_{{\mathfrak j} _1}v \right|^2
\cdots\left|
\nabla_{{\mathfrak j} _l} v  \right|^2\right|_{(x,\mathbf{t}_{{\mathfrak j}  })} \mathbf{t}_{\mathfrak j}  d\mathbf{t}_{\mathfrak j}  dx
  \end{split}\end{equation*}    
  where $C=\max_a\max_{[0,1]} \Phi_a^2$,   the summation is taken over partitions ${\mathfrak j} _1, \cdots,{\mathfrak j} _l$  of ${\mathfrak j} $,   and
    \begin{equation*}
     v (x,\mathbf{t}_{{\mathfrak j}  })= P_{\mathbf{t}_{{\mathfrak j}  } }  (\chi_\varepsilon  )(x)= P_{\mathbf{t}_{{\mathfrak j}  } +  \boldsymbol {\varepsilon }}  (\chi
     )(x).
  \end{equation*}

Recall that
$\left |  P_{t_\lambda}(h)\right|\lesssim M_\lambda(h),
   \left |t _\lambda\nabla_ \lambda  P_{t_\lambda}(h)\right|\lesssim M_\lambda(h)
$ for $h\in L^1(\mathbb{R})$.
   This together with
      \begin{equation*}
       \nabla_{{\mathfrak j} _\alpha}v (x,\mathbf{t}_{{\mathfrak j}  })=\chi_\varepsilon  *_{{\mathfrak j}_1  } P_{\mathbf{t}_{{\mathfrak j}_1  } }\cdots*_{{\mathfrak
       j}_\alpha  } \widetilde{\nabla}_{{\mathfrak j} _\alpha}P_{\mathbf{t}_{{\mathfrak j}_\alpha  } }*\cdots *_{{\mathfrak j}_l  } P_{\mathbf{t}_{{\mathfrak j}_l  } }
       = {\nabla}_{{\mathfrak j} _\alpha}\left(\chi_\varepsilon   *_{{\mathfrak j}_\alpha  }P_{\mathbf{t}_{{\mathfrak j}_\alpha  } }\right)
       *_{\mathbf{t}_{ {\mathfrak j}\setminus {\mathfrak j} _\alpha} }
       P_{\mathbf{t}_{ {\mathfrak j}\setminus {\mathfrak j} _\alpha}},
   \end{equation*}by the  commutativity \eqref{eq:convolution-commutativity} of  partial  convolutions
and \eqref{eq:heat-max},  implies
      \begin{equation*}
      \left|\mathbf{t}_{{\mathfrak j} _\alpha }\nabla_{{\mathfrak j} _\alpha}v (x,\mathbf{t}_{{\mathfrak j}  })\right|^2\lesssim \left|M_{ {\mathfrak j}\setminus {\mathfrak j} _\alpha
      }\left(\mathbf{t}_{{\mathfrak j} _\alpha}\nabla_ {{\mathfrak j} _\alpha}P
      _{\mathbf{t}_{{\mathfrak j} _\alpha}}(\chi_\varepsilon)\right) \right|^2,
   \end{equation*}where    $ {\mathfrak j}\setminus {\mathfrak j} _\alpha$ is    the complement of ${\mathfrak j} _\alpha$ in ${\mathfrak j} $, and $M_{ {\mathfrak j}\setminus {\mathfrak j} 
   _\alpha}=M_{\lambda_a}\circ
   \cdots \circ M_{\lambda_1}$ is an iterated maximal function if $ {\mathfrak j}\setminus {\mathfrak j} _\alpha=\{\lambda_1, \cdots ,\lambda_a\}$. For a subset $  \mathfrak j=\{ {j_1},\dots, j_a\}$ of $\mathfrak 
   m $, denote
$ \frac{ d\mathbf{t}_{\mathfrak j}} {\mathbf{t}_{\mathfrak j}}:=\frac{d t_{j_1} \cdots d t_{j_a}}{  t_{j_1} \cdots  t_{j_a}}.
$
 Then,
 \begin{equation*}\begin{split}
\int_{\mathbb{ R}^n\times  (\mathbb{R}_+) ^{|\mathfrak j|} }      \left|\nabla_{{\mathfrak j} _1}v \right|^2  \cdots\left| \nabla_{{\mathfrak j} _l} v  \right|^2
\mathbf{t}_{\mathfrak j}  d\mathbf{t}_{\mathfrak j}  dx \lesssim& \int_{\mathbb{ R}^n\times  (\mathbb{R}_+) ^{|\mathfrak j|} }  \prod_{\alpha=1}^l
     \left|M_{ {\mathfrak j}\setminus {\mathfrak j} _\alpha}\left(\mathbf{t}_{{\mathfrak j} _\alpha}\nabla_ {{\mathfrak j} _\alpha}P _{\mathbf{t}_{{\mathfrak j} _\alpha}}(\chi_\varepsilon)\right)
     \right|^2 dx\frac
    {d\mathbf{t}_{\mathfrak j} }{\mathbf{t}_{\mathfrak j} }\\=&
    \int_{\mathbb{ R}^n }  \prod_{\alpha=1}^l\int_{   (\mathbb{R}_+ )^{|{\mathfrak j} _\alpha|} }
     \left|M_{ {\mathfrak j}\setminus {\mathfrak j} _\alpha}\left(\mathbf{t}_{{\mathfrak j} _\alpha}\nabla_ {{\mathfrak j} _\alpha}P _{\mathbf{t}_{{\mathfrak j} _\alpha}}(\chi_\varepsilon)\right)
     \right|^2 \frac {d\mathbf{t}_{{\mathfrak j} _\alpha}}{\mathbf{t}_{{\mathfrak j} _\alpha}} { dx} \\
\leq & \prod_{\alpha=1}^l \left|\int_{\mathbb{ R}^n }\left(\int_{  ( \mathbb{R}_+ )^{|{\mathfrak j} _\alpha|} }
     \left|M_{ {\mathfrak j}\setminus {\mathfrak j} _\alpha}\left(\mathbf{t}_{{\mathfrak j} _\alpha}\nabla_ {{\mathfrak j} _\alpha}P _{\mathbf{t}_{{\mathfrak j} _\alpha}}(1-\chi_\varepsilon)\right)
     \right|^2 \frac {d\mathbf{t}_{{\mathfrak j} _\alpha}}{\mathbf{t}_{{\mathfrak j} _\alpha}}\right)^ldx\right|^{\frac 1l}\\\lesssim & \prod_{\alpha=1}^l \left|\int_{\mathbb{ R}^n
     } g_{{\mathfrak j} _\alpha   }^{2l} (1-\chi_\varepsilon)dx\right|^{\frac 1l}\\
\lesssim &  \|1-\chi_{\varepsilon}\|_{L^ {2l}(\mathbb{R}^n)} ^{2l}\rightarrow  \|1-\chi \|_{L^ {2l}(\mathbb{R}^n)} ^{2l}= | E_\beta(\lambda)^c|.
 \end{split}\end{equation*}by using H\"older inequality, the vector-valued inequality for the Hardy-Littlewood
maximal functions,
 and  the $L^{2l}$ boundedness of  the partial Littlewood-Paley  $g$-function in Proposition \ref{thm:g-function} .
 \end{proof}

   \begin{proof}[Proof of Theorem \ref{thm:Fefferman-Stein-inequality}] Since $f_\varepsilon=f*P_{\boldsymbol \varepsilon}$ is smooth on $\overline{T_\Omega}$, we see that
 $f_\varepsilon(x)  \leq  {N}^\beta(f_\varepsilon)(x )\leq  {N}^\beta(f )(x )
$ by the definition of   nontangential maximal function,
  and  similarly, $v(x,\mathbf{0})=(\chi_{E_\beta(\lambda)})_\varepsilon$. Hence,
 \begin{equation}\label{eq:Fefferman-Stein-inequality-varepsilon} \begin{split}
I=& \int_{  \mathbb{ R}^n\times (\mathbb{R}_+ )^{m }  } \left|   \nabla_\mathfrak {m}( f_\varepsilon*P_{\mathbf{t} })  \right|^2\phi^2 (v  )  \mathbf{t} d\mathbf{t}  dx    \\\lesssim&  \int_{\mathbb{ R}^n }  
f_\varepsilon  ^2   \phi ^2(v )  (x)    dx+\left.  \sum_{a=1}^{m}\sum_{|{\mathfrak j} |={ a }}
   \int_{\mathbb{
   R}^n\times ( \mathbb{R}_+)^{ a} } P_{\mathbf{t}_{\mathfrak j}  }(f_\varepsilon)  ^2   \Phi ^2_{a}(v ) \Sigma_{{\mathfrak j}  } (v)\right|_{(x,{\mathbf{  0 }}_{{\mathfrak j} ^c}
   ,{\mathbf{  {
   t}}}_{\mathfrak j} )} {\mathbf{  { t}}}_{\mathfrak j}   d{\mathbf{ { t}}}_{\mathfrak j}    dx\\\lesssim &\int_{\mathbb{ R}^n } N^\beta( f )  ^2 \phi ^2\left((\chi_{E_\beta(\lambda)})_\varepsilon \right)  (x)       
   dx+
\lambda^2 | E_\beta(\lambda)^c|,
 \end{split} \end{equation}
by   Proposition \ref{prop:estimate-Fefferman-Stein0} and Lemma \ref{prop:estimate-u2-G-ki}. As $\varepsilon\rightarrow0$, $  \nabla_\mathfrak {m}( f_\varepsilon*P_{\mathbf{t} })\rightarrow  \nabla_\mathfrak {m}( 
f *P_{\mathbf{t} })  $ and $    v( \cdot,\mathbf{t} )\rightarrow
    \chi  {{*}}  P_{\mathbf{t }}$ a.e.  The resulting inequality \eqref{eq:Fefferman-Stein-inequality} follows from the inequality \eqref{eq:S-estimate-I0} and the inequality by  taking limit 
    $\varepsilon\rightarrow0$ in \eqref{eq:Fefferman-Stein-inequality-varepsilon} and using Fatou's lemma and Lebesgues' dominated convergence theorem.
  \end{proof}

  \section{Littlewood-Paley $g$ function and $S$   functions}
  \subsection{The Plancherel-P\'olya type inequality for $\mathfrak{H} $-valued  functions}
  Let   $ \mathfrak  H $ be a separable Hilbert space and let
$\{h_1,h_2,\dots\}$ be an orthonormal basis of $\mathfrak  H $.
A $\mathfrak{ {H}} $-valued function  $\mathfrak f   $ called {\it measurable} if  $\mathfrak f  =\sum_{j=1}^\infty f_j(x)h_j$ converges a.e. with each $f_j $ measurable.
 A measurable function $ \mathfrak f \in   L
^p
(\mathbb{ R},\mathfrak{H})$ if $\|\mathfrak f \|_{L
^p
(\mathbb{ R},\mathfrak{H})}:=(\int_{\mathbb{ R}}|\mathfrak f(x)|_{\mathfrak{H}}^pdx)^\frac 1p<\infty$.  In particular, $\int_{\mathbb{ R}}|\mathfrak f(x)|^2_{\mathfrak{H}}dx=\sum_{j=1}^\infty\int
_{\mathbb{ R}}| f_j(x)|^2dx$.

  Fix $ 0 < \beta < 1$ and $\gamma  > 0$, we say that
  $\mathfrak f$
  defined on  $\mathbb{  R}^1$
  belongs
to $ \mathcal{ {M} }  (\beta, \gamma, r , x_0;  \mathfrak {H})$ with $r > 0$ and  $ x_0\in \mathbb{  R}^1$ if  it satisfies $\int_{\mathbb{  R}^1}\mathfrak f( x  )dx=0$ and
\begin{equation}\label{eq:test} \begin{split}
      |\mathfrak f( x  )|_{\mathfrak{H}}&\leq C\frac { r ^\gamma}{(r +|x-x_0|)^{1+\gamma}},\\
       |\mathfrak  f(  x)- \mathfrak  f(  x' )|_{\mathfrak{H}}&\leq C \left(\frac { |x-x'|}{ r +|x-x_0| }\right )^\beta\frac { r ^\gamma}{(r +|x-x_0|)^{1+\gamma}},
 \end{split}  \end{equation}for $ |x-x'|\leq \frac {r +|x-x_0|}2$. Denote
$
    \|\mathfrak   f\|_{\mathcal{ {M} }  (\beta, \gamma, r , x_0; \mathfrak {H})}:=\inf\{C;\eqref{eq:test}\hskip 2mm {\rm hold}\}.
 $ It is a Banach space as in the scalar case. In particular, $\mathcal{ {M} }  (\beta, \gamma, r , x_0 )=\mathcal{ {M} }  (\beta, \gamma, r , x_0; \mathbb{R})$.

Denote  $\rho_t:=t\frac {\partial P_t}{\partial t}$. { Recall   (cf. e.g. \cite[Section 2.3]{HLLW})
that there exists  function $\varphi\in \mathcal{S}(\mathbb{R})$ such that \\
(1) supp $\varphi\subseteq [-1,1]$; \\
(2) $\int_{\mathbb{R}} s^a
\varphi(s) ds=0$ for $a= 0,1,2$;\\
(3) $\int_{0}^{+\infty} e^{-\xi} \widehat{\varphi}(\xi) \, d\xi=1$.}

 Then we have
the Calder\'on
reproducing formula for $\mathfrak{H} $-valued functions on $\mathbb{R}$:
\begin{equation}\label{eq:reproducing}  \mathfrak f( x  )=\lim_{\substack{T\rightarrow+\infty  \\\varepsilon\rightarrow+0}}\int_{\varepsilon}^T\varphi_ { t }*\rho_ { {t}}* \mathfrak
f( x  )\, \frac { d {t}}{ {t}}=
\int_{0}^{+\infty}\varphi_ { t }*\rho_ { {t}}* \mathfrak f( x  )\, \frac { d {t}}{ {t}},
 \end{equation}for $ \mathfrak f  \in L
^2
(\mathbb{ R},\mathfrak{H})$. This is because $\mathfrak f =\sum_i f_ih_i$, 
{ and} each entry $f_i\in L^2
(\mathbb{ R} )$   { satisfies} the usual Calder\'on
reproducing formula on $\mathbb{R}$. Here $\phi*\mathfrak f =\sum_i \phi*f_ih_i$ for $\phi\in L^1(\mathbb{R})$ decaying like the Poisson kernel. $\phi* \mathfrak f   \in   L
^p
(\mathbb{ R},\mathfrak{H})$ if $ \mathfrak f   \in   L
^p
(\mathbb{ R},\mathfrak{H})$. This is because
\begin{equation*}\begin{split}
   \|\phi*\mathfrak f  \|_{L
^p
(\mathbb{ R},\mathfrak{H})}&=\left(\int_{\mathbb{ R}}\left(\sum_i |\phi*f_i |^2( x  )\right) ^\frac p2 dx\right)^\frac 1p \lesssim \left(\int_{\mathbb{ R}}\left(\sum_i |M(f_i )|^2( x  )\right) ^\frac p2 
dx\right)^\frac 1p\\& \lesssim \left(\int_{\mathbb{ R}}\left(\sum_i | f_i  |^2( x  )\right) ^\frac p2 dx\right)^\frac 1p=\| \mathfrak f  \|_{L
^p
(\mathbb{ R},\mathfrak{H})}, 
    \end{split} \end{equation*}
by the vector-valued inequality for the Hardy-Littlewood
maximal functions.

We have  the wavelet Calder\'on reproducing formula for $\mathfrak{H} $-valued  functions. See the appendix for its proof.
In the sequel, for $j\in\mathbb{ Z}$, a {\it  dyadic interval $I$ with side-length  $\ell (I) = 2^{-\alpha(j+N)}$} means  $ I   =(0, 2^{-\alpha( j  +N) }]+l2^{-\alpha( j  +N) } $ for  some $l\in\mathbb{ Z}$.
   \begin{prop}  \label{prop:discrete-Calderon}
  There exist a fixed small $\alpha>0$  and a large integer
$N$ such that for each dyadic interval  $I$ with side-length  $\ell (I) = 2^{-\alpha(j+N)}$, $j\in\mathbb{ Z}$,  and    any fixed point  $x_I$   in $I $,  there exists $\phi_{j}(x , x_I)  \in\mathcal{ M} (\beta, 
\gamma,  2^{-\alpha j},
 x_I;\mathfrak{H})$ satisfying
\begin{equation} \label{eq:wavelet-Calderon} \begin{split}
 \mathfrak f (x)=\sum_{j\in \mathbb{Z}} \sum_{I   }   \mathfrak
 f_{j;I   } \phi_{j  } (  \cdot, x_I ), 
        \end{split}   \end{equation} 
        where the summation is taken over dyadic intervals with side-length  $\ell (I) = 2^{-\alpha(j+N)}$,  and
        \begin{equation}\label{eq:widetilde-rho} \mathfrak
               f_{j;I    } = c_\alpha \ell(I) |\widetilde{ \rho}_ {{j}}* \mathfrak f(x_I) |_{\mathfrak{H}},\qquad \widetilde{\rho}_ {j}:=\frac 1{c_\alpha }\int_{2^{-\alpha j }}^{2^{-\alpha(j -1)  }}  \rho_ { t
         } \frac { d {t}}{ {t}},\qquad  c_\alpha =\int_{2^{-\alpha j }}^{2^{-\alpha(j-1)  }}   \frac { d {t}}{ {t}}=\alpha \ln    2 .
        \end{equation} The series in \eqref{eq:wavelet-Calderon} converges   in  $  L^2(\mathbb{R},\mathfrak{H})$
 and in   the Banach space $ {M}   (\beta, \gamma, r, x_0;\mathfrak{H})  $.
 \end{prop}

 We also need  the Calder\'on
reproducing formula with compactly supported $\psi$:
  \begin{equation*}
    \mathfrak   f(x)=
\int_{0}^{+\infty}\psi_ {t}*\psi_ {t}*\mathfrak  f(x)\frac { dt}{t},
  \end{equation*} for any   $ \mathfrak f  \in L
^2
(\mathbb{ R},\mathfrak{H})$. Here we can require   $\psi $ to be smooth and compactly supported. The $S$-function associated  $\psi $ and  $g$-function associated to  the Poisson kernel for $\mathfrak{H} $-valued 
functions on $\mathbb{R}$ are
defined as
\begin{equation*} \begin{split}
   S_\psi(\mathfrak  f)(x): =\left(\int_{0}^{+\infty}\int_{|x-x'|<t}\left|\psi_t*\mathfrak  f  (x' )  \right|_{  \mathfrak  {H}}^2 \frac
   {dt  dx'} t\right)^{\frac 12} \quad {\rm and } \quad g_P(\mathfrak  f)(x):=\left(\int_{0}^{+\infty}\left|\rho_t*\mathfrak  f \right|_{   \mathfrak H }^2 (x  ) \frac
   {dt  } t\right)^{\frac 12},
     \end{split}\end{equation*}respectively, where $\psi_t(x)=\frac 1t \psi(x/t)$.  See e.g. \cite{HLS} for vector valued Littlewood-Paley $g$ function and $S$   functions.
  The wavelet Calder\'on
reproducing formula  \eqref{prop:discrete-Calderon} yields the following Plancherel-P\'olya type inequality  (cf. e.g. \cite[Theorem 2.16-2.17]{HLLW} for the scalar case).
\begin{prop}\label{thm:Plancherel-Polya} Let $\alpha>0$  and
$N$ as in Proposition \ref{prop:discrete-Calderon}.   For a fixed $C_0>0$ and any  $\mathfrak f \in L^2(\mathbb{R},\mathfrak{H})$,
    \begin{equation*}  \begin{split}
     \left\|\left(\sum_{j\in \mathbb{Z}}\sum_I   \int_{2^{-\alpha j}}^{2^{-\alpha (j-1)}}
     \!\!\!\! \sup_{u\in C_0I} |\psi_t* \mathfrak  f(u )|_{ \mathfrak   {H}}^2
    \frac  {dt}{  t}\chi_I( \cdot)\right)^{\frac
    12}\right \|_{{ L^1}}
     \lesssim
     \left\|\left(\sum_{j\in \mathbb{Z}}\sum_I \inf_{u\in I}  \int_{2^{-\alpha j}}^{2^{-\alpha (j-1)}} 
     \!\!\!\! \left|\rho_t* \mathfrak  f(u )\right|_{ \mathfrak   {H}}^2
     \frac     {dt}{  t}\chi_I( \cdot)\right)^{\frac
      12}\right\|_{{ L^1}},
        \end{split}   \end{equation*} where the summation is taken over all  dyadic intervals $I$ with side-lengths $\ell (I) =2^{-\alpha( j+N)}$,  and $\chi_I$ is the indicator function  of $I$. The implicit 
        constant   depends   on $\alpha, N$ and $C_0$.
\end{prop}
This inequality { actually holds  for certain   distributions}, but here the inequality for $L^2$ functions is sufficient for our { purposes}.
See the appendix for { a} proof.

 \subsection{The estimate  $ \| f \|_{{ L^1}}   \lesssim  \|  g (f) \|_{{ L^1}} $} At first, we use the   Plancherel-P\'olya type inequality  to prove the control the $S$-function by 
 $g$-function for $\mathfrak{H} $-valued functions.

\begin{prop}\label{prop:S<g} For    $\mathfrak f \in L^2(\mathbb{R},\mathfrak{H})$ with    $g_P(\mathfrak f) \in L^1(\mathbb{R} )$, we have $S_\psi(\mathfrak f) \in L^1(\mathbb{R} )$  and $ \|S_\psi(\mathfrak 
f)\|_{L^1(\mathbb{R} )}   \lesssim \|g_P(\mathfrak f)\|_{L^1(\mathbb{R} )} $.
  \end{prop}
  \begin{proof}  We can write
  \begin{equation*}  \begin{split}
     S_\psi(\mathfrak  f)(x)= \left(\sum_{j\in \mathbb{Z}}\sum_I   \int_{2^{-\alpha j}}^{2^{-\alpha (j-1)}}\int_{\mathbb{ R} }\chi_t(x-x')  |\psi_t* \mathfrak  f(x' )|_{
     \mathfrak  {H}}^2\chi_I(x)\frac
     {dx'dt}{  t}\right)^{\frac 12}  ,
        \end{split}   \end{equation*}where the summation is taken over all  dyadic  intervals $I$ with side-lengths $\ell (I) =2^{-\alpha( j+N)}$.  Note that there exists a fixed constant $C_0$ depending only on 
        $\alpha, N$ such that for $ 2^{-\alpha j} \leq  t  \leq
2^{-\alpha (j-1)}$ and $ x' \in\mathbb{ R} $,
\begin{equation*}
   \chi_t(x-x')  |\psi_t* \mathfrak  f(x' )|_{  \mathfrak {H}}\chi_I(x)\leq \chi_t(x-x')\sup_{u\in C_0I} |\psi_t*\mathfrak   f(u )|_{ \mathfrak   {H}}\chi_I(x).
\end{equation*}
Therefore, { we have}
 \begin{equation*}  \begin{split}
     \|S_\psi(\mathfrak  f)\|_{L^1(\mathbb{R})}&\leq \left\|\left(\sum_{j\in \mathbb{Z}}\sum_I   \int_{2^{-\alpha j}}^{2^{-\alpha (j-1)}}\int_{\mathbb{ R} }\chi_t(\cdot-x')\sup_{u\in C_0I} |\psi_t*
     \mathfrak  f(u )|_{  \mathfrak {H}}^2\chi_I(\cdot)\frac {dx'dt}{
     t}\right)^{\frac 12}\right \|_{L^1(\mathbb{R})}
     \\&=\sqrt 2 \left\|\left(\sum_{j\in \mathbb{Z}}\sum_I   \int_{2^{-\alpha j}}^{2^{-\alpha (j-1)}}  \sup_{u\in C_0I} |\psi_t* \mathfrak  f(u )|_{  \mathfrak {H}}^2\chi_I(\cdot)\frac { dt}{
     t}\right)^{\frac 12}\right \|_{L^1(\mathbb{R})}
     \\&
     \lesssim \left\|\left(\sum_{j\in \mathbb{Z}}\sum_I  \inf_{u\in I}  \int_{2^{-\alpha j}}^{2^{-\alpha (j-1)}}  \left|\rho_t* \mathfrak  f(u )\right|_{ \mathfrak
     {H}}^2\frac { dt}{
     t}\chi_I(\cdot)\right)^{\frac 12}\right \|_{L^1(\mathbb{R})} \\&
     \leq \left\|\left(\sum_{j\in \mathbb{Z}}\sum_I     \int_{2^{-\alpha j}}^{2^{-\alpha (j-1)}}  \left|\rho_t* \mathfrak  f(\cdot )\right|_{ \mathfrak   {H}}^2\frac { dt}{
     t}\chi_I(\cdot)\right)^{\frac 12}\right \|_{L^1(\mathbb{R})}
   \\& =\left\| \left  (   \int_{0}^{+\infty}  \left|\rho_t* \mathfrak  f(\cdot )\right|_{ \mathfrak   {H}}^2\frac { dt}{
     t} \right)^{\frac 12}\right \|_{L^1(\mathbb{R})}=\|g_P(\mathfrak f)\|_{L^1(\mathbb{R} )},
        \end{split}   \end{equation*}
by the Plancherel-P\'olya type inequality in Proposition \ref{thm:Plancherel-Polya}.
 \end{proof}

\begin{prop}\label{prop:g=L1} For any $\mathfrak f\in L ^2(\mathbb{R}, \mathfrak { H})$ with $S_\psi(\mathfrak f) \in L^1(\mathbb{R} )$, we have $\mathfrak f\in L ^1(\mathbb{R}, \mathfrak { H})$  and  $ 
\|\mathfrak  f \|_{L^1(\mathbb{R}, \mathfrak { H})}   \lesssim \|
S_\psi(\mathfrak f)\|_{L^1(\mathbb{R})} $.
  \end{prop}
 \begin{proof} It follows from atomic decomposition for
$ \mathfrak { H}$-valued $L ^2$ functions with $L^1$ integrable   $S$-functions. This   can be established in a  standard way exactly   as the scalar case (cf. e.g. \cite{DLWY,WW24}). We omit details.
 \end{proof}
 Proposition \ref{prop:S<g}  and \ref{prop:g=L1} imply the following corollary.

 \begin{cor}\label{cor:L1<g}
    For    $\mathfrak f \in L^2(\mathbb{R},\mathfrak{H})$ with    $g_P(\mathfrak f) \in L^1(\mathbb{R} )$,   we have $\mathfrak f\in L ^1(\mathbb{R}, \mathfrak { H})$  and $ \|\mathfrak  f \|_{L^1(\mathbb{R}, 
    \mathfrak { H})}   \lesssim \| g_P(\mathfrak
    f)\|_{L^1(\mathbb{R})} $.
 \end{cor}
\begin{prop}  \label{prop:g-L1} For    $  f \in L^2(\mathbb{R}^n)$ with    $g (  f) \in L^1(\mathbb{R}^n)$, we have $  f\in L ^1(\mathbb{R}^n )$  and  $ \| f \|_{L^1(\mathbb{ R}^n)}\lesssim  \|  g (f) 
\|_{L^1(\mathbb{ R}^n)} $.
\end{prop}
 \begin{proof}
Let  $\mathfrak {H}:=L
^2
\big((\mathbb{R}_+)^{m-1},\frac {d\mathbf{t}_{{\mathfrak m}_2}} {\mathbf{t}_{{\mathfrak m}_2}} \big)$ with the usual $L^2$-norm. It   is a separable Hilbert space.
Given $  f \in L^2(\mathbb{ R}^n)$ with $  g(f) \in L^1(\mathbb{ R}^n)$,
define measurable  functions
\begin{equation*}
   \mathfrak  f_{{x^\perp }}(s)(t_2,{ \dots} ,{t}_m ):=f     *_2{t}_2\partial_{t_2} P_{t_2}  *\cdots * _m {t}_m\partial_{t_ m} P_{t_m} ({x^\perp }+s  e_1).
\end{equation*}For almost all ${x^\perp }\in e_1^\perp$,  $f_{{x^\perp }} $ is $\mathfrak{H} $-valued because
  \begin{equation}\label{eq:g-m2}\begin{split}
  \int_{  \mathbb{ R}^{n-1} } 
 \int_{  \mathbb{ R}  }   | \mathfrak  f_{{x^\perp }}(s)|^2_{\mathfrak { H}}ds{ dx^\perp}&= \int_{  \mathbb{ R}^{n } } 
\int_{(\mathbb{R}_+)^{m-1}} \left| f     *_2{t}_2\partial_{t_2} P_{t_2}  *\cdots * _m
    {t}_m\partial_{t_m}
    P_{t_m} ({x^\perp }+s  e_1 )  \right |^2  \frac {d\mathbf{t}_{{\mathfrak m}_2}}{\mathbf{t}_{{\mathfrak m}_2}} { ds  dx^\perp  } \\&
 \leq \| g_{\mathfrak m_2  }(f) \|_{{ L^2}}^2 \lesssim \|  f  \|_{{ L^2}}^2,
  \end{split}\end{equation}
    by Proposition \ref{thm:g-function}. Consequently,   $\mathfrak f_{{x^\perp }} \in L^2(\mathbb{ R}, { \mathfrak {H}} )$ for almost all ${x^\perp }\in e_1^\perp$.
 On the other hand,
\begin{equation} \label{eq:g-varphi-3}\begin{split}
   g_P  (\mathfrak  f_{{x^\perp }} )(s)& = \left(\int_{ \mathbb{R} _+  }\left|      \mathfrak  f_{{x^\perp }}
   * {t}_1\partial_{t_1}P_{t_1}(  s   )  \right|_{\mathfrak {H}}  ^2
   \frac {dt_1}{t_1}
   \right)^{\frac 12}
   \\ &=\left(\int_{(\mathbb{R} _+)^m}\left| f     *_2{t}_2\partial_{t_2} P_{t_2}  *\cdots * _m {t}_m\partial_{t_m} P_{t_m}
   *_1{t}_1\partial_{t_1} P_{t_1}({x^\perp }+s  e_1 )  \right|  ^2 \frac
   {d\mathbf{t}_{{\mathfrak m}_2}}{\mathbf{t}_{{\mathfrak m}_2}}
   \frac {dt_1}{t_1} \right)^{\frac 12}
   \\& =  \left(\int_{(\mathbb{R}_+)^m}\left| f  {{*}}_{1 }\partial_{t_1} P_{t_1}  {{*}}_{ 2} \cdots {*  }_m  \partial_{t_m} P_{t_m} \right|^2({x^\perp }+s  e_1 )
   \mathbf{t}d\mathbf{t}\right)^{\frac 12}\leq g(f) ({x^\perp }+s  e_1 )
   ,
 \end{split}\end{equation}by  the commutativity \eqref{eq:convolution-commutativity} of partial convolutions along lines. Consequently, by $g(f) \in L^1(\mathbb{ R}^n)$,
   we see that $g_P  (\mathfrak  f_{{x^\perp }} ) \in L^1(\mathbb{ R} )$ for almost all ${x^\perp }\in e_1^\perp$.
Then applying   Corollary \ref{cor:L1<g} ,  we get  $\mathfrak f_{{x^\perp }} \in L^1(\mathbb{ R}, { \mathfrak {H}} )$ for almost all ${x^\perp }\in e_1^\perp$, and
\begin{equation*} \begin{split} \int_{  \mathbb{ R}^{n-1} } 
 \int_{  \mathbb{ R}  } \left|
 \mathfrak f_{{x^\perp }}  \right|_{\mathfrak{H}} (s )  \, ds { dx^\perp} \lesssim \int_{  \mathbb{ R}^{n-1} }\int_{  \mathbb{ R}  } g_P  (\mathfrak f_{{x^\perp }}     ) (s) \,ds { dx^\perp}= 
 \int_{
 \mathbb{ R}^n } g   (
 f  ) ({x^\perp }+s  e_1 ) \,dsdx^\perp .
 \end{split}\end{equation*}
It can be rewritten as
 \begin{equation*}  \begin{split}
\int_{  \mathbb{ R}^n  } \left (\int_{    (\mathbb{R}_+) ^{m-1}} \left| f*_2 {t}_2\partial_{t_2} P_{t_2}  *\cdots * _m {t}_m\partial_{t_m} P_{t_m}
  (x ) \right |^2  \frac
  {d\mathbf{t}_{{\mathfrak m}_2}} {\mathbf{t}_{{\mathfrak m}_2}}  \right)^{\frac
  12}   dx
\lesssim \|g (f)\|_{{ L^1(\mathbb R^n)}} .
 \end{split}\end{equation*}If we define $\mathfrak {H}_1:=L
^2
\big((\mathbb{R}_+)^{m-2},\frac {d\mathbf{t}_{{\mathfrak m}_3}} {\mathbf{t}_{{\mathfrak m}_3}} \big)$ with the usual norm,
and   define   functions $\mathfrak f_{1;x'}:\mathbb{ R} \longrightarrow  { \mathfrak {H}}_1 $ for almost all $x'\in e_2^\perp$ by
$
   \mathfrak f_{1;x' }(s)=f     *_3{t}_3\partial_{t_3} P_{t_3}  *\cdots * _m {t}_m\partial_{t_m} P_{t_m} (x'+s  e_2 ) .
$
Similarly, as in \eqref{eq:g-m2}-\eqref{eq:g-varphi-3}, we get  $ \mathfrak f_{1;x' } \in L^2(\mathbb{ R}, { \mathfrak {H}}_1 )$ and  $g_P  (\mathfrak  f_{1;x' }  ) \in L^1(\mathbb{ R} )$  for almost all $x'\in 
e_2^\perp$, and
\begin{equation*}   \begin{split} &
\int_{  \mathbb{ R}^n  } \left (\int_{    (\mathbb{R}_+) ^{m-2}} \left| f*_3 {t}_3\partial_{t_3} P_{t_3}  *\cdots * _m {t}_m
\partial_{t_m} P_{t_m}
  (x ) \right |^2  \frac
  {d\mathbf{t}_{{\mathfrak m}_3}} {\mathbf{t}_{{\mathfrak m}_3}}  \right)^{\frac
  12}   dx \\
\lesssim &\int_{  \mathbb{ R}^n  } \left (\int_{    (\mathbb{R}_+) ^{m-1}} \left| f*_2 {t}_2\partial_{t_2} P_{t_2}  *\cdots * _m {t}_m\partial_{t_m} P_{t_m}
  (x ) \right |^2  \frac
  {d\mathbf{t}_{{\mathfrak m}_2}} {\mathbf{t}_{{\mathfrak m}_2}}  \right)^{\frac
  12}   dx .
 \end{split} \end{equation*}
Repeating this procedure, we finally get
\begin{equation*}  \begin{split}
 \| f
  \|_{L^1( \mathbb{ R}^n  )}  \lesssim \int_{  \mathbb{ R}^n  } \left (\int_{    \mathbb{R} _+ } | f  * _m {t}_m\partial_{t_m} P_{t_m} (x )  |^2  \frac {d t_m}{ {t}_m} \right)^{\frac 12}   dx\lesssim 
  \|g (f)\|_{{ L^1(\mathbb R^n)}} .
 \end{split}\end{equation*}
 The proposition is proved.
  \end{proof}

  \begin{cor}\label{cor:g-L1}For    $  f \in L^2(\mathbb{R}^n)$ with    $g (  f) \in L^1(\mathbb{R}^n)$, we have
 {
 $$\sup_{\mathbf{t}\in(\mathbb{R}_+)^m} \|f*P_{\mathbf{t}}\|_{L^1(\mathbb{R}^n)}\lesssim \|g(f)\|_{L^1(\mathbb{R}^n)} .
 $$}
  \end{cor}
  \begin{proof} For fixed $\mathbf{t} \in  (\mathbb{R}_+)^m$, $f*P_{\mathbf{t}}\in L^2(\mathbb{R}^n)$  by the  $L^2$-boundedness of the maximal function, and
  \begin{equation*}  \begin{split} g (f*P_{\mathbf{t}})(x ) =  &   \left(\int_{(\mathbb{R}_+)^m}\left| \nabla_{\mathfrak {m}}( f*P_{\mathbf{t}}*P_{\mathbf{t}' }) (x ) \right|^2
  \mathbf{t}'d\mathbf{t} '
   \right)^{\frac 12} =      \left(\int_{(\mathbb{R}_+)^m}\left| \nabla_{\mathfrak {m}}( f*P_{\mathbf{t}+\mathbf{t}' })(x )  \right|^2
   \mathbf{t}'d\mathbf{t} '
   \right)^{\frac 12} \end{split}\end{equation*} by $P_t*P_{t'}=P_{t+t'}$. Thus, it equals to
\begin{equation*} \begin{split} &   \  \left(\int_{   \mathbf{t}+(\mathbb{R}_+)^m}\left| \nabla_{\mathfrak {m}}( f*P_{\mathbf{t}' }) (x ) \right|^2
   (\mathbf{t}'-\mathbf{t})d\mathbf{t} '
   \right)^{\frac 12}  \leq    \  \left(\int_{(\mathbb{R}_+)^m}\left| \nabla_{\mathfrak {m}}( f*P_{\mathbf{t}' }) (x ) \right|^2  \mathbf{t}' d\mathbf{t}
   '
   \right)^{\frac 12} = g (f )(x ) ,
 \end{split}\end{equation*}
i.e. $g (f*P_{\mathbf{t}})\in L^1(\mathbb{R}^n)$. So we can apply Proposition \ref{prop:g-L1} to get $f*P_{\mathbf{t}}\in L^1(\mathbb{R}^n)$ and 
{ $$
   \|f*P_{\mathbf{t}}\|_{L^1(\mathbb{R}^n)}\lesssim \|g(f)\|_{L^1(\mathbb{R}^n)} ,
   $$}
    where the implicit constant is independent of $\mathbf{t}$.
   \end{proof}

\section{Proof of the  equivalence of various characterizations }
In this section, we prove the  equivalence of various characterizations in  Theorem \ref{thm:equivalence}.
\vskip 2mm
 \noindent  {\bf (1) implies (2)}.
Let $f_\varepsilon:=F(\cdot +\mathbf{i}\pi( \boldsymbol
  { {\varepsilon }}))\in L^1(\mathbb{ R}^n   )$.
         Apply     Proposition
  \ref{prop:subharmonic-ineq}  to get
 \begin{equation*}
 | F ( x+\mathbf{i}\pi(\mathbf{t}+\boldsymbol
  { {\varepsilon }}))|^q \leq  | f_\varepsilon
    |^q *  P_{\mathbf{t }} (x)
 \end{equation*} for   $(x,\boldsymbol {t})\in  \mathbb{ R}^n \times(\mathbb{R}_+)^m$, if we choose $0<q<1$. Then, $| f_\varepsilon|^q\in L^r(\mathbb{ R}^n   )$ with $r=\frac{1}{q}>1$, and
 \begin{equation*} \begin{split}
 \sup\limits_{(x',\mathbf{ {t}}   )\in  \Gamma_\beta(x)} |F ( x'+\mathbf{i}\pi(\mathbf{t}+\boldsymbol
  { {\varepsilon }}))|^q&\leq  \sup\limits_{(x',\mathbf{ {t}}   )\in  \Gamma_\beta (x)}  | f_\varepsilon
    |^q *  P_{\mathbf{t }} (x')  \lesssim
   M_{it}\left( |f_\varepsilon|^q \right)(x),
 \end{split} \end{equation*}by using \eqref{eq:heat-max},
where  $M_{it}$ is
  the iterated maximal function on
  $\mathbb{ R}^n   $. Therefore,
\begin{equation}   \label{maximal-ineq} \begin{split}
 \left\| \sup\limits_{(x',\mathbf{ {t}}   )\in  \Gamma_\beta(x)} |F ( x'+\mathbf{i}\pi(\mathbf{t}+\boldsymbol
  { {\varepsilon }}))|^q\right\|^r_{L^r(\mathbb{ R}^n , { dx}  )} &
 \lesssim \left\| M_{it}\left( |f_\varepsilon|^q \right) \right\|^r_{L^r(\mathbb{ R}^n   )}
  \lesssim \|\left|f_\varepsilon\right|^q\|^r_{L^r(\mathbb{ R}^n   )} =\| f_\varepsilon \|_{L^1(\mathbb{ R}^n   )},
 \end{split}\end{equation}
 where implicit constants are independent of $F $ and $  {\varepsilon } $.
Letting $\varepsilon  \rightarrow 0$ in \eqref{maximal-ineq}, we obtain
\begin{equation}   \label{maximal-ineq2} \begin{split} \left\| \sup\limits_{(x',\mathbf{ {t}}   )\in  \Gamma_\beta(x)} |F ( x'+\mathbf{i}\pi(\mathbf{t} ))|  \right\|_{L^1(\mathbb{
R}^n  , { dx} )} &= \left\| \lim_{ \varepsilon \rightarrow 0}\sup\limits_{
(x',\mathbf  {t}    )\in (0, \boldsymbol   {\varepsilon } )+\Gamma_\beta(x)}
|F ( x'+\mathbf{i}\pi(\mathbf{t} ))|^q\right\|_{L^r(\mathbb{ R}^n ,  { dx}
  )}^r\\ &
 \leq\liminf_{ \varepsilon \rightarrow 0}  \left\|\sup\limits_{(x',\mathbf{ {t}}   )\in  \Gamma_\beta(x)} |F ( x'+\mathbf{i}\pi(\mathbf{t}+\boldsymbol
  { {\varepsilon }}))|^q \right\|_{L^r(\mathbb{ R}^n  , { dx}  )}^r\\
& \lesssim  \liminf_{ \varepsilon \rightarrow 0} \| f_\varepsilon \| _{L^1(\mathbb{ R}^n  )}
 \end{split}\end{equation}  by Fatou's theorem, since domains $(0, \boldsymbol
  { {\varepsilon }})+\Gamma_\beta (x) $
are increasing when $\varepsilon \downarrow 0$ and the limit is the set
 $\Gamma_\beta(x)$. Therefore,
 $\|  \mathbb{N} ^\beta( F )\|_{L^1({  \mathbb{ R}^n})}
  \lesssim \|F\|_{{H^1( T_\Omega)}} $.

\vskip 2mm
\noindent   {\bf (2) implies (3)}. At first,  $F(x+\mathbf{i} y )\rightarrow 0 $ as $| y| \rightarrow +\infty$  in $y_0+\Omega $ for any fixed $y_0\in  \Omega$, by the boundary growth  in Proposition 
\ref{prop:upperbound-hardy}.
  Denote $
   F_\varepsilon ( x+\mathbf{i}y  ) :=  F ( x+\mathbf{i}y+ \mathbf{i}\pi( \boldsymbol
  { {\varepsilon }})).
 $ Then
  $F_\varepsilon $ is smooth on $\overline{T_\Omega}$. By the estimate in Proposition \ref{prop:upperbound-hardy},
      \begin{equation*}
       | F_\varepsilon ( x+\mathbf{i}\pi(\mathbf{t}))|\lesssim \frac {  \|F\|_{H^1( T_\Omega)}  }
       {|   R  (0,\boldsymbol     {\varepsilon } )| } , \quad {\rm for }\quad (x,\boldsymbol {t})\in  \mathbb{ R}^n \times(\mathbb{R}_+)^m ,
    \end{equation*}
   i.e.
   $F_\varepsilon$ is bounded on $\overline{T_\Omega}$, and so it belongs
  to $   H^p(T_\Omega)$ for any $p>1$. Then by Theorem \ref{prop:integral-representation}, we get
   $
     F_\varepsilon ( x+\mathbf{i}\pi(\mathbf{t}))=
   f_\varepsilon*  P_{\mathbf{t }}
$
   with $ f_\varepsilon(\cdot)= F  ( \cdot+\mathbf{i}\pi( \boldsymbol
  { {\varepsilon }})  )\in L^p(\mathbb{R}^n)$.
   Apply
    Corollary \ref{cor:Fefferman-Stein-inequality} to get
      \begin{equation} \label{eq:S-N}
       \| S(  {f}_\varepsilon)\|_{L^1(\mathbb{R}^n )}\lesssim\|N^\beta( f_\varepsilon) \|_{L^1(\mathbb{R}^n)}   =    \| \mathbb{N }^\beta( F_\varepsilon
       )\|_{L^1(\mathbb{R}^n)}\leq \|\mathbb{ N} ^\beta( {F} ) \|_{L^1(\mathbb{R}^n )}.
     \end{equation}

   On the other hand,
  $ \| S(  {f}_\varepsilon)\|_{L^1(\mathbb{R}^n )}=\| \mathbb{S}( F_\varepsilon)\|_{L^1(\mathbb{R}^n)}
$ by definition, and
\begin{equation}\label{eq:change-t}\begin{split}
  \mathbb{ S}^2 ( F_{\varepsilon })(x)& =\int_{\Gamma(0)}
\left|  \nabla_\mathfrak {m}(F  (x+x' +\mathbf{i}\pi(\mathbf{t}) +\mathbf{i}\pi( \boldsymbol
  { {\varepsilon }} ) ) \right|^2  \frac { \mathbf{t} d\mathbf{t}dx' }{| R  (0,\mathbf{t} )|} \\&=\int_{\mathbb{R}^n \times(\mathbb{R}_+)^m }
\left|  \nabla_\mathfrak {m}(F  (x+x' +\mathbf{i}\pi(\mathbf{t})  ) \right|^2\chi_{
  (0, \boldsymbol   {\varepsilon } )+\Gamma_\beta(x)} ( x',\mathbf{t}   )  \frac { (t_1-\varepsilon)\cdots
(t_m-\varepsilon)d\mathbf{t}dx' }{| R  (0,\mathbf{t}-\boldsymbol
  { {\varepsilon }} )|}.\end{split}\end{equation}
Given $\mathbf{t} $, let  $\mathfrak l$ be $\{l_1 , \dots, l_n \} $ such that $t_{l_1}, \dots,t_{l_n}$ are the largest $n$ numbers among   $\{t_{ 1}  , \dots,t_{m} \}$. Then,  $(t_{l_1}  -\varepsilon , 
\dots,t_{l_m}-\varepsilon)$ are also the largest
$n$ numbers among $\{t_1-\varepsilon , \dots,t_{
n}-\varepsilon\}$. Hence,
\begin{equation*}
    \frac { (t_1-\varepsilon)\cdots (t_m-\varepsilon)  }{| R  (0,\mathbf{t}-\boldsymbol
  { {\varepsilon }} )|}   \rightarrow   \frac { t_1 \cdots  t_m  }{| R  (0,\mathbf{t} )|},
\end{equation*}as  $\varepsilon\rightarrow 0$, by Proposition \ref{prop:twist-R} and \eqref{eq:volum}.
Thus,
\begin{equation*}\begin{split} \mathbb{ S}^2 ( F_{\varepsilon })(x)
 &\rightarrow \int_{\Gamma(0)}
\left|  \nabla_\mathfrak {m}(F  (x+x' +\mathbf{i}\pi(\mathbf{t})  ) \right|^2 \frac { \mathbf{t} d\mathbf{t}dx' }{| R  (0,\mathbf{t} )|}
 =  \mathbb{ S}^2 ( F )(x),
\end{split}\end{equation*}by  Lebesgues' dominated convergence theorem.
   The result follows by taking limit $\varepsilon\rightarrow 0$ in \eqref{eq:S-N}.

 \vskip 2mm
  \noindent {\bf (3) implies (4)}.
 \begin{prop}\label{prop:g-S}
For $ f \in L^1( \mathbb{ R}^n )$, we have $g (f)(x) \lesssim  {S} (f)(x)$.
In particular, if $F$ is     holomorphic   on
$ T_\Omega $, then $ \mathbb{ G} (F)(x) \lesssim  \mathbb{ S} (F)(x)$.
  \end{prop}
  \begin{proof} By translations, it is sufficient to prove the inequality for $x=0$.  By Proposition \ref{prop:triangle-j},
      $\triangle_\mu u      (x,\mathbf{t})=0$ for $u(x,\mathbf{t})=f_\varepsilon*   P_{  \mathbf{t }}(x)$ with $\varepsilon>0$, where $f_{
   \varepsilon  } =f* P_{    \boldsymbol    \varepsilon  } \in L^1( \mathbb{ R}^n )$ and is smooth. Thus $\triangle_\mu \nabla_\nu u
     (x,\mathbf{t})=0$, i.e.   $ \nabla_\nu u(x) $        is also harmonic in the half plane spanned by $e_\mu$ and $t_\mu$. Then for given $\mathbf{t}\in(\mathbb{R}_+)^m $,  we get
          \begin{equation} \label{eq:mean}
        \nabla_\mathfrak {m}u(0,\mathbf{t})=\frac {2 ^{2m}}{\pi^m \gamma_0^{2m} t^2_1\cdots t^2_m}\int_{D_1( \gamma_0 t_1/2)\times\cdots\times D_m( \gamma_0 t_m/2)}  \nabla_\mathfrak {m}
        u\left(\sum_{j=1}^m s_j'e_j,\mathbf{t}'\right) d\mathbf{s}'d\mathbf{t}'.
     \end{equation}by repeatedly using the mean value formula for harmonic
     functions, where disc $D_j( r):=\{(s_j',t_j'); |s_j'|^2+ |t_j'-t_j|^2\leq r^2    \}$. Note that $t_j'/2  \leq t_j\leq 2t_j'$ when $(s_j',t_j')\in D_j( \gamma_0 t_j/2)$.
     Apply Cauchy-Schwarz inequality to \eqref{eq:mean} to get
       \begin{equation*}\begin{split}
    g^2(f_\varepsilon)(0)&=  \int_{(\mathbb{R}_+)^m}  | \nabla_\mathfrak {m} u(0,\mathbf{t})|^2\mathbf{t} \, d\mathbf{t}
    \\&\lesssim  \int_{(\mathbb{R}_+)^m}
 \int_{D_1( \gamma_0 t_1/2)\times\cdots\times D_m( \gamma_0 t_m/2)} \left | \nabla_\mathfrak {m}  u\left(\sum_{j=1}^m s_j'e_j,\mathbf{t}'\right)\right|^2d\mathbf{s}'d\mathbf{t}'  
{  \frac  {d\mathbf{t}}{  t _1\cdots t   _m}} \\
    &\lesssim   \int_{(\mathbb{R}_+)^m}  
    \int_{|s_1'| \leq \gamma_0 {t_1' } }   \cdots \int_{|s_m'| \leq \gamma_0{t_m' } }
   \left | \nabla_\mathfrak {m}  u\left(\sum_{j=1}^m s_j'e_j,\mathbf{t}'\right)\right|^2 { 
   ds_m'\cdots ds_1' \, d\mathbf{t}' }\\
    &\lesssim \int_{(\mathbb{R}_+)^m}  |  \nabla_\mathfrak {m }u(\cdot, \mathbf{t}') |^2* \chi_{\gamma_0 \mathbf{t} '   }(0)\, {  \mathbf{t}'d\mathbf{t}'  } 
 \\&\lesssim \int_{(\mathbb{R}_+)^m} \int_{R (0,\mathbf{t}')}| \nabla_\mathfrak {m}u(x',\mathbf{t}') |^2dx'
 { \frac { \mathbf{t}'d\mathbf{t}'}{| R  (0,\mathbf{t}')|}} 
    \\&=\int_{\Gamma(0)}
    | \nabla_\mathfrak {m} u(x',\mathbf{t}')  |^2\frac { \mathbf{t}'d\mathbf{t}'dx'}{| R  (0,\mathbf{t}')|}=S^2(f_\varepsilon)(0)   .
  \end{split}  \end{equation*}  by using Proposition \ref{prop:it-twist}.   The result follows by taking limit $\varepsilon\rightarrow 0$.
    \end{proof}

\vskip 2mm
 \noindent  {\bf  (4) implies (1)}.
For $0<\varepsilon<M$, let
\begin{equation}\label{eq:f-varepsilon-K} \begin{split}
 F_{\varepsilon M}(  x+\mathbf{i}\pi(\mathbf{t} )):=&\sum_{\mathfrak j} (-1)^{|{\mathfrak j} |}F (  x+\mathbf{i}\pi(\mathbf{t} +\mathbf{M}_{\mathfrak j} ) )
\\= & \int_{t_1+ \varepsilon}^{t_1+M }\cdots \int_{t_m+ \varepsilon}^{t_m+M }  \partial_{t_1'}\cdots \partial_{t_m'}( F )  (  x+\mathbf{i}\pi(\mathbf{t}' ))
 d\mathbf{t}',
 \end{split}\end{equation}where  the summation is taken over subsets ${\mathfrak j}  $   of $\{1,2,\dots,m\}$, and
  \begin{equation*}
    \mathbf{M}_{\mathfrak j}=(M_1,\dots,M_m), \qquad M_j=\left\{
    \begin{array}{ll} M,\qquad & j\notin \mathfrak j,\\
    \varepsilon,\qquad & j\in \mathfrak j.
           \end{array}
    \right.
 \end{equation*}

 By applying Cauchy-Schwarz inequality to \eqref{eq:f-varepsilon-K}, we get
 \begin{equation*}\begin{split}
     | F_{\varepsilon M}(  x+\mathbf{i}\pi(\mathbf{t} ))|
     &\leq \left(\ln\frac M\varepsilon \right)^{\frac m2}\left(\int_{t_1+ \varepsilon}^{t_1+M}\cdots \int_{t_m+ \varepsilon}^{t_m+M}  |\partial_{t_1'}\cdots \partial_{t_m'}( F )  (
     x+\mathbf{i}\pi(\mathbf{t}' )) |^2 \mathbf{t}'   d\mathbf{t}'   \right)^{\frac 12}\\
 &\leq \left(\ln\frac M\varepsilon \right)^{\frac m2}\mathbb{  G } (F)(x)<\infty.
  \end{split}\end{equation*}
 Thus,
\begin{equation*}
 \sup_{\mathbf{ {t}}\in (\mathbb{R}_+)^m} \left \|  F_{\varepsilon M}(  \cdot+\mathbf{i}\pi(\mathbf{t} ))\right\|_{L^1(\mathbb{R}^n)}\leq \left(\ln\frac M\varepsilon \right)^{\frac m2} \|\mathbb{  G  }(
 F )\|_{L^1(\mathbb{R}^n)}<+\infty,
\end{equation*}
i.e. $F_{\varepsilon M}\in H^1(T_\Omega)$.
Consequently, as in (2) implies (3), $F_{\varepsilon M}( \cdot  +\mathbf{i}\pi(\boldsymbol
    {\varepsilon })  )  $ is smooth on $\overline{T_\Omega}$ by definition \eqref{eq:f-varepsilon-K}, and by the estimate in Proposition \ref{prop:upperbound-hardy}, it  is bounded on $\overline{T_\Omega}$, and so 
    it
belongs to $   H^2(T_\Omega)$. Then by Proposition \ref{prop:integral-representation}, we get
    \begin{equation}\label{eq: F-varepsilon-rep}
    F_{\varepsilon M} ( x+\mathbf{i}\pi(\mathbf{t}+\boldsymbol
    {\varepsilon })  )=
  f_{\varepsilon M}*  P_{\mathbf{t }} ,
    \end{equation}
 where   $ f_{\varepsilon M}(\cdot)= F_{\varepsilon M}( \cdot  +\mathbf{i}\pi(\boldsymbol
    {\varepsilon })  )\in L^2 (\mathbb{R}^n)$. Moreover,
        $
       \|g(
 f_{\varepsilon M})\|_{L^1(\mathbb{R}^n)} =\|\mathbb{  G  }(
 F_{\varepsilon M}  )\|_{L^1(\mathbb{R}^n)}
 $
     by  \eqref{eq: F-varepsilon-rep}, and so
     \begin{equation}\label{eq:widetilde-f-g-f-g'}\begin{split}  g( f_{\varepsilon M})  (x)&=
  \mathbb{ {G}} ( F_{\varepsilon M}(\cdot ))(x +\mathbf{i}\pi(\boldsymbol
    {\varepsilon }) ) \\
    &=\left(\int_{(\mathbb{R}_+)^m  }
\left| \nabla_\mathfrak {m}(F_{\varepsilon M} (x +\mathbf{i}\pi(\mathbf{t})  +\mathbf{i}\pi(\boldsymbol
    {\varepsilon })  )  \right|^2
 { \mathbf{t}}  {d\mathbf{t}  }\right)^{\frac 12}\\
 &\lesssim\left(\sum_{\mathfrak j}\int_{(\mathbb{R}_+)^m  }
\left| \nabla_\mathfrak {m}(F  (x +\mathbf{i}\pi(\mathbf{t}+ \mathbf{M}_{\mathfrak j} ) +\mathbf{i}\pi(\boldsymbol
    {\varepsilon })  )  \right|^2
 { \mathbf{t}}  {d\mathbf{t}  }\right)^{\frac 12}\\
 & \lesssim\left(\int_{(\mathbb{R}_+)^m  }
\left| \nabla_\mathfrak {m}
(F  (x +\mathbf{i}\pi(\mathbf{t} ) )  \right|^2
 { \mathbf{t}}  {d\mathbf{t}  }\right)^{\frac 12}\approx  \mathbb{ {G}} ( F )(x),
\end{split}\end{equation}
by definition \eqref{eq:f-varepsilon-K} of $ F_{\varepsilon M}$ and changing  coordinates as in \eqref{eq:change-t}.
     Thus,
 by Proposition \ref{prop:g-L1}, we see that $f_{\varepsilon M} \in L^1(\mathbb{R}^n)$ and $\| f_{\varepsilon M} \|_{L^1(\mathbb{R}^n)} \lesssim  \|  \mathbb{ {G}} ( F ) \|_{L^1(\mathbb{R}^n)}$.

On the other hand,  by Corollary \ref{cor:g-L1} and \eqref{eq: F-varepsilon-rep}, we have
\begin{equation*}\begin{split}
   \| F_{\varepsilon M} (\cdot +\mathbf{i}\pi(\mathbf{t}+\boldsymbol
    {\varepsilon })  )\|_{L^1(\mathbb{R}^n)}
   \lesssim   \| g( f_{\varepsilon M} ) \|_{L^1(\mathbb{R}^n)} \lesssim  \|  \mathbb{ {G}} ( F ) \|_{L^1(\mathbb{R}^n)},
\end{split}\end{equation*}for any  $\mathbf{t}\in (\mathbb{R}_+)^m$.
  Since for $y\in\Omega$,  $F(x+\mathbf{i} y +\mathbf{i}y')\rightarrow 0 $ as $| y'| \rightarrow +\infty$ by the assumption,  it follows from  definition \eqref{eq:f-varepsilon-K} that
    \begin{equation*}
     F_{\varepsilon M}   (
  x+\mathbf{i}y ) \rightarrow (-1)^m F  (
x+\mathbf{i}y  )
  \end{equation*}
    as $\varepsilon\rightarrow 0$ and $M\rightarrow+\infty$. So we  conclude that
\begin{equation*}
 \| F  (\cdot +\mathbf{i}y)\|_{L^1(\mathbb{R}^n)} \leq \liminf_{\substack{\varepsilon\rightarrow 0\\ M\rightarrow+\infty} } \| F_{\varepsilon M} (\cdot +\mathbf{i}y +\mathbf{i}\pi(\boldsymbol
    {\varepsilon }) )\|_{L^1(\mathbb{R}^n)}
    \lesssim \|  \mathbb{ {G} } ( {f} )\|_{L^1(\mathbb{R}^n)},
\end{equation*}
by Fatou's theorem, where the implicit constants are independent of $y\in\Omega $. This complete the proof.

 \appendix
\section{The Plancherel-P\'olya type inequality for $\mathfrak{H} $-valued  functions}
See e.g. \cite{FJ,Han94,Han,HLLW} for Plancherel-P\'olya type inequality.
 For  $\mathfrak f  =\sum_{j=1}^\infty f_j(x)h_j$, by definition,  $\mathfrak  f_{k }=\sum_{j=1}^k
f_j(x)h_j$ converges to  $\mathfrak  f $ under the norm $L
^p
(\mathbb{ R},\mathfrak{H})$, i.e. $ \int_{\mathbb{ R}}\left|\mathfrak  f -\mathfrak  f_{k }\right|^p_{\mathfrak{H}}dx=\int_{\mathbb{ R}}(\sum_{j=k+1}^\infty| f_j(x)|^2)^\frac
p2 { dx}\rightarrow 0$.

By the wavelet Calder\'on
reproducing formula \eqref{eq:reproducing} on $\mathbb{R} $, we get
\begin{equation*}  \begin{split}
   \mathfrak   f(x)=&
\int_{\mathbb{R} _+}\varphi_ { t }*\rho_ { {t}}*  \mathfrak  f(x)\frac { d {t}}{ {t}}
=  \sum_{ {j}\in \mathbb{Z} }  \int_{2^{-\alpha j} }^{2^{-\alpha (j-1) }}  \varphi_ { t }*\rho_ { {t}}* \mathfrak f(x)\frac { d {t}}{ {t}}\\
=&c_\alpha\sum_{j } \widetilde{\varphi}_{j}* \widetilde{ \rho}_ {{j}}* \mathfrak f(x)
  +\sum_{ {j}\in \mathbb{Z} } \int_{2^{-\alpha j} }^{2^{-\alpha (j-1) }} ( \varphi_ { t }*\rho_ { {t}}-\widetilde{\varphi}_ {j}* \widetilde{ \rho}_ {j})*\mathfrak f(x)\frac { d {t}}{ {t}}
   := \mathcal  T_\alpha \mathfrak  f (x) +\mathcal  R_\alpha \mathfrak f (x),
          \end{split}  \end{equation*}for $ \mathfrak f \in   L
^2
(\mathbb{ R},\mathfrak{H})$
where
\begin{equation}
   \widetilde{\varphi}_ {j}:=\frac 1{c_\alpha}\int_{2^{-\alpha j} }^{2^{-\alpha (j-1) }}  \varphi_ { t } \frac { d {t}}{ {t}}.
\end{equation}
Now decompose $\mathbb{R} $ into disjoint dyadic intervals $I$
with $\ell(I)=2^{-\alpha( j  +N) }$, where $N$ is a large positive integer fixed later. Then
 \begin{equation*} \begin{split}T_\alpha  \mathfrak f (x)&:= c_\alpha\sum_{j \in \mathbb{Z}} \widetilde{\varphi}_{j}* \widetilde{ \rho}_ {{j}}*  \mathfrak  f(x) =c_\alpha\sum_{j \in
 \mathbb{Z}}\sum_I \int_{I}\widetilde{\varphi}_{j}(x-x')  \widetilde{ \rho}_ {{j}}* \mathfrak f(x') dx' \\
 &=c_\alpha\sum_{j\in \mathbb{Z} }\sum_I   \ell(I) \left(\frac 1{\ell(I)}\int_{I }\widetilde{\varphi}_{j}(x-x') dx'\right) \widetilde{ \rho}_ {{j}}* \mathfrak f(x_I)+\mathcal R_{\alpha,N} '
 \mathfrak f (x)
   \end{split}\end{equation*}
   with   $x_I$  any fixed point  in $I$, and
    \begin{equation*} \begin{split}\mathcal  R_{\alpha,N} ' \mathfrak f (x) &:=  c_\alpha\sum_{j\in \mathbb{Z} }\sum_I \int_{I}\widetilde{\varphi}_{j}(x-x') \Big  (\widetilde{
    \rho}_ {{j}}*\mathfrak f(x') -\widetilde{ \rho}_ {{j}}* \mathfrak f(x_I)\Big )dx'.
   \end{split}\end{equation*}
Now we can write
  \begin{equation}\label{eq:T-f}
    \mathfrak f (x)= \mathcal{T}  \mathfrak f (x)+\mathcal  R_\alpha \mathfrak f (x) +R_{\alpha,N} ' \mathfrak f (x), \qquad {\rm where}  \qquad\mathcal{T}  \mathfrak f (x)=\sum_{ j }\sum_{I}   \mathfrak
 f_{j;I    }   \widetilde{\phi}_{j  } (  \cdot, x_I ),
 \end{equation}
 where    $\mathfrak f_{j;I    } $'s are given by \eqref{eq:widetilde-rho}, and
 \begin{equation}   \label{eq:phi-I}\begin{split}
 \widetilde{\phi}_{j  } ( x, x_I ):= &\frac 1{\ell(I)}\int_{I }\widetilde{\varphi}_{j}(x-x') dx'\frac {\widetilde{ \rho}_ {{j}}*\mathfrak f(x_I)  }{ |\widetilde{ \rho}_ {{j}}*\mathfrak f(x_I)  |_{\mathfrak{H}}} 
 \in \mathcal{ M} (\beta, \gamma,  2^{-\alpha j},
 x_I;\mathfrak{H}) .
              \end{split}  \end{equation}Here we take $ \widetilde{\phi}_{j  } (  \cdot, x_I ) \equiv 0$ if $|\widetilde{ \rho}_ {{j}}*\mathfrak f(x_I)  |_{\mathfrak{H}}=0$.
            Since $\varphi$ is compactly supported and smooth, it is easy to see that  $\widetilde{\phi}_{j  } (  \cdot, x_I )$ have uniformly bounded norm in $\mathcal{ M} (\beta, \gamma,  2^{-\alpha j},
 x_I;\mathfrak{H})$ by the scalar case.
                            To invert $\mathcal{T}$, we now need to show   $\mathcal{R}_{\alpha,N} '   $ and  $\mathcal{R}_{\alpha }     $ bounded on $  {M}   (\beta, \gamma, r, x;\mathfrak{H})  $  with small 
                            norms.
Here
$\mathcal R_{\alpha,N} ' $ has kernel
\begin{equation} \label{eq:R-alpha-N}  \begin{split}\mathcal R_{\alpha,N} ' (x,u)=c_\alpha    \sum_{j } \sum_I \int_{I}\widetilde{\varphi}_{j}(x-x')\left [ \widetilde{ \rho}_ {{j}}(x'-u)-\widetilde{
\rho}_ {{j}}(x_I-u)\right]    dx'.
       \end{split}  \end{equation}

        Let $T$ be a bounded linear operator on $L^2(\mathbb{R})$ associated
with a scalar kernel $K(x, y)$ defined on $\{(x, y) \in\mathbb{R}\times\mathbb{R}: x \neq y\}$, given initially by
$T f(x) = \int_{\mathbb{R}}
K(x, y)f(y)dy$ or $x \in $ supp $f$
for $f \in { \mathscr C^\infty_c}(\mathbb{R})$, where $K(x, y)$ satisfies the following conditions: there exists a constant $C > 0$ such that for all $x \neq y$,
\\
(i) $|K(x, y)|\leq C|x -y|^{-1}$;
\\
(ii) $|K(x, y)- K(x ', y)|\leq C|x-x ' ||x - y|^{-2}$,
  if $|x- x' |\leq |x- y|/2$;
  \\
(iii) $|K(x, y) - K(x, y' )|\leq C|y -y' ||x-y|^{-2}$,
  if $|y - y' |\leq|x - y|/2$;
  \\
(iv) $|K(x, y) - K(x' , y) - K(x, y' ) + K(x' , y' )| \leq C|x - x' ||y - y' ||x - y|^{-3}$,
  if $|x - x' | $, $|y -y' |\leq|x-y|/2$;
   \\
(v) $ T(1) =
T ^*(1) = 0$.

We denote by $ \|K\|_{\mathbb{R}}$   the smallest constant $C$ that satisfies (i)-(iv) above. The
operator norm of $T$ is defined by ${ \|T\|_*} := \| T\|_{ L^2(\mathbb{R})\rightarrow L^2(\mathbb{R})} + \| K\|_{\mathbb{R}}$. It is known that
$\| T\|_{ L^2(\mathbb{R})\rightarrow L^2(\mathbb{R})}\rightarrow 0$ as $C\rightarrow 0$.

       \begin{prop}\label{prop:T-M}
          Suppose that $T$ is an operator as  above. Then $T$ is bounded on the test function space ${M}   (\beta, \gamma, r, x_0;\mathfrak{H})$  for $  \beta,\gamma \in
(0, 1),r> 0$ and $x_0\in\mathbb{ R}$. Moreover, there exists a constant $C$ independent of $\beta, \gamma, r, x_0$ such that
\begin{equation*}
  \|T(f)\|_{{M}   (\beta, \gamma, r, x_0;\mathfrak{H})}\leq C { \|T\|_*}   \| f \|_{{M}   (\beta, \gamma, r, x_0;\mathfrak{H})}.
\end{equation*}
       \end{prop}
    \begin{proof}     For the scalar case, this is  \cite[Lemma 2.9]{HLLW}, i.e. the theorem holds for ${M}   (\beta, \gamma, r, x_0 )$. For the $ \mathfrak{H} $-valued function $\mathfrak f  =\sum_{j=1}^\infty 
    f_j(x)h_j\in {M}   (\beta,
    \gamma, r, x_0;\mathfrak{H})$, we must have $ f_i \in {M}   (\beta, \gamma, r, x_0 ) $ for each $i$, and   $\|\mathfrak f\|_{{M}   (\beta, \gamma, r, x_0;\mathfrak{H})}=(\sum_i
   \|f_i\|_{{M}   (\beta, \gamma, r, x_0 )}^2)^\frac 12$ by definition. Because   $T\mathfrak f = \sum_{j=1}^\infty Tf_j h_j$, we see that
       \begin{equation*}  \begin{split}
    \|T(\mathfrak f)\|_{{M}   (\beta, \gamma, r, x_0;\mathfrak{H})}&=\left(\sum_i  \|T(f_i)\|_{{M}   (\beta, \gamma, r, x_0 )}^2\right)^\frac 12\\
    &\leq  C{ \|T\|_*} \left(\sum_i  \| f_i \|_{{M}   (\beta, \gamma, r, x_0 )}^2\right)^\frac 12 \\
    &=C{ \|T\|_*} \, \|  \mathfrak f \|_{{M}   (\beta, \gamma, r, x_0;\mathfrak{H})}.
          \end{split}  \end{equation*}
          The result follows.
      \end{proof}

    \begin{proof}[Proof of Proposition \ref{prop:discrete-Calderon}]
    Noting that for $x'\in I$,  $ \widetilde{ \rho}_ {{j}}(x'-u)-\widetilde{ \rho}_ {{j}}(x_I-u) \sim  2^{-\alpha N } \widetilde{ \rho}_ {{j}}(x'-u)$ in terms of the size and smoothness
condition,  we get
\begin{equation} \label{eq:R-alpha}
  R_{\alpha,N} '(x,u)\sim     2^{-\alpha N } \sum_{j } \widetilde{\varphi}_{j}* \widetilde{ \rho}_ {{j}}(x -u)
\end{equation}
by definition \eqref{eq:R-alpha-N},
where $\sim$ denotes the equivalence in terms of   the size and smoothness
conditions of $ R_{\alpha,N} '(x,u)$. It is { straightforward to verify} that the kernel $ R_{\alpha,N} '(x,u)$ satisfies conditions (i)-(iv) and the cancellation { condition}, e.g. if 
$|s|\approx 2^{-\alpha l  }$ for some $l\in \mathbb{Z}$,
then
\begin{equation*}  \begin{split}
    \sum_{j } |\widetilde{\varphi}_{j}* \widetilde{ \rho}_ {{j}}|(s)  \lesssim \sum_{j   }\frac {2^{-\alpha j\gamma  }}{(2^{-\alpha j  }+|s|)^{1+\gamma}} \leq\sum_{j > l  }\frac
    {2^{-\alpha j\gamma  }}{ |s| ^{1+\gamma}}+\sum_{j \leq l  } 2^{ \alpha j   } \approx \frac{2^{-\alpha l\gamma  }}{ |s| ^{1+\gamma}}+2^{ \alpha l   } \approx \frac
    1{ |s|}.
          \end{split}  \end{equation*}
 Similarly, $\mathcal R_{\alpha } $ has
the similar expression as \eqref {eq:R-alpha} with $\alpha$ instead of $2^{-\alpha N }   $ and satisfies the same conditions.

Hence,  applying Proposition \ref{prop:T-M},  we obtain that
\begin{equation}
   \| \mathcal R_{\alpha }    \|_{{M}   (\beta, \gamma, r, x_0;\mathfrak{H})}\lesssim \alpha  \leq \frac 13,
\end{equation}
 { provided} we choose $\alpha$   sufficiently small,
and then
  \begin{equation}
    \|R_{\alpha,N} '\|_{{M}   (\beta, \gamma, r, x_0;\mathfrak{H})} \lesssim   2^{-\alpha N } \leq \frac 13,
 \end{equation}
  if  we choose $N$  sufficiently large.
Therefore, $\mathcal{T} =id-R_{\alpha,N} '-\mathcal R_{\alpha } $ is invertible and its inverse $\mathcal{T}^{-1}$ is also bounded on ${M}   (\beta, \gamma, r,
x_0;\mathfrak{H})$, with a bound independent of $r,
x_0$.
Consequently, for
$\widetilde{\phi}_{j  } (  \cdot, x_I )$ given by \eqref{eq:phi-I},
$\mathcal{T}^{-1}(\widetilde{\phi}_{j  } (  \cdot, x_I ))$ is also in ${M}   (\beta, \gamma, 2^{-\alpha j}, x_I;\mathfrak{H})$. Denote it by $ \phi _ {j}(x,x_I)$. Also, $\mathcal{T}  $ is invertible and its 
inverse $\mathcal{T}^{-1}$ is also bounded on $  L^2(\mathbb{R},\mathfrak{H})$. Then applying $\mathcal{T}^{-1}$ to \eqref{eq:T-f}, we get
\begin{equation}\begin{split}
   \mathfrak   f(x)=&\mathcal{T}^{-1} \mathcal{T}\mathfrak f =\sum_{ j }\sum_{I} c_\alpha\mathfrak  |I||\widetilde{ \rho}_ {{j}}*  \mathfrak  f(x_I) |_{\mathfrak{H}} \widetilde{ \phi} _ {j}(x,x_I).
          \end{split}  \end{equation}
          Here  $ \|\phi _ {j}( \cdot,x_I)\|_{{M}   (\beta, \gamma, 2^{-\alpha j}, x_I;\mathfrak{H})}$ are uniformly bounded by Proposition \ref{prop:T-M}.
\end{proof}

   \begin{proof}[Proof of Proposition \ref{thm:Plancherel-Polya}] { Convolving}  the wavelet Calder\'on reproducing formula \eqref{eq:wavelet-Calderon}  with $ \psi_t$  { we obtain}
   \begin{equation} \label{eq:psi*f} \begin{split}
\psi_t *  \mathfrak   f(x)=& \sum_{ j }\sum_{I} \mathfrak f_{j;I} \psi_t* \widetilde{\phi} _ {j} (\cdot,x_I)(x).
          \end{split}  \end{equation}
          But   for $2^{-\alpha k } \leq t<2^{-\alpha(k -1)}$, we have the following almost orthogonality
estimate for $\widetilde{\phi} _ {j} (\cdot,x_I)$ and $\psi_t$ (cf.  \cite[Lemma 6]{HLS}),
          \begin{equation}   \label{eq:almost-orthogonality} \begin{split}
 \left |\psi_t*\widetilde{\phi} _ {j} (\cdot,x_I)(x)\right|_{\mathfrak{H}} \leq C 2^{-\alpha |k-j|\beta} \frac {2^{-\alpha  (j \wedge k)\gamma}}{(2^{-\alpha  (j \wedge k)}+|x-x_I|
 )^{  1+ \gamma}}.
          \end{split}  \end{equation}
          Let $   \mathfrak  T_{ 0 } := \big\{I ; \ell(I)= 2^{-\alpha( j+N)},  \frac { |x-x_{I }| }{2^{-\alpha(j \wedge k)   }}\leq 1    \big\} $, and for $ {l}\in \mathbb{N} $,
          let
   \begin{equation}\label{eq:T-li}  \begin{split}
     \mathfrak T_{ {l} } :=\left\{I ; \ell(I)= 2^{-\alpha( j+N)},2^{\alpha(l -1)}\leq \frac { |x-x_{I }| }{2^{-\alpha(j \wedge k)   }}\leq 2^{\alpha  l }  \right \}.
      \end{split}   \end{equation}Then $\bigcup_{I \in T_{  l} } I$ is contained in an interval $\widetilde{I}$ centered at $x$ with length
     $
      \left | \bigcup_{I \in   \mathfrak  T_{  l} } I\right| \lesssim 2^{ -\alpha(j \wedge k)    } 2^{\alpha  l } .
  $
       Therefore,
           \begin{equation}  \label{eq:sum1}\begin{split} 
  \sum_{\ell(I)= 2^{-\alpha( j+N)}} & \frac {2^{-\alpha (j \wedge k)\gamma}}{(2^{-\alpha (j \wedge k) }+|x-x_I|   )^{  1+ \gamma}}c_\alpha\ell(I) |\widetilde{ \rho}_ {{j}}*  \mathfrak  f(x_I)
  |_{\mathfrak{H}}\\
  \lesssim& \sum_{ {l}=0 }^\infty  2^{-\alpha  l (1+ \gamma) }  2^{ \alpha (j \wedge k -j)}\sum_{I\in    \mathfrak  T_{  l} } |\widetilde{ \rho}_ j*  \mathfrak  f(x_I)  |_{\mathfrak{H}} \\
        \lesssim   &\sum_{ {l}=0 }^\infty 2^{-\alpha  l (1+ \gamma) }  2^{ \alpha (j \wedge k -j)}\left(\sum_{I\in    \mathfrak  T_{  l} } |\widetilde{ \rho}_ j*  \mathfrak  f(x_I)  
        |^r_{\mathfrak{H}}\right)^{\frac
        1r}    \\
       = &\sum_{ {l}=0 }^\infty 2^{-\alpha  l (1+ \gamma) }  2^{ \alpha (j \wedge k -j)}\left(\frac 1{\ell(I)}\int_{\mathbb{R}}\sum_{I\in    \mathfrak  T_{  l} } |\widetilde{ \rho}_ j*  \mathfrak  f(x_I)
       |^r_{\mathfrak{H}}\chi_I \right\}^{\frac 1r} \\
       \lesssim  &\sum_{ {l}=0 }^\infty 2^{-\alpha  l (1+ \gamma - \frac {1}r) }  2^{ \alpha (j \wedge k -j)\left(1  - \frac {1}r\right)}
        \left\{M  \left(\sum_{\ell(I)= 2^{-\alpha( j+N)}}
       |\widetilde{ \rho}_ j*  \mathfrak  f(x_I)  |^r_{\mathfrak{H}}\chi_I\right) (x^* ) \right\}^{\frac 1r} \end{split}  \end{equation}
  for any $|x^*-x|\leq 2C_02^{-\alpha( k +N) }$ and $0<r<1$. In the last inequality,  we replace the integral by { the}  average over the interval $C_0\widetilde{I}$ and then by the maximal function. 
  The implicit constants above  { depend} on  $\alpha, N,C_0$.

 Note that the summation in \eqref{eq:sum1} over $ {l} $ converges if we choose  $1>r>  \frac 1 { 1+ \gamma}    $. If substituting \eqref{eq:almost-orthogonality} and \eqref{eq:sum1} into \eqref {eq:psi*f},
 we get for given   dyadic interval $J$
with $\ell(J)=2^{-\alpha( k +N) }$,
   \begin{equation*}  \begin{split}\sup_{v\in C_0J} |\psi_t* \mathfrak   f(v) |_{\mathfrak{H}}\chi_J(u) & \lesssim    \sum_{ j }   2^{-\alpha |k-j|\beta}2^{ \alpha (j \wedge k
   -j)\left(1
   - \frac {1}r\right)}
   \left\{M  \left(\sum_{\ell(I)= 2^{-\alpha( j+N)}}   |\widetilde{ \rho}_ j*  \mathfrak  f(x_I)  |^r_{\mathfrak{H}}\chi_I\right) (u )\right\}^{\frac 1r}\chi_J(u) \end{split}  \end{equation*}
   for any $u \in J$.
 { Using} H\"older inequality as in the proof of { the} discrete Young inequality $\|a*b\|_{l^2(\mathbb{Z})}\leq \|a \|_{l^1(\mathbb{Z})} \| b\|_{l^2(\mathbb{Z})}$, and { the fact 
 that} 
 \begin{equation*}
    \sum_k 2^{- \alpha |k-j|\beta}2^{ \alpha (j \wedge k -j)\left(1  - \frac {1}r\right)} <C<\infty
 \end{equation*}
 for some $C>0$ independent of $j$, we { obtain} 
   \begin{equation*}
    \begin{split}\sum_{ k }\sum_{J}\int_{2^{-\alpha k}}^{2^{-\alpha (k -1)  }}   \sup_{v\in C_0J} |\psi_t* \mathfrak   f(v) |^2_{\mathfrak{H}}\chi_J (u) \frac { d {t}}{ {t}}  &
    \lesssim \sum_{ j }
   \left\{M  \left(\sum_{ I }\inf_{x\in I}  |\widetilde{ \rho}_ j*  \mathfrak   f(x)  |^r_{\mathfrak{H}}\chi_I\right) (u)\right\}^{\frac 2r}  ,
   \end{split}  \end{equation*} where the summations are taken over   dyadic interval $J$  and $I$
with $\ell(J)=2^{-\alpha( k +N) }$ and $\ell(I)=2^{-\alpha( j +N) }$, respectively.
Then
   \begin{equation*}
    \begin{split}\left\|\left(\sum_{ k }\sum_{j\in \mathbb{Z}}\int_{2^{-\alpha k}}^{2^{-\alpha (k -1)  }}  \sup_{v\in C_0J} |\psi_t* \mathfrak   f(v) |^2_{\mathfrak{H}}\chi_J  \frac { d {t}}{ {t}}
    \right)^{\frac 12}\right\|_{{ L^1}}  &  \lesssim \left\|\left(\sum_{ j }
   \left\{M  \left(\sum_{I}\inf_{x\in I}   |\widetilde{ \rho}_ j*  \mathfrak  f(x)  |^r_{\mathfrak{H}}\chi_I\right) \right\}^{\frac 2r} \right)^{\frac r 2}
   \right\|_{{{ L^\frac 1r}}  }^\frac 1r\\
   &  \lesssim \left\|\left(\sum_{ j }
   \sum_{I}\inf_{x\in I}   |\widetilde{ \rho}_ j*  \mathfrak  f(x)  |^2_{\mathfrak{H}}\chi_I   \right)^{\frac r 2}  \right\|_{{{ L^\frac 1r}}  }^\frac 1r\\
   & =\left\| \left(\sum_{ j }
   \sum_{I}\inf_{x\in I}   |\widetilde{ \rho}_ j*  \mathfrak  f(x)  |^2_{\mathfrak{H}}\chi_I      \right)^{\frac 12} \right\|_{{ L^1}} 
   \end{split}  \end{equation*}
      by the Fefferman-Stein vector-valued maximal function inequality \cite{Gra}
for $1/r> 1$. We { derive the required} conclusion by
 \begin{equation*}
    \begin{split}\inf_{x\in I}   |\widetilde{ \rho}_ j*  \mathfrak  f(x)  |^2_{\mathfrak{H}}&=\frac 1{c_\alpha^2}\inf_{x\in I}  \left | \int_{2^{-\alpha j} }^{2^{-\alpha (j-1) }} \rho_ { t } *  \mathfrak  f(x) 
    \frac { d
    {t}}{
    {t}}\right |^2_{\mathfrak{H}} \leq\frac 1{c_\alpha^2}\inf_{x\in I}\left\{\int_{2^{-\alpha j} }^{2^{-\alpha (j-1) }}  \left | \rho_ { t } *  \mathfrak  f(x)\right |_{\mathfrak{H}} \frac { d {t}}{
    {t}}\right\}^2\\
    &\leq  \frac 1{c_\alpha}\inf_{x\in I} \int_{2^{-\alpha j} }^{2^{-\alpha (j-1) }} \left | \rho_ { t } *  \mathfrak  f(x)\right |^2_{\mathfrak{H}} \frac { d {t}}{ {t}},
   \end{split}  \end{equation*}
 using the Minkowski  inequality and  the Cauchy-Schwarz  inequality.
\end{proof}

\end{document}